\documentclass{birkjour}
\usepackage[utf8x]{inputenc}

\usepackage{hyperref}
\usepackage{dsfont}
\usepackage{amssymb}
\usepackage{esint}
\usepackage{mathptmx,times}
\usepackage[all]{xy}
\usepackage{mathtools}
\usepackage{enumitem}
\usepackage{color}

\setitemize[0]{leftmargin=15pt}

 \newtheorem{thm}{Theorem}[section]
 \newtheorem{cor}[thm]{Corollary}
 \newtheorem{lem}[thm]{Lemma}
 \newtheorem{prop}[thm]{Proposition}
 \theoremstyle{definition}
 \newtheorem{defn}[thm]{Definition}
 \theoremstyle{remark}
 \newtheorem{rem}[thm]{Remark}
 
 \numberwithin{equation}{section}

\newcommand{\ti}{\breve{\i}}
\newcommand{\vds}{\breve{p}}

\newcommand{\N}{{\mathbb N}}

\newcommand{\R}{{\mathbb R}}
\newcommand{\Z}{{\mathbb Z}}

\newcommand{\LL}{\mathcal{L}}
\newcommand{\PP}{\mathcal{P}}
\newcommand{\ZZ}{\mathcal{Z}}

\newcommand{\id}{\operatorname{id}}
\newcommand{\im}{{\operatorname{im}}}
\newcommand{\ev}{{\operatorname{ev}}}
\newcommand{\cov}{{\operatorname{cov}}}
\newcommand{\curv}{{\operatorname{curv}}}
\newcommand{\Ad}{\operatorname{Ad}}
\newcommand{\Hom}{\operatorname{Hom}}
\newcommand{\pr}{\operatorname{pr}}
\newcommand{\Ext}{\operatorname{Ext}}
\newcommand{\Tor}{\operatorname{Tor}}

\newcommand{\dom}{\operatorname{dom}}

\newcommand{\Hdr}{H_{dR}}
\newcommand{\g}{\mathfrak{g}}

\newcommand{\Ul}{{\mathrm{U}(1)}}       

\newcommand{\OO}{\mathrm{O}}
\newcommand{\Spin}{\mathrm{Spin}}
\newcommand{\String}{\mathrm{String}}

\newcommand{\cl}{\mathrm{cl}}

\renewcommand{\phi}{\varphi}

\begin{document}
%%%%%%%%%%%%%%%%%%%%%%%%%%%%%%%%%%%%%%%%%%%%%%%%%%%%%%%%%%%%%%%%%%%%%%%%%

\title[Cheeger-Chern-Simons theory]{Cheeger-Chern-Simons theory and differential String classes}
\author[C.~Becker]{Christian Becker}
\address{%
Universit\"at Potsdam\\ 
Institut f\"ur Mathematik\\ 
Karl-Liebknecht-Str.~24--25\\ 
14476 Potsdam\\ 
Germany
}
\email{becker@math.uni-potsdam.de}

\subjclass{Primary 53C08, Secondary 55N20}

\keywords{Cheeger-Simons character, differential character, Chern-Simons theo\-ry, differential cohomology, transgression, geometric String structure, String geometry, Wess-Zumino model, Dijkgraaf-Witten correspondence, Hopf theorem}

\date{\today}

\begin{abstract}
We construct new concrete examples of relative differential characters, which we call {\em Cheeger-Chern-Simons characters}.
They combine the well-known Cheeger-Simons characters with Chern-Simons forms.
In the same way as Cheeger-Simons characters generalize Chern-Simons invariants of oriented closed manifolds, Cheeger-Chern-Simons characters generalize Chern-Simons invariants of oriented manifolds with boundary.

We study the differential cohomology of compact Lie groups $G$ and their classifying spaces $BG$.
We show that the even degree differential cohomology of $BG$ canonically splits into Cheeger-Simons characters and topologically trivial characters.
We discuss the transgression in principal $G$-bundles and in the universal bundle. 
We introduce two methods to lift the universal transgression to a differential cohomology valued map. 
They generalize the Dijkgraaf-Witten correspondence between $3$-dimensional Chern-Simons theories and Wess-Zumino-Witten terms to fully extended higher order Chern-Simons theories. 
Using these lifts, we also prove two versions of a differential Hopf theorem.

Using Cheeger-Chern-Simons characters and transgression, we introduce the notion of differential trivializations of universal characteristic classes.
It gene\-ralizes well-established notions of differential String classes to arbitrary degree. 
Specializing to the class $\frac{1}{2} p_1 \in H^4(B\Spin_n;\Z)$ we recover isomorphism classes of geometric String structures on $\Spin_n$-bundles with connection and the corresponding Spin structures on the free loop space.  

The Cheeger-Chern-Simons character associated with the class $\frac{1}{2} p_1$ together with its transgressions to loop space and higher mapping spaces defines a Chern-Simons theory, extended down to points.
Differential String classes provide trivializations of this extended Chern-Simons theory.
This setting immediately generalizes to arbitrary degree: for any universal characteristic class of principal $G$-bundles, we have an associated Cheeger-Chern-Simons character and extended Chern-Simons theory.
Differential trivialization classes yield trivializations of this extended Chern-Simons theory.
\end{abstract}

\maketitle
\tableofcontents

%%%%%%%%%%%%%%%%%%%%%%%%%%%%%%%%%%%%%%%%%%%%%%%%%%%%%%%%%%%%%%%%%%%%%%%%%
\section{Introduction}
%%%%%%%%%%%%%%%%%%%%%%%%%%%%%%%%%%%%%%%%%%%%%%%%%%%%%%%%%%%%%%%%%%%%%%%%%

The present article contributes to the program of \emph{String geometry} and its higher order generalizations.
By String geometry we understand the study of geometric structures on a smooth manifold $X$ that correspond to so-called Spin structures on the free loop space $\LL X$.

As the name suggests, the notion of String geometry originates in mathematical physics.
String structures on a manifold $X$ were introduced by Killingback \cite{K87} in the context of anomaly cancellation for the world sheet action of strings.
Like Spin structures account for supersymmetric point particles being well-defined, String structures are supposed to play an analagous role for supersymmetric strings. 

Nonlinear sigma models, i.e.~$2$-dimensional field theories on $X$, may be considered as $1$-dimensional field theories on the free loop space $\LL X$.
As in the case of the manifold $X$ itself, one needs a notion of Spin structures on the loop space $\LL X$ to formulate the fermionic part of the theory.
Such a notion was introduced in \cite{K87} and further elaborated on in \cite{CP89}.

Actually, consistency of the supersymmetric sigma model requires slightly stronger geometric structures than Spin structures on $\LL X$, so-called geometric String structures on $X$.
In \cite{ST04}, geometric String structures are defined as triviali\-zations of the extended Chern-Simons theory associated with a given principal $\Spin_n$-bundle with connection $\pi:(P,\theta) \to X$.
A geometric notion of String connections that satisfies this definition was established by Waldorf in \cite{W13} in terms of trivializations of bundle 2-gerbes with compatible connections. 
The full picture of the correspondence between Spin structures on $\LL X$ and String structures on $X$ was established recently in \cite{W14} as an equivalence of categories between (geometric) String structures on $X$ and (superficial geometric) fusion Spin structures on $\LL X$.
Geometric String structures are classified up to isomorphism by certain degree $3$ differential cohomology classes, called differential String classes.

In the present paper, we generalize the notion of (isomorphism classes of) differential String structures to arbitrary universal characteristic classes for principal $G$-bundles.
Instead of bundle (2-)gerbes we use differential characters to represent differential cohomology classes.
While the use of bundle (2-)gerbes is limited to low degrees in cohomology, our approach applies to universal characteristic classes of arbitrary degree.
We introduce a notion of differential trivializations of characteristic classes and discuss trivializations of the corresponding extended higher order Chern-Simons theories.

Now let us describe the above mentioned geometric structures in more detail.
Given a principal $\Spin_n$-bundle $\pi:P \to X$, the loop space functor yields a principal $\LL \Spin_n$-bundle $\pi:\LL P \to \LL X$.
To construct associated vector bundles, one needs nice representations of the loop group.
However, all positive energy representations of the loop group $\LL \Spin_n$ are projective \cite{PS86}.
Thus one aims at lifting the structure group of the loop bundle from $\LL \Spin_n$ to its universal central extension $\widehat{\LL \Spin_n}$.
The obstruction to such lifts is a certain cohomology class in $H^3(\LL X;\Z)$, most easily described as the transgression to loop space of the class $\frac{1}{2} p_1(P) \in H^4(X;\Z)$.
Actual bundle lifts are referred to as \emph{Spin structures} on the loop bundle.

Instead of lifts of loop bundles one may also consider trivializations of the class $\frac{1}{2}p_1(P)$ on the manifold $X$ itself.
This is the program of \emph{String geometry}.
On the one hand, $\frac{1}{2}p_1(P)$ is the obstruction to lift the structure group of $\pi:P \to X$ from $\Spin_n$ to its $3$-connected cover $\String_n$.
As is well known, $\String_n$ cannot be realized as a finite dimensional Lie group, but as an infinite dimensional Fr\'echet Lie group \cite{NSW13}.
From this perspective, String geometry may be regarded as the study of principal $\String_n$-bundles (with connection) lifting the given principal $\Spin_n$-bundle (with connection).
Actual such lifts may be called (geometric) String structures in the Lie theo\-retic sense.
On the other hand, $\frac{1}{2}p_1(P)$ is the characteristic class of a higher cate\-gorical geometric structure, the so-called \emph{Chern-Simons bundle $2$-gerbe} \cite{CJMSW05,W13}.
From this perspective, String geometry is the study of trivializations of the Chern-Simons bundle $2$-gerbe toge\-ther with compatible connections.
Actual such triviali\-zations (with connection) are called (geometric) String structures in the gerbe theoretic sense.

Isomorphism classes of $\Spin_n$-bundles $\pi:P\to X$ are in 1-1 correspondence with homotopy classes of maps $f:X \to B\Spin_n$ to the classifying space.
Likewise, isomorphism classes of $\String_n$-lifts $\pi:P\to X$ are in 1-1 correspondence with homotopy classes of lifts 
\begin{equation*}
\xymatrix{
&& B\String_n \ar[d] \\
X \ar@{.>}[urr]^{\widetilde f} \ar[rr]_f && B\Spin_n .
}
\end{equation*}
By work of Redden \cite{R11}, the latter are also in 1-1 correspondence with certain cohomology classes in $H^3(P;\Z)$, called \emph{String classes}.
The set of all String classes on $\pi:P \to X$ is a torsor for the cohomology $H^3(X;\Z)$ of the base. 
String classes are recovered by String structures, in both senses mentioned above.
Associated with a String class on a compact riemannian manifold $(X,g)$ is a canonical $3$-form $\varrho \in \Omega^3(X)$ which trivializes the class $\frac{1}{2}p_1$, see \cite{R11}.
Transgression to loop space maps a String class to the Chern class of a line bundle over $\LL P$.
The total space of this line bundle is the total space of the $\widehat{\LL \Spin_n}$-lift of the loop bundle $\pi:\LL P \to \LL X$.

String connections as described above refine String classes to differential cohomology classes.
More precisely, isomorphism classes of geometric String structures are in 1-1 correspondence with differential String classes \cite{W14}.
The set of all differential String classes is a torsor for the differential cohomology $\widehat H^3(X;\Z)$ of the base \cite{W13}.

String classes are a special case of the more general concept of trivialization classes for universal characteristic classes $u \in H^*(BG;\Z)$ for principal $G$-bundles \cite{R11}.
For a principal $G$-bundle with connection $(P,\theta) \to X$, we expect to have canonical refinements of trivialization classes to differential cohomology classes.
In analogy with differential String classes, an appropriate notion of differential trivialization classes should satisfy the following two properties:
firstly, restriction to any fiber should yield a canonical class in the differential cohomology of the Lie group $G$; secondly, the set of all differential differential trivialization classes should be a torsor for the action of the differential cohomology $\widehat H^*(X;\Z)$ of the base.

In the present article, we establish a notion of differential trivialization classes for universal characteristic classes of principal $G$-bundles (with connection) and show that it satisfies these two properties.
This notion is not as obvious as it may seem: taking the set of all differential cohomology classes with characteristic class a trivialization class is far too large, even after imposing the condition to restrict to a fixed differential cohomology class along the fibers.
Associated with a differential trivialization class in our sense is a differential form on the base $\varrho \in \Omega^{*}(X)$ which trivializes the given characteristic class $u(P) \in H^*(X;\Z)$.
Specializing to the characteristic class $\frac{1}{2}p_1 \in H^4(B\Spin_n;\Z)$ we recover differential String classes in the sense of \cite{W14} and their canonical $3$-forms on $X$.

On a compact Riemannian manifold $(X,g)$, we find another equivalent charac\-terization of geometric String structures:
By Hodge theory and adiabatic limits \cite{R11}, one obtains a canonical $3$-form $\varrho \in \Omega^3(X)$ such that $CS_\theta(\frac{1}{2}p_1) - \pi^*\varrho \in \Omega^3(P)$ represents a given String class.
We show that this setting uniquely determines a dif\-feren\-tial String class with differential form $\varrho$.

The main tool in our concept of differential trivialization classes is the notion of \emph{Cheeger-Chern-Simons characters} $\widehat{CCS}_\theta \in \widehat H^*(\pi;\Z)$.
These are new concrete examples of relative differential characters in the sense of \cite{BT06}.
Differential characters were introduced by Cheeger and Simons in \cite{CS83} as certain $\Ul$-valued characters on the group of smooth singular cycles in $X$. 
The ring of differential characters on a manifold $X$ is nowadays called the \emph{differential cohomology} of $X$.
We denote it by $\widehat H^*(X;\Z)$. 
Out of a differential character $h \in \widehat H^k(X;\Z)$ one obtains a smooth singular cohomology class $c(h) \in H^k(X;\Z)$ -- its \emph{characteristic class} -- and a differential $k$-form $\curv(h) \in \Omega^k(X)$ with integral periods -- its \emph{curvature}. 
Both the characteristic class and the curvature map are well-known to be surjective. 
In this sense, differential characters are refinements of smooth singular cohomology classes by differential forms.

As a particular example, Cheeger and Simons construct certain even degree differential characters on the base $X$ of a principal $G$-bundle with connection $(P,\theta)\to X$ with curvature given by the Chern-Weil forms of the connection $\theta$.
The construction thus lifts the Chern-Weil map to a differential character valued map.
We refer to these particular characters as \emph{Cheeger-Simons characters} and denote them by $\widehat{CW}_\theta$ to emphasize their relation to the Chern-Weil map.
They are also called differential characteristic classes by some authors \cite{B12,H13}.

Relative differential characters were introduced by Brightwell and Turner in \cite{BT06} as differential characters on the mapping cone cycles of a smooth map $\varphi:A \to X$.
Thus it would also be appropriate to call them mapping cone characters.
A relative character $h \in \widehat H^k(\varphi;\Z)$ determines an absolute character $\vds(h) \in \widehat H^k(X;\Z)$, see \cite{BB13}.
Out of a relative character $h \in \widehat H^k(\varphi;\Z)$ one obtains an additional differential form $\cov(h) \in \Omega^{k-1}(A)$ -- its \emph{covariant derivative}.
Relative differential characters in $\widehat H^k(\varphi;\Z)$ may be regarded as sections along the map $\varphi:A \to X$ of the absolute characters $\vds(h) \in \widehat H^k(X;\Z)$.
An absolute character in $\widehat H^k(X;\Z)$ has sections along a smooth map $\varphi$ if and only if its characteristic class vanishes upon pull-back by $\varphi$; see \cite{BB13} for further details.

Bundle gerbes on $X$ provide a particular class of examples of relative differential characters: any bundle gerbe $\mathcal G$ with connection, defined by a submersion $\pi:Y \to X$, determines a relative differential character $h_{\mathcal G} \in \widehat H^3(\pi;\Z)$ with covariant derivative the curving $B \in \Omega^2(Y)$ of the bundle gerbe, see \cite{BB13}.
The absolute character $\vds(h_{\mathcal G}) \in \widehat H^3(X;\Z)$ corresponds to the stable isomorphism class of the bundle gerbe.
Analogously, bundle 2-gerbes with connection, defined by a submersion $\pi:Y \to X$, determine relative characters in $\widehat H^4(\pi;\Z)$ with curvature given by the curvature $4$-form of the bundle 2-gerbe and covariant derivative given by the curving $B \in \Omega^3(Y)$.

The Cheeger-Simons construction mentioned above is refined in the present article to a relative differential character valued map, which factorizes through the Cheeger-Simons map.
We call the resulting relative characters \emph{Cheeger-Chern-Simons characters} and denote them by $\widehat{CCS}_\theta$. 
Two elementary observations lead to the construction of these Cheeger-Chern-Simons characters:
First of all, universal characteristic classes of principal $G$-bundles vanish upon pull-back to the total space.
Thus Cheeger-Simons characters admit sections along the bundle projection $\pi:P \to X$.
Secondly, the pull-back of a Chern-Weil form $CW_\theta$ to the total space has a canonical trivialization: the associated Chern-Simons form $CS_\theta$.
Putting these observations together, we obtain our notion of Cheeger-Chern-Simons characters:
Let $G$ be a Lie group with finitely many components and $(P,\theta) \to X$ a principal $G$-bundle with connection.
Associated with an invariant homogeneous polynomial $\lambda$ of degree $k$ on the Lie algebra $\mathfrak g$ and a corresponding universal characteristic class $u \in H^{2k}(BG;\Z)$ is a natural relative differential character $\widehat{CCS}_\theta(\lambda,u) \in \widehat H^{2k}(\pi;\Z)$.
It is uniquely determined by the requirement that it projects to the Cheeger-Simons character $\widehat{CW}_\theta(\lambda,u)$ and that its covariant derivative is given by the Chern-Simons form $CS_\theta(\lambda)$.
The construction of the Cheeger-Chern-Simons character $\widehat{CCS}_\theta(\lambda,u)$ relies on the same arguments as the construction of the Cheeger-Simons character $\widehat{CW}_\theta$ in \cite{CS83}.
We also discuss multiplicativity and dependende upon the connection.

In the same way as Cheeger-Simons characters $\widehat{CW}_\theta$ generalize Chern-Simons invariants of oriented closed manifolds, Cheeger-Chern-Simons characters $\widehat{CCS}_\theta$ generalize Chern-Simons invariants of oriented manifolds with boundary.
Specializing to the universal characteristic class $u=\frac{1}{2}p_1 \in H^4(B\Spin_n;\Z)$, the Cheeger-Chern-Simons character $\widehat{CCS}_\theta(\frac{1}{2}p_1)$ coincides with the relative differential cohomology class $h_\mathcal{C\!S} \in \widehat H^4(\pi;\Z)$ represented by the so-called Chern-Simons bundle $2$-gerbe $\mathcal{C\!S}$ on $X$, constructed in \cite{CJMSW05}. 

We use Cheeger-Chern-Simons characters for two kinds of applications: to introduce our notion of differential trivializations of universal characteristic classes of principal $G$-bundles and to define trivializations of higher order Chern-Simons theories.
For both these applications we need two kinds of transgression:
transgression to loop space and higher mapping spaces by means of fiber integration, and transgression in fiber bundles from the base space to the fiber.

Transgression of differential cohomology classes to loop space was constructed by various authors in several models of differential cohomology \cite{B93,BKS10,DL05,GT00,HLZ03,HS05,L06}.
In \cite{BB13} we study fiber integration and transgression of absolute and relative differential cohomology systematically.
We prove that fiber integration is uniquely determined by naturality and compatibility with differential forms and construct particularly nice and geometric fiber integration and transgression maps.  

Transgression in the cohomology of fiber bundles from the base to the fiber is a classical tool, mainly used to study the algebraic topology of principal bundles \cite{B53}.
To the best of our knowledge, there are no known generalizations to differential cohomology.
We discuss two different ways to construct the cohomology transgression by appropriate diagram chase arguments: one using the inclusion of the fiber in the total space, the other one using the mapping cone cohomology of the bundle projection.
We show that only the latter nicely generalizes to differential cohomo\-logy.

For an arbitrary fiber bundle, the transgression map is defined as usual on the transgressive elements in the differential cohomology of the base and takes values in a quotient of the differential cohomology of the fiber.
On a universal principal $G$-bundle, all differential cohomology classes of the base are transgressive.
We establish two different methods to lift the transgression to a map $\widehat T:\widehat H^*(BG;\Z) \to \widehat H^*(G;\Z)$ which takes values in differential cohomology of the fibers, not just a quotient.
We use these lifts to prove two versions of a differential Hopf theorem: the differential cohomology of a compact Lie group $G$ is generated, up to certain topologically trivial characters, by trangressed Cheeger-Simons characters on $BG$.

The cohomology transgression $T:H^4(BG;\Z) \to H^3(G;\Z)$ for compact Lie groups has a well-known interpretation in mathematical physics terms \cite{DW90}:
elements of $H^4(BG;\Z)$ classify topological charges of $3$-dimensional Chern-Simons theories, while elements of $H^3(G;\Z)$ are topological charges of two-dimensional sigma models on $G$, i.e.~Wess-Zumino models.
By work of Dijkgraaf and Witten, the correspondence between $3$-dimensional Chern-Simons theories and Wess-Zumino models on $G$ is determined by the transgression, applied to the topological charges.
Our differential cohomology transgression $\widehat T: \widehat H^*(BG;\Z) \to \widehat H^*(G;\Z)$ generalizes the Dijkgraaf-Witten correspondence to fully extended higher order Chern-Simons theories.
These higher order Chern-Simons theories are constructed from higher degree Cheeger-Simons and Cheeger-Chern-Simons characters by means of transgression to loop space and higher mapping spaces.
Trivializations of extended higher order Chern-Simons theories are constructed in terms of transgression of differential String classes (in the case of $3$-dimensional Chern-Simons theories) and universal differential trivializations (in the general case).

The article is organized as follows:
Section~\ref{sec:hat_CCS} reviews the well-known Chern-Weil, Chern-Simons and Cheeger-Simons constructions and constructs \emph{Cheeger-Chern-Simons characters}.
Section~\ref{sec:transgression} discusses several notions of \emph{transgression} for Cheeger-Simons characters: transgression to loop space and transgression in fiber bundles.
It further establishes basic results on the differential cohomology of compact groups and their classifying spaces.
Section~\ref{sec:diff_triv} introduces \emph{differential trivializations} of universal charac\-teristic classes of principal $G$-bundles.
Section~\ref{sec:diff_string} specializes to the case of \emph{differential String classes} and discusses trivializations of fully extended higher order Chern-Simons theories.  
The appendices provide background information on differential characters, bundle $2$-gerbes and transgression.

\pagebreak[4]

\noindent
\emph{Acknowledgement.}\\
\noindent
It is a great pleasure to thank Corbett Reddden and Konrad Waldorf for very helpful discussion.
Part of the paper was written during a stay at the \emph{Erwin Schr\"odinger International Institute for Mathematical Physics (ESI)} in Vienna.
The author thanks \emph{ESI} for hospitality.
Financial support is gratefully acknowledged to the \emph{DFG Sonderforschungsbereich ``Raum Zeit Materie''}, to the \emph{DFG Network ``String Geometry''}, and to \emph{ESI}. 

%%%%%%%%%%%%%%%%%%%%%%%%%%%%%%%%%%%%%%%%%%%%%%%%%%%%%%%%%%%%%%%%%%%%%%%%%
\section{Cheeger-Chern-Simons theory}\label{sec:hat_CCS}
%%%%%%%%%%%%%%%%%%%%%%%%%%%%%%%%%%%%%%%%%%%%%%%%%%%%%%%%%%%%%%%%%%%%%%%%%

Throughout this section let $G$ be a Lie group with finitely many components.
Let $\mathfrak g$ be its Lie algebra.
Let $(P,\theta) \to X$ be a principal $G$-bundle with connection.
More explicitly, $P$ and $X$ are smooth manifolds (not necessarily finite dimensional\footnote{In an infinite dimensional setting, one may take $P$ and $X$ to be either Banach or Fr\'echet manifolds; see \cite{H82} for an overview of Fr\'echet manifolds and also \cite{KM97} for the de Rham complex on several categories of infinite dimensional manifolds, including Fr\'echet manifolds.}) together with a smooth right $G$-action on $P$, a diffeomorphism $P/G \xrightarrow{\approx} X$ and a connection $1$-form $\theta \in \Omega^1(P;\mathfrak g)$.
Let $\lambda$ be an invariant polynomial on $\mathfrak g$, homogeneous of degree $k$.
By the classical \emph{Chern-Weil construction}, the polynomial $\lambda$ associates with $(P,\theta)$ the Chern-Weil form $CW_\theta(\lambda) \in \Omega^{2k}(X)$.
In fact, the Chern-Weil form $CW_\theta(\lambda)$ is closed and its de Rham cohomology class $[CW_\theta(\lambda)] \in H^{2k}_{dR}(X)$ does not depend upon the connection $\theta$.

The Chern-Weil construction has two well-known refinements, the Chern-Simons-construction \cite{CS74} and the Cheeger-Simons construction \cite{CS83}:
The pull-back of the Chern-Weil form $CW_\theta(\lambda)$ along the bundle projection $\pi:P \to X$ is an exact form.
The Chern-Simons form $CS_\theta(\lambda) \in \Omega^{2k-1}(P)$, constructed in \cite{CS74}, satisfies $dCS_\theta(\lambda) = \pi^*CW_\theta(\lambda)$.
Moreover, the Chern-Weil construction has a unique lift to the differential cohomology $\widehat H^{2k}(X;\Z)$, constructed in \cite{CS83}.
In other words, the Chern-Weil form $CW_\theta(\lambda)$ is the curvature of a differential character $\widehat{CW}_\theta(\lambda) \in \widehat H^{2k}(X;\Z)$.   

In this section, we further refine these well-known constructions: we show that there is a canonical natural relative differential character for the bundle projection $\pi:P \to X$ with covariant derivative $CS_\theta(\lambda)$, which maps to $\widehat{CW}_\theta(\lambda)$ under the map $\vds: \widehat H^{2k}(\pi;\Z) \to \widehat H^{2k}(X;\Z)$.
Since this construction combines the Chern-Simons form $CS_\theta(\lambda)$ with the Cheeger-Simons differential character $\widehat{CW}_\theta(\lambda)$, we call it the \emph{Cheeger-Chern-Simons} construction.

For convenience of the reader, we review the classical Chern-Weil and Chern-Simons constructions.
The basic notions of relative and absolute differential characters are reviewed in Appendix~\ref{app:characters}.

\subsection{The Chern-Weil, Chern-Simons and Cheeger-Simons constructions}\label{sec:CW_CS_CCS}
In this section, we briefly review the results of the classical Chern-Weil and Chern-Simons constructions.
We review in more detail the refinement of the classical Chern-Weil map to a differential character valued map, constructed by Cheeger and Simons in \cite{CS83}.
We term this the \emph{Cheeger-Simons construction}.

\subsubsection{Universal bundles and connections}\label{sec:universal}
Let $G$ be a Lie group with finitely many components.
Let $\mathfrak g$ be its Lie algebra.
We denote by $\pi_{EG}:EG \to BG$ the universal principal $G$-bundle over the classifying space of $G$.
Let $x \in BG$ and denote by $EG_x:= \pi_{EG}^{-1}(x)$ the fiber of $EG$ over $x$.

Any principal $G$-bundle $\pi:P \to X$ can be written as pull-back of the universal bundle via a pull-back diagram
\begin{equation*}
\xymatrix{
P = f^*EG \ar^(0.55){F}[rr] \ar_{\pi}[d] && EG \ar^{\pi_{EG}}[d] \\
X \ar_(0.55){f}[rr] && BG \,. \\
} 
\end{equation*}
A map $f:X \to BG$ such that $f^*EG = P$ as principal $G$-bundles is called a classifying map for $P \to X$.

The pull-back bundle $\pi_{EG}^*EG \to EG$ is trivial as a principal $G$-bundle.
A trivilization is given by the tautological section that maps any point $p \in EG$ to itself, considered as a point in the fiber $(\pi_{EG}^*EG)_{p}=EG_{\pi(p)}$ over $\pi(p) \in BG$.

A universal characteristic class for principal $G$-bundles is a cohomology class in $H^*(BG;\Z)$.
Since the total space $EG$ of the universal principal $G$-bundle is contractible, we have $H^*(EG;\Z) =\{0\}$.
Thus any universal characteristic class $u \in H^*(BG;\Z)$ satisfies $\pi^*u =0$.

Connections on a principal $G$-bundle $P \to M$ can also be induced via pull-back from certain universal data, called \emph{universal connections}:
Narasimhan and Ramanan \cite{NR61} construct universal connections on $n$-classifying principal $G$-bundles.\footnote{A principal $G$-bundle $EG_n \to BG_n$ is called $n$-classifying, if any principal $G$-bundle on any manifold $M$ of dimension $\leq n$ is a pull-back of $EG_n \to BG_n$.}
A universal connection for principal $G$-bundles in the sense of \cite{NR61} is a family of universal connections on $n$-classifying $G$-bundles $EG_n \to BG_n$ (or a direct limit of those).

While $n$-classifying bundles can be realized as finite dimensional smooth manifolds, the universal principal $G$-bundle cannot. 
As a topological bundle, $\pi_{EG}:EG \to BG$ is given by the Milnor construction \cite{M56}.
It can also be given a smooth structure either as a Banach manifold \cite{S80} or as a differentiable space \cite{M79}.\footnote{We expect that $\pi_{EG}:EG \to BG$ can also be realized as a Fr\'echet manifold and that one can construct universal connections as connection $1$-forms in the Fr\'echet manifold sense. 
But to the best of our knowledge, no such statement exists in the literature so far.}
In both cases, universal connections are realized as connection $1$-forms on $EG$ in the respective sense.

Now for any principal $G$-bundle $\pi:(P,\theta) \to X$ with connection, there exists a smooth map $f:M \to BG$ such that $(P,\theta) \to X$ and $f^*(EG,\Theta) \to X$ are isomorphic as principal $G$-bundles with connection over $X$.
See also \cite{K82,R82} for further aspects of universal connections.
Another construction of universal connections on universal principal $G$-bundles has appeared more recently in \cite{BHS12}.

A map $f:X \to BG$ such that $(P,\theta)=f^*(EG,\Theta)$ is called a classifying map for $(P,\theta) \to X$ as principal $G$-bundle with connection.
Note that neither universal connections nor classifying maps for a bundle with connection are unique.

A completely different point of view on classifying spaces, universal bundles and universal connections is taken in \cite{FH13}.
By means homotopical algebra, Freed and Hopkins construct a universal principal $G$-bundle, denoted $E_\nabla G \to B_\nabla G$, where both the total space and base are simplicial sheaves (instead of infinite dimensional manifolds).
In this setting, there is a \emph{canonical} universal connection for principal $G$-bundles, which is a $\g$-valued $1$-form on $E_\nabla G$.
Moreover, this universal connection induces \emph{unique} classifying maps for principal $G$-bundles with connection.

Although it would be interesting to study the Cheeger-Simons and Cheeger-Chern-Simons characters in terms of this new notion of universal connection, we will not pursue this approach in the present paper.
For a fairly general exposition of generalized differential cohomology in terms of homotopical algebra, we refer to \cite{B12}.

\subsubsection{The Chern-Weil construction}\label{sec:Chern-Weil}
Let $G$ be a Lie group with finitely many components.
Let $\mathfrak g$ be its Lie algebra.
Following the notation of \cite{CS83}, we set
$$
I^k(G) := \Big\{ \lambda: \underbrace{\g \otimes \ldots \otimes \g}_{k} \to \R \,\Big|\, \mbox{$\lambda$ symmetric, multilinear, $\Ad_G$-invariant} \Big\}
$$ 
for the space of $\Ad_G$-invariant symmetric multilinear real valued functions from the $k$-fold tensor product of $\g$.
Such functions are called invariant homogeneous polynomials of degree $k$ on $\g$.

Let $\pi:(P,\theta) \to X$ be a principal $G$-bundle with connection.
We denote by $F_\theta \in \Omega^2(X,\Ad(P))$ the curvature $2$-form of the connection $\theta$.
The {\em Chern-Weil map} $CW_\theta:I^k(G) \to \Omega^{2k}(X)$ associates to an invariant polynomial $\lambda \in I^k(G)$ the {\em Chern-Weil form}\footnote{In \cite{CS74} this form is denoted $P(\theta)$ where $P$ denotes an invariant polynomial on $\g$.} 
\begin{equation}\label{eq:def_CW}
CW_\theta(\lambda) 
:= 
\lambda(F_\theta^k)
=
\lambda(\underbrace{F_\theta  \wedge \ldots \wedge F_\theta}_{k}) \in \Omega^{2k}(X) \,. 
\end{equation}
The Chern-Weil form $CW_\theta(\lambda)$ is a closed differential form whose de Rham cohomology class does not depend upon the choice of connection $\theta$.

More precisely, given two connections $\theta_0$, $\theta_1$, the Chern-Weil forms $CW_{\theta_0}$, $CW_{\theta_1}$ differ by the differential of the Chern-Simons form $CS(\theta_0,\theta_1;\lambda)$.
The construction of the Chern-Simons form is reviewed in Section~\ref{sec:Chern-Simons} below.

\subsubsection{The Cheeger-Simons construction}\label{sec:Cheeger-Simons}
As above, let $G$ be a Lie group with finitely many components.
Let $\g$ be its Lie algebra.
Let $\Theta$ be a fixed universal connection on the universal principal $G$-bundle $\pi_{EG}:EG \to BG$.
We denote the corresponding universal Chern-Weil map to the real cohomology of $BG$ by 
$$
CW: I^k(G) \to H^{2k}(BG;\R)\,, \quad \lambda \mapsto [\lambda(F_\Theta^k)]_{dR} \in H^{2k}_{dR}(BG) \cong H^{2k}(BG;\R) \,.
$$
For any principal $G$-bundle with connection $\pi:(P,\theta) \to X$,  we have the commutative diagram:
\begin{equation*}
\xymatrix{
I^k(G) \ar[rr]^{CW} \ar[d]_{CW_\theta} && H^{2k}(BG;\R) \ar[d]^{cl_\R} && H^{2k}(BG;\Z) \ar[ll] \ar[d]^{cl_\Z} \\
\Omega^{2k}_\cl(X) \ar[rr]_{dR} && H^{2k}(X;\R) && H^{2k}(X;\Z) \ar[ll]\,.
}
\end{equation*}
The maps $cl_\R$ and $cl_\Z$ are induced by a classifying map $f:X \to BG$ for the bundle with connection $(P,\theta)$.
They do not depend upon the choice of classifying map.
The map $dR:\Omega^{2k}_\cl(X) \to H^{2k}(X;\R)$ is the projection $\Omega^{2k}_\cl(X) \to \Hdr^{2k}(X)$, followed by the de Rham isomorphism. 
The horizontal maps in the right square are the change of coefficients maps induced by the inclusion $\Z \hookrightarrow \R$.
For an integral cohomology class $u \in H^{2k}(X;\Z)$, we denote by $u_\R$ its image in $H^{2k}(X;\R)$. 

The question arises whether the Chern-Weil form $CW_\theta(\lambda)$ may be represented as curvature of an appropriate differential character (in case it has integral periods).
The question is answered affirmatively in \cite{CS83} by what we term the Cheeger-Simons construction.
Following the notation established there, we put:
\begin{equation*}\label{eq:def:K2k}
K^{2k}(G;\Z) := \big\{ (\lambda,u) \in I^k(G) \times H^{2k}(BG;\Z) \;\big|\; CW(\lambda) = u_\R \big\} 
\end{equation*}
for the set of pairs of invariant polynomials and integral universal characteristic classes that match in real cohomology.
Moreover, denote by
\begin{equation*}
R^n(X;\Z) := \big\{ (\omega,w) \in \Omega^n_0(X) \times H^n(X;\Z) \;\big|\; [\omega]_{dR} = w_\R \big\} 
\end{equation*}
the set of pairs of closed forms with integral periods $\omega$ and smooth singular cohomology classes $w$ that match in real cohomology.

It is shown in \cite{CS83} that the Chern-Weil map $CW_\theta$ has a unique natural lift to a differential character valued map in the following sense: 
For any $k \geq 1$ and any principal $G$-bundle with connection $\pi:(P,\theta) \to X$, there exists a unique natural map $\widehat{CW}_\theta$ such that the diagram
\begin{equation}\label{eq:diag_def_CWhat}
\xymatrix{
&& \widehat H^{2k}(X;\Z) \ar[d]^{(\curv,c)} \\
K^{2k}(G;\Z) \ar@{.>}[urr]^{\widehat{CW}_\theta} \ar[rr]_{CW_\theta \times cl_\Z} && R^{2k}(X;\Z)
}
\end{equation}
commutes.
We call the differential character $\widehat{CW}_\theta(\lambda,u) \in \widehat H^{2k}(X;\Z)$ the \emph{Cheeger-Simons character} associated with $(\lambda,u) \in K^{2k}(G;\Z)$.\footnote{In \cite{CS83} this character is denoted $S_{P,u}(\alpha)$ where $P$ denotes an invariant polynomial, $u$ a universal characteristic class and $\alpha$ a principal $G$-bundle with connection.}
By diagram \eqref{eq:diag_def_CWhat}, the curvature of the Cheeger-Simons character is the Chern-Weil form  
\begin{align}
\curv(\widehat{CW}_\theta(\lambda,u)) 
&= CW_\theta(\lambda)  \label{eq:curv_CW} \\
\intertext{and its characteristic class is given by} 
c(\widehat{CW}_\theta(\lambda,u))
&= c_\Z(u) = f^*u = u(P). \label{eq:c_CW}
\end{align}
Naturality of the map $\widehat{CW}_\theta$ means that for any smooth map $g:X' \to X$ and the pull-back bundle $g^*(P,\theta) \to X'$, we have
\begin{align}
g^*\big(\widehat{CW}_\theta(\lambda,u)\big) = \widehat{CW}_{g^*\theta}(\lambda,u)  \in \widehat H^{2k}(X';\Z).
\end{align}
We call the map $\widehat{CW}_\theta: K^*(G;\Z) \to \widehat H^*(X;\Z)$ the \emph{Cheeger-Simons construction}. 
By \cite[Cor.~2.3]{CS83}, it is a ring homomorphism with respect to the ring structures of $K^*(G;\Z)$ and $\widehat H^*(X;\Z)$.

\subsubsection{The Chern-Simons construction}\label{sec:Chern-Simons}
As before, let $\pi:(P,\theta) \to X$ be a principal $G$-bundle with connection and $\lambda \in I^k(G)$ an invariant polynomial.
The pull-back bundle $\pi^*P \to P$ has a tautological section $\sigma_\mathrm{taut}$, which maps any point $p \in P$ to itself, now considered as a point in the fiber $(\pi^*P)_p = P_{\pi(p)}$.
The tautological section $\sigma_\mathrm{taut}$ yields a trivialization  $P \times G \xrightarrow{\cong} \pi^*P $, $(p,g) \mapsto \sigma_\mathrm{taut}(p) \cdot g$.
In particular, $\pi^*P$ carries a canonical flat connection $\theta_\mathrm{taut}$, obtained from the trivial connection on $P \times G$ by pull-back via the inverse of the trivialization. 
Since the de Rham cohomology classes of Chern-Weil forms do not depend upon the choice of connection, all Chern-Weil classes of $\pi^*P$ vanish.
In particular, the pull-backs $\pi^*CW_\theta(\lambda)=CW_{\pi^*\theta}(\lambda) \in \Omega^{2k}(P)$ of Chern-Weil forms are exact forms on $P$.

The {\em Chern-Simons form} of a connection $\theta$ was first constructed in \cite{CS74} as an invariant $(2k-1)$-form on $P$ whose differential is the pull-back $\pi^*CW_\theta(\lambda)$ of the Chern-Weil form.
In fact, there are two different notions of Chern-Simons forms, closely related to one another: the Chern-Simonms form of two connections $\theta_0,\theta_1$ is a $(2k-1)$-form on the base $CS(\theta_0,\theta_1;\lambda) \in \Omega^{2k-1}(X)$ with differential the difference of the corresponding Chern-Weil forms, while the Chern-Simons form for one connection $\theta$ is a $(2k-1)$-form on the total space $CS_\theta(\lambda) \in \Omega^{2k-1}(P)$ with differential the pull-back of the Chern-Weil form.    

Denote by $\mathcal A(P)$ the space of connections on the principal $G$-bundle $P \to X$.
It is an affine space for the vector space $\Omega^1(X;\Ad(P))$ of $1$-forms on the base $X$ with values in the associated bundle $\Ad(P):= P \times_{\Ad} \g$.
In particular, $\mathcal A(P)$ is path connected.

Let $\theta_0,\theta_1$ be connections on $\pi:P \to X$.
Let $\theta:[0,1] \to \mathcal A(P)$ be a smooth path joining them.
It defines a connection $\theta \in \mathcal A(P \times [0,1])$ on the principal $G$-bundle $P \times [0,1] \to X \times [0,1]$. 
Integrating the Chern-Weil form of this connection over the fiber of the trivial bundle $X \times [0,1] \to X$, we obtain a $(2k-1)$-form $\fint_{[0,1]} CW_{\theta(s)}(\lambda) \in \Omega^{(2k-1)}(X)$.
By the fiberwise Stokes theorem\footnote{The orientation conventions for fiber bundles with boundary are as in \cite[Ch.~4, 7]{BB13}.}, we have:
\begin{equation}\label{eq:CS_two_CW}
CW_{\theta_1}(\lambda) - CW_{\theta_0}(\lambda)
=
-d\fint_{[0,1]} CW_{\theta}(\lambda)
+ \fint_ {[0,1]} \underbrace{dCW_{\theta}(\lambda)}_{=0}.
\end{equation}

Since $\mathcal A(P)$ is an affine space, there is a canonical path joining two connections $\theta_0$ and $\theta_1$, namely the straight line $\theta(t) := (1-t)\theta_0 + t \theta_1$ from $\theta_0$ to $\theta_1$.
The \emph{Chern-Simons form} for two connections is the $(2k-1)$-form on the base, obtained as above, for the straight line:
\begin{equation}\label{eq:def_CS_two}
CS(\theta_0,\theta_1;\lambda) 
:=
-\fint_{[0,1]} CW_{(1-t)\theta_0 + t \theta_1}(\lambda) \in \Omega^{(2k-1)}(X). 
\end{equation}
As above, it satisfies $dCS(\theta_0,\theta_1;\lambda) = CW_{\theta_1}(\lambda) - CW_{\theta_0}(\lambda)$.

The \emph{Chern-Simons form} $CS_\theta(\lambda) \in \Omega^{2k-1}(P)$ for one connection $\theta$ on the bundle $P \to X$ is defined as the Chern-Simons form for the two connections $\theta_\mathrm{taut}$ and $\pi^*\theta$ on the pull-back bundle $\pi^*P \to P$:
\begin{equation}\label{eq:def_CS_one}
CS_\theta(\lambda) 
:=
CS(\theta_\mathrm{taut},\pi^*\theta;\lambda).
\end{equation}
Since the tautological connection is flat, we have $dCS_\theta(\lambda) \stackrel{\eqref{eq:CS_two_CW}}{=} CW_{\pi^*\theta}(\lambda) = \pi^*CW_\theta(\lambda)$.

We call the map $CS_\theta:I^k(G) \to \Omega^{2k-1}(P)$, $\lambda \mapsto CS_\theta(\lambda)$, the \emph{Chern-Simons construction}.\footnote{In \cite{CS74} the invariant polynomials are denoted by $P$. The corresponding Chern-Simons form is denoted $TP(\theta)$ to emphasize its relation to the transgression of the Chern-Weil form.}
Up to an exact remainder, the Chern-Simons form $CS_\theta(\lambda)$ is the unique natural $(2k-1)$-form on $P$ with differential $\pi^*CW_\theta(\lambda)$.
In fact, any two $(2k-1)$-forms on $EG$ with differential $\pi_{EG}^*CW_\Theta(\lambda)$ differ by an exact form, since $EG$ is contractible.
 
In general, the Chern-Simons form $CS_\theta(\lambda)$ is an invariant form on $P$, but it is neither horizontal nor closed.
It depends in a well-known manner upon the connection $\theta$ (see \cite[Prop.~3.8]{CS74} and Section~\ref{subsec:depend_conn_1} below).
The pull-back of the Chern-Simons form $CS_\theta(\lambda)$ to the fiber $P_x \cong G$ over any point $x \in X$ does not depend upon the choice of connection $\theta$.
It can be expressed solely (and explicitly) in terms of the Maurer-Cartan form of $G$.
Moreover, for low dimensional bases $X$, the Chern-Simons form $CS_\theta(\lambda)$ is independent of the connection $\theta$ and is itself a closed form, see \cite[Thm.~3.9]{CS74}.

Let $\Theta$ be a universal connection on the universal principal $G$-bundle $\pi_{EG}:EG \to BG$.
Then we have $dCS_\Theta(\lambda) = \pi_{EG}^*CW_\Theta(\lambda) = \curv({\pi_{EG}}^*\widehat{CW}_\Theta(\lambda,u))$.
From the exact sequence \eqref{eq:short_ex_sequ} and contractibility of $EG$ we conclude $\iota(CS_\Theta(\lambda))={\pi_{EG}}^*\widehat{CW}_\Theta(\lambda,u)$.
Thus the Chern-Simons form provides a topological trivialization of the pull-back of the Cheeger-Simons character $\widehat{CW}_\Theta(\lambda,u)$ to $EG$.

Now let $(P,\theta) \to X$ be a principal $G$-bundle with connection.
Let $f:X \to BG$ be a classifying map for $(P,\theta)$ and denote by $F:(P,\theta) = f^*(EG,\Theta) \to (EG,\Theta)$ the induced map of bundles with connection.
Then we have:
\begin{align}\label{eq:pi_CW_CS}
\iota(CS_\theta(\lambda))
&=
F^*\iota(CS_\Theta(\lambda)) \notag \\
&=
F^*{\pi_{EG}}^*\widehat{CW}_\Theta(\lambda,u) \notag \\
&=
\pi^*f^*\widehat{CW}_\Theta(\lambda,u) \notag \\
&=
\pi^*\widehat{CW}_\theta(\lambda,u) \,.
\end{align}
Here $\iota:\Omega^{2k-1}(EG) \to \widehat H^{2k}(EG;\Z)$ denotes topological trivialization of differential characters, as explained in Appendix~\ref{app:characters}.

\subsubsection{The Chern-Simons action}\label{sec:action}
Let $(P,\theta) \to X$ be a principal $G$-bundle with connection.
Let $f:M \to X$ be a smooth map.  
Suppose that the pull-back bundle $\pi:f^*P \to M$ admits a section $\sigma:M \to f^*P$, hence $\pi \circ \sigma = \id_M$.
Then we obtain a trivialization $M \times G \to f^*P$ by $(x,g) \mapsto \sigma(x) \cdot g$.
Thus the bundle $\pi: f^*P \to M$ can be represented by a constant map $f:M \to BG$ and hence all its characteristic classes vanish.

In particular, any Cheeger-Simons character $\widehat{CW}_{f^*\theta}(\lambda,u)$ is topologically trivial.
In fact, topological trivializations are given by pull-back via $\sigma$ of the corresponding Chern-Simons form $CS_{f^*\theta}(\lambda)$.
Namely, from \eqref{eq:pi_CW_CS} we obtain:
\begin{equation}
\widehat{CW}_{f^*\theta}(\lambda,u)
=
\sigma^*(\pi^*\widehat{CW}_{f^*\theta}(\lambda,u))
=
\iota(\sigma^*CS_{f^*\theta}(\lambda)) \,.
\end{equation}
Since the left hand side is independent of the choice of section $\sigma$, the same holds for the right hand side.

In particular, if $M$ is a closed oriented $(2k-1)$ manifold and the pull-back bundle $f^*P \to M$ admits sections, then we obtain the Chern-Simons invariant of $M$ by evaluating the Cheeger-Simons character on the fundamental class:
$$
\big(\widehat{CW}_{f^*\theta}(\lambda,u)\big)[M]
=
\exp \Big( 2\pi i \int_M \sigma^*CS_{f^*\theta}(\lambda) \Big) \,.
$$
This happens e.g.~if $G$ is simply connected and $M$ is a closed oriented $3$-manifold, for in this case, any principal $G$-bundle $\pi:f^*P \to M$ admits sections.
In this sense, the Cheeger-Simons character $\widehat{CW}_\theta(\lambda,u)$ generalizes the classical Chern-Simons invariants of closed oriented $3$-manifolds. 

In Section~\ref{subsec:CS_action_bound}, we generalize this observation to the Chern-Simons action of oriented manifolds with boundary.

\subsection{The Cheeger-Chern-Simons construction}\label{sec:Cheeger-Chern-Simons}
In this section, we combine the Cheeger-Simons and Chern-Simons constructions to a relative differential character valued map.
This map will be called the Cheeger-Chern-Simons construction.

As above, let $G$ be a Lie group with finitely many components.
Fix a classifying connection $\Theta$ on the universal principal $G$-bundle $\pi_{EG}:EG \to BG$. 
Let $\pi:(P,\theta) \to G$ be principal $G$-bundle with connection and let $f:X \to BG$ be a classifying map for the bundle with connection.
Since $EG$ is contractible, universal characteristic classes for principal $G$-bundles vanish upon pull-back to the total space.
Thus any differential character on $X$ with characteristic class a universal characteristic class for principal $G$-bundles is topologically trivial along the bundle projection.

This holds in particular for the Cheeger-Simons character $\widehat{CW}_\theta(\lambda,u) \in \widehat H^{2k}(X;\Z)$: since $u \in H^{2k}(BG;\Z)$, we have 
$$
\pi^*c(\widehat{CW}_\theta(\lambda,u)) 
\stackrel{\eqref{eq:c_CW}}{=} 
\pi^*cl_\Z u 
= 
f^*\pi^*u 
=
0 \,.
$$
From the exact sequence \eqref{eq:long_ex_sequ} we conclude that $\widehat{CW}_\theta(\lambda,u)$ admits sections along $\pi$.
A \emph{canonical} such section will be obtained by the Cheeger-Chern-Simons construction below.  

\subsubsection{Prescribing the covariant derivative}
To begin with, we lift the map \mbox{$CW_\theta \times c_\Z$} from to $R^{2k}(\pi;\Z)$:

\begin{prop}\label{lem:lift_CCS}
Let $(P,\theta) \to X$ be a principal $G$-bundle with connection.
Then the Chern-Weil map $CW_\theta$ has a canonical natural lift $CCS_\theta$ such that the diagram
\begin{equation*}
\xymatrix{
&& R^{2k}(\pi;\Z) \ar[dd] \\
&& \\
K^{2k}(BG;\Z) \ar@{.>}[uurr]^{CCS_\theta} \ar[rr]_{CW_\theta \times c_\Z} && R^{2k}(X;\Z) 
} 
\end{equation*}
commutes.
\end{prop}

\begin{proof} 
Since $EG$ is contractible, the long exact sequence for the mapping cone complex of the bundle projection $\pi_{EG}:EG \to BG$ reads: 
$$
\cdots \to
\underbrace{H^{2k-1}(EG;\Z)}_{=\{0\}} 
\to
H^{2k}(\pi;\Z) 
\xrightarrow{\cong}
H^{2k}(BG;\Z) 
\to 
\underbrace{H^{2k}(EG;\Z)}_{=\{0\}} 
\to \cdots
$$
In particular, we obtain isomorphisms $p:H^{2k}(\pi_{EG};\Z) \xrightarrow{\cong} H^{2k}(BG;\Z)$.
For a universal characteristic class $u \in H^{2k}(BG;\Z)$, we denote by $\tilde u:= p^{-1}(u) \in H^{2k}(\pi_{EG};\Z)$ its pre-image under this isomorphism.

Now let $\pi:(P,\theta) \to X$ be a principal $G$-bundle with connection.
We define the lift $CCS_{\theta}: K^{2k}(BG;\Z) \to R^{2k}(\pi;\Z)$ by 
\begin{equation}
CCS_\theta(\lambda,u)
:=(CW_\theta(\lambda),CS_\theta(\lambda),cl_\Z(\tilde u)) \,.
\end{equation}
As above, $cl_\Z:H^{2k}(\pi_{EG};\Z) \to H^{2k}(\pi;\Z)$ denotes the pull-back with the classifying map $f:X \to BG$.
In terms of a universal connection $\Theta$ on $\pi_{EG}:EG \to BG$ we have:
\begin{align*}
CCS_\theta(\lambda,u)
&=
f^*CCS_\Theta(\lambda,u)\\
&=
f^*(CW_\Theta(\lambda),CS_\Theta(\lambda),\tilde u) \,.
\end{align*}
Clearly, the composition of the map $CCS_\theta$ with the forgetful map $R^{2k}(\pi;\Z) \to R^{2k}(X;\Z)$ yields the map $CW_\theta \times cl_\Z$.
By construction, the map $CCS_\theta$ is natural with respect to pull-back of bundles with connection by smooth maps $g:X' \to X$.

It remains to check that $CCS_\Theta$ indeed takes values in $R^{2k}(\pi;\Z)$.
Let $(\lambda,u) \in K^{2k}(G;\Z)$.
We show that $(CW_\Theta(\lambda),CS_\Theta(\lambda)) \in \Omega^{2k}(\pi_{EG})$ is $d_{\pi_{EG}}$-closed with integral periods:\footnote{This essentially follows from the proof of \cite[Prop.~3.15]{CS83}, but for convenience of the reader, we give the full argument.} 
By definition of the Chern-Weil and Chern-Simons forms, we have:
$$
d_{\pi_{EG}}(CW_\Theta(\lambda),CS_\Theta(\lambda))
=
(dCW_\Theta(\lambda),{\pi_{EG}}^*CW_\Theta(\lambda) - dCS_\Theta(\lambda)) 
=
0.
$$
Since $(\lambda,u) \in K^{2k}(G;\Z)$, the Chern-Weil form $CW_\Theta(\lambda)$ has inte\-gral periods.
The $\Ul$-valued cocycle $\exp(2\pi i \, CW_\Theta(\lambda)) \in Z^{2k}(BG;\Ul)$ vanishes on integral cycles and represents the trivial class in $H^{2k}(BG;\Ul) \cong \Hom(H_{2k}(BG;\Z),\Ul)$.
Hence there is a $\Ul$-valued cochain $w \in C^{2k-1}(BG;\Ul)$ satisfying $\exp(2\pi i \, CW_\Theta(\lambda)) = \delta w$.
We then have:
\begin{align*}
\delta (\pi^*w)
&= \pi^*(\delta w) \\
&= \pi^*(\exp(2\pi i \, CW_\Theta(\lambda)) \\
&= \exp(2\pi i \, \pi^*CW_\Theta(\lambda)) \\
&= \exp(2\pi i \, dCS_\Theta(\lambda) ) \\
&= \delta \exp(2\pi i \, CS_\Theta(\lambda)).
\end{align*}
Thus $(\exp(2\pi i \, CS(\lambda)) - \pi^*w)$ is a cocycle in $Z^{2k-1}(EG;\Ul)$.
Since the total space $EG$ is contractible, we find a cochain $v \in C^{2k-2}(EG;\Ul)$ such that \mbox{$(\exp(2\pi i CS(\lambda)) - \pi^*w) = \delta v$}.

Now let $(s,t) \in C_{2k}(\pi_{EG};\Z)$ be a relative cycle.
Then we have:
\begin{align*}
\exp\Big( 2\pi i \int_{(s,t)} (CW_\Theta(\lambda),CS_\Theta(\lambda)) \, \Big) 
&= (\delta w,\pi^*w + \delta v)(s,t) \\
&= (w,v)(\partial_\pi(s,t)) \\
&=1
\end{align*}
Clearly, this implies $\int_{(s,t)} (CW_\Theta(\lambda),CS_\Theta(\lambda)) \in\Z$.
Hence the pair of differential forms \mbox{$(CW_\Theta(\lambda),CS_\Theta(\lambda)) \in \Omega^{2k}(\pi_{EG})$} has integral periods.  
In particular, the image of the relative de Rham class $[CW(\lambda),CS(\lambda)]_{dR}$ under the de Rham isomorphism lies in the image of the reduction of coefficients map $H^{2k}(\pi_{EG};\Z) \to H^{2k}(\pi_{EG};\R)$, $\tilde u \mapsto \tilde u_\R$.

It remains to show that $(CW_\Theta(\lambda),CS_\Theta(\lambda))$ and $\tilde u_\R$ match in real cohomology, i.e.~that $dR(CW_\Theta(\lambda),CS_\Theta(\lambda)) = \tilde u_\R \in H^{2k}(\pi_{EG};\R)$.
This follows from the commutative diagram
\begin{equation*}
\xymatrix{
\Omega^{2k}_0(\pi_{EG}) \ar[rr]^(0.45){dR} \ar[d]_{pr_1} && H^{2k}(\pi_{EG};\R) \ar^p[d]_{\cong} && H^{2k}(\pi_{EG};\Z) \ar[ll] \ar[d]^{\cong} \\
\Omega^{2k}_0(BG) \ar[rr]_(0.45){dR} && H^{2k}(BG;\R) && H^{2k}(BG;\Z). \ar[ll] \\
}
\end{equation*}
Explicitly, we have $p(dR(CW_\Theta(\lambda),CS_\Theta(\lambda))) = dR(CW_\Theta(\lambda)) = u_\R$ from \eqref{eq:def:K2k}
and hence $dR(CW_\Theta(\lambda),CS_\Theta(\lambda)) = \tilde u_\R$. 
\end{proof}

\subsubsection{The Cheeger-Chern-Simons character}
In the same way as the Cheeger-Simons construction $\widehat{CW}_\theta$ lifts the map Chern-Weil construction $CW_\theta \times cl_\Z$ to a differential character valued map, there is a canonical natural lift of $CCS_\theta$ to a relative differential character valued map $\widehat{CCS}_\theta$.
The relative differential character $\widehat{CCS}_\theta(\lambda,u)$ trivializes the differential character $\widehat{CW}(\lambda,u)$ along the bundle projection $\pi$ like the Chern-Simons form $CS_\theta(\lambda)$ trivializes the Chern-Weil form $\pi^*CW(\lambda)$ in de Rham cohomology.
It is the unique natural section of the Cheeger-Simons character with prescribed covariant derivative equal to the Chern-Simons form. 

\begin{thm}[Cheeger-Chern-Simons construction]\label{thm:CCS}
Let $G$ be a Lie group with finitely many components.
Let $(\lambda,u) \in K^{2k}(G;\Z)$.
For any principal $G$-bundle with connection $\pi:(P,\theta) \to X$, there exists a unique relative differential character $\widehat{CCS}_\theta(\lambda,u) \in \widehat H^{2k}(\pi;\Z)$ such that the following holds:

\noindent
The curvature and covariant derivative of $\widehat{CCS}_\theta(\lambda,u)$ are given by the Chern-Weil and Chern-Simons form:
\begin{equation}\label{eq:CCS_comm_1}
(\curv,\cov)(\widehat{CCS}_\theta(\lambda,u)) 
= (CW_\theta(\lambda),CS_\theta(\lambda)) \,.  
\end{equation}
The \emph{Cheeger-Chern-Simons character} $\widehat{CCS}_\theta(\lambda,u)$ trivializes the Chern-Simons character $\widehat{CW}_\theta(\lambda,u)$ along the bundle projection $\pi$:
\begin{equation}\label{eq:CCS_comm_2}
\vds_\pi (\widehat{CCS}_\theta(\lambda,u)) = \widehat{CW}_\theta(\lambda,u) \,. 
\end{equation}
The \emph{Cheeger-Chern-Simons construction} $\widehat{CCS}_\theta$ is natural with respect to pull-back by smooth maps, i.e., for any smooth map $f:X' \to X$ and the pull-back bundle $f^*(P,\theta)$, we have:\footnote{Here we denote the pull-back along the pull-back diagram of $f:X \to X'$ simply by $f^*$. Strictly speaking, we would have $(f,F)^*$, where $F:f^*P \to P$ is the induced bundle map on the pull-back bundle. Likewise, the pull-back connection $f^*\theta$ is given by the connection $1$-form $F^*\theta \in \Omega^1(f^*P)$.}
\begin{equation}\label{eq:CCS_comm_3}
f^*\widehat{CCS}_\theta(\lambda,u) = \widehat{CCS}_{f^*\theta}(\lambda,u) \,. 
\end{equation}
From \eqref{eq:CCS_comm_1} and \eqref{eq:CCS_comm_2}, we obtain the commutative diagram:
\begin{equation}\label{eq:diag_CCS_CCS}
 \xymatrix{
&&& & \widehat H^{2k}(\pi;\Z) \ar[dd]^{(\curv,\cov,c)} \ar[dl]^{\vds_\pi} \\
K^{2k}(G;\Z) \ar[rrr]^(0.6){\widehat{CW}_\theta} \ar@{.>}[rrrrd]_{CCS_\theta} \ar[rrr] \ar@{.>}[rrrru]^{\widehat{CCS}_\theta} \ar[ddrrr]_{CW\times cl_\Z} &&& \widehat H^{2k}(X;\Z) \ar[dd]  & \\
 &&& & R^{2k}(\pi;\Z) \ar[dl] \\
 &&&  R^{2k}(X;\Z) \,. & 
} 
\end{equation}%
In particular, we have 
\begin{equation}\label{eq:CCS_CCS}
(\curv,\cov,c)(\widehat{CCS}_\theta(\lambda,u)) 
=
CCS_\theta(\lambda,u)
=
(CW_\theta(\lambda),CS_\theta(\lambda),u(P))\,.
\end{equation}
\end{thm}

\begin{proof}
We first prove uniqueness.
By the requirement \eqref{eq:CCS_comm_3} that the Cheeger-Chern-Simons construction be natural with respect to pull-back of principal $G$-bundles with connections, it is uniquely determined by the map $\widehat{CCS}_\Theta: K^{2k}(G;\Z) \to \widehat H^{2k}(\pi_{EG};\Z)$ on the universal principal $G$-bundle $\pi_{EG}:EG \to BG$ with a fixed universal connection $\Theta$.
We show that this map is uniquely determined by \eqref{eq:CCS_comm_1} and \eqref{eq:CCS_comm_2}.

It is well-known that $H^{2k-1}(BG;\R) = \{0\}$ for any $k \geq 1$.
For convenience if the reader we briefly sketch the argument:
By assumption, $G$ has finitely many components.
Thus it has a maximal compact subgroup $K$ such that $K \subset G$ is a homotopy equivalence \cite[Ch.~XV.3]{H65}.
The induced map of classifying spaces yields isomorphisms $H^*(BG;\R) \cong H^*(BK;\R)$.
Let $K_0 \subset K$ be the connected component.
The induced map $BK_0 \to BK$ is a finite covering.
Let $T \subset K_0$ be a maximal torus.
Then we have $H^*(BK_0;\R) \cong H^*(BT;\R)^W$, where $W=$ is the Weyl group of $K_0$ \cite{H65}.
Since $H^*(BT;\R)$ is an exterior algebra in even degree generators, also $H^*(BG;\R)$ is generated by elements of even degree.

Now consider the long exact sequence of the mapping cone complex of the bundle projection $\pi_{EG}:EG \to BG$: 
\begin{equation*}
\cdots \to \underbrace{H^{2k-2}(EG;\R)}_{=\{0\}} \to H^{2k-1}(\pi_{EG};\R) \to \underbrace{H^{2k-1}(BG;\R)}_{=\{0\}} \to H^{2k-1}(EG;\R) \to \cdots
\end{equation*}
Thus $H^{2k-1}(\pi_{EG};\R) = \{0\}$ and the exact sequence \eqref{eq:sequ_R} reads:
\begin{equation}
0
\to 
\underbrace{\frac{H^{2k-1}(\pi_{EG};\R)_{\phantom{\R}}}{H^{2k-1}(\pi_{EG};\Z)_\R}}_{=\{0\}}
\to
\widehat H^{2k}(\pi_{EG};\Z) 
\xrightarrow{(\curv,\cov,c)}
R^{2k}(\pi_{EG};\Z) 
\to 
0 \,.\label{eq:ex_sequ_R_EG}
\end{equation}
Hence the map $(\curv,\cov,c): \widehat H^{2k}(\pi_{EG};\Z) \to R^{2k}(\pi_{EG};\Z)$ is an isomorphism.

Now let $\widehat{CCS}_\Theta:K^{2k}(G;\Z) \to \widehat H^{2k}(\pi_{EG};\Z)$ be any map satisfying \eqref{eq:CCS_comm_1} and \eqref{eq:CCS_comm_2}.
Let $(\lambda,u) \in K^{2k}(BG;\Z)$ and $\tilde u \in H^{2k}(\pi_{EG};\Z)$ as in the proof of Lemma~\ref{lem:lift_CCS}.
By \eqref{eq:CCS_comm_2} we have $c(\vds (\widehat{CCS}_\Theta(\lambda,u))) = c(\widehat{CW}(\lambda,u)) =u$.
The isomorphism $p:H^{2k}(\pi_{EG};\Z) \to H^{2k}(BG;\Z)$, $\tilde u \mapsto u$, from the mapping cone exact sequence yields the identification $c(\widehat{CCS}_\Theta(\lambda,u)) = \tilde u$.
Together with \eqref{eq:CCS_comm_1}, we obtain 
$\widehat{CCS}_\Theta(\lambda,u) = (\curv,\cov,c)^{-1} (CW_\Theta(\lambda),CS_\Theta(\lambda),\tilde u)$.
Thus the Cheeger-Chern-Simons map $\widehat{CCS}_\Theta: K^{2k}(G;\Z) \to \widehat H^{2k}(\pi_{EG};\Z)$ for a universal bundle with universal connection $\pi_{EG}:(EG,\Theta) \to BG$ is the unique lift in the diagram
\begin{equation*}
\xymatrix{
&& \widehat H^{2k}(\pi_{EG};\Z) \ar[d]^{(\curv,\cov,c)} \\
K^{2k}(G;\Z) \ar[rr]_{CCS} \ar@{.>}[rru]^{\widehat{CCS}} && R^{2k}(\pi_{EG};\Z) \,. 
} 
\end{equation*}
In other words, $\widehat{CCS} = (\curv,\cov,c)^{-1} \circ CCS$.

To prove existence, we define the Cheeger-Chern-Simons map by the above formula and show that this construction satisfies the requirements: 
Let $\pi:(P,\theta) \to X$ be a principal $G$-bundle with connection.
Fix a classifying map $f:X \to BG$ such that $(P,\theta) = f^*(EG,\Theta)$.
Then put:
\begin{align}\label{eq:def_CCS}
\widehat{CCS}_\theta(\lambda,u) 
&:= 
f^*\widehat{CCS}_\Theta(\lambda,u) \notag \\
&=
f^*\big((\curv,\cov,c)^{-1}(CW_\Theta(\lambda),CS_\Theta(\lambda),\tilde u)\big) \,.
\end{align}
By the very definition, the construction is natural with respect to pull-back of $G$-bundles with connection, hence it satisfies \eqref{eq:CCS_comm_3}.
Moreover, 
\begin{align*}
(\curv,\cov)(\widehat{CCS}_\theta(\lambda,u)) 
&= 
f^*(\curv,\cov)(\widehat{CCS}_\Theta(\lambda,u)) \\
&= 
f^*(CW_\Theta(\lambda),CS_\Theta(\lambda)) \\
&=
(CW_\theta(\lambda),CS_\theta(\lambda)) \,. 
\end{align*}
This shows \eqref{eq:CCS_comm_1}.

Exactness of the sequence 
$$
0
\to
\underbrace{\frac{H^{2k-1}(BG;\R)\phantom{_\R}}{H^{2k-1}(BG;\Z)_\R}}_{=0} 
\to
\widehat H^{2k}(BG;\Z)
\to
R^{2k}(BG;\Z)
\to
0
$$
implies that the Cheeger-Simons character $\widehat{CW}_\Theta(\lambda,u)$ is uniquely determined by its curvature $\curv(\widehat{CW}_\Theta(\lambda,u)) = CW(\lambda)$ and characteristic class $c(\widehat{CW}_\Theta(\lambda,u)) = u$.
Since $c(\vds(\widehat{CCS}_\Theta(\lambda,u))) = u$ and $\curv(\vds(\widehat{CCS}_\Theta(\lambda,u))) = CW(\lambda)$ we conclude $\vds(\widehat{CCS}_\Theta(\lambda,u)) = \widehat{CW}_\Theta$.
By naturality of the Cheeger-Simons construction, we thus obtain:
$$
\vds(\widehat{CCS}_\theta) 
= 
\vds(f^*\widehat{CCS}_\Theta) 
=
f^*\vds(\widehat{CCS}_\Theta) 
= 
f^*(\widehat{CW}_\Theta)
=
\widehat{CW}_\theta \,.
$$
This proves \eqref{eq:CCS_comm_2}.

Finally, diagram \eqref{eq:diag_CCS_CCS} and formula \eqref{eq:CCS_CCS} are immediate from \eqref{eq:def_CCS}. 
\end{proof}

\begin{rem}
Existence of sections of the Cheeger-Simons character $\widehat{CW}_\theta(\lambda,u)$ with prescribed covariant derivative $CS_\theta(\lambda)$ also follows from \cite[Prop.~75]{BB13}, since the pair $(CW_\theta(\lambda),CS_\theta(\lambda)) \in \Omega^{2k}(\pi_{EG})$ is closed with integral periods. 
\end{rem}

\subsubsection{Multiplicativity}
It is well-known that the Cheeger-Simons construction is multiplicative: it defines a ring homomorphism $\widehat{CW}:K^{2*}(G;\Z) \to \widehat H^{2*}(X;\Z)$.
In \cite{BB13} we show that for any smooth map $\varphi:A \to X$, the graded group $\widehat H^*(\varphi;\Z)$ is a right module over the ring $\widehat H^*(X;\Z)$.
It is easy to see that the Cheeger-Chern-Simons construction is (almost) multiplicative with respect to this module structure:

\begin{prop}[Multiplicativity]
Let $\pi:(P,\theta) \to X$ be a principal $G$-bundle with connection.
Let $(\lambda_1,u_1)\in K^{2k_1}(G;\Z)$ and $(\lambda_2,u_2)\in K^{2k_2}(G;\Z)$.
Let $k:= k_1 + k_2$.
Then there exists a differential form $\varrho \in \Omega^{(2k-2)}(P)$ such that we have:
\begin{equation}\label{eq:CCS_mult}
\widehat{CCS}_\theta(\lambda_1,u_1) * \widehat{CW}_\theta(\lambda_2,u_2)
=
\widehat{CCS}_\theta(\lambda_1\cdot \lambda_2,u_1 \cup u_2) + \iota_\pi(0,\varrho). 
\end{equation}
\end{prop}

\begin{proof}
It suffices to prove this for the universal $G$-bundle $\pi_{EG}:EG \to BG$ with universal connection $\Theta$.
Relative differential characters in $\widehat H^{2*}(\pi_{EG};\Z)$ are uniquely determined by their curvature, covariant derivative and characteristic class.
Hence it suffices to compare those data for the two sides of \eqref{eq:CCS_mult}.

The Chern-Weil map $CW_\Theta$ is multiplicative while the Chern-Simons map $CS_\Theta$ is multiplicative only up to an exact form. 
Thus there exists a differential form $\varrho \in \Omega^{(2k-2)}(EG)$ such that $CS_\theta(\lambda_1) \wedge \pi_{EG}^*CW_\theta(\lambda_2) = CS_\theta(\lambda_1 \cdot \lambda_2) - d\varrho$.
This yields:
\begin{align*}
(\curv,\cov)\big(\widehat{CCS}_\Theta&(\lambda_1,u_1) \,*\, \widehat{CW}_\Theta(\lambda_2,u_2)\big) \\
&=
(\curv,\cov)(\widehat{CCS}_\Theta(\lambda_1,u_1)) \wedge \pi_{EG}^*\curv(\widehat{CW}_\Theta(\lambda_2,u_2)) \\
&=
\big(CW_\Theta(\lambda_1) \wedge CW_\Theta(\lambda_2),CS_\Theta(\lambda_1) \wedge \pi_{EG}^*CW_\Theta(\lambda_2)\big) \\
&=
\big(CW_\Theta(\lambda_1 \cdot \lambda_2),CS_\Theta(\lambda_1 \cdot \lambda_2)\big) + d_{\pi_{EG}}(0,\varrho) \\
&=
(\curv,\cov)(\widehat{CCS}_\Theta(\lambda_1 \cdot \lambda_2)+\iota_{\pi_{EG}}(0,\varrho))  \\
\end{align*}
and
\begin{align*}
c(\widehat{CCS}_\Theta(\lambda_1,u_1) * \widehat{CW}_\Theta(\lambda_2,u_2))
&=
c(\widehat{CCS}_\theta(\lambda_1,u_1)) \cup \pi_{EG}^*c(\widehat{CW}_\Theta(\lambda_2,u_2)) \\
&= \tilde u_1 \cup \pi^* u_2 \\
&=
\widetilde{u_1 \cup u_2} \\
&= 
c(\widehat{CCS}_\Theta(\lambda_1\cdot \lambda_2,u_1 \cup u_2)+\iota_{\pi_{EG}}(0,\varrho)) \,. \qedhere
\end{align*}
\end{proof}

\subsection{The Chern-Simons action}\label{subsec:CS_action_bound}
In the same way as the Cheeger-Simons character $\widehat{CW}_\theta(\lambda,u)$ generalizes the classical Chern-Simons action along oriented closed $(2k-1)$-manifolds, the Cheeger-Chern-Simons character generalizes the classical Chern-Simons action along oriented manifolds with boundary: 

Let $\pi:(P,\theta) \to X$ be a principal $G$-bundle with connection and $(\lambda,u) \in K^{2k}(G;\Z)$.
Let $M$ be a compact oriented $(2k-1)$-manifold with boundary, and denote by $i_{\partial M}:\partial M \to M$ the inclusion of the boundary. 
Let $f:M \to X$ be a smooth map and $F: f^*P \to P$ the induced bundle map.
Let $\sigma:M \to f^*P$ be a smooth section of the pull-back bundle.
Put $g:= F \circ \sigma|_{\partial M}:\partial M \to P$.
This way we obtain a map of pairs $(M,\partial M)\xrightarrow{(f,g)}(X,P)$.
The character $(f,g)^*\widehat{CCS}_\theta(\lambda,u) \in \widehat H^{2k}(i_{\partial M};\Z)$ is topologically trivial, since $H^{2k}(i_{\partial M};\Z)=0$ for dimensional reasons.
We factorize the map $(f,g)$ as 
$$
\xymatrix{
& (f^*P,f^*P) \ar[dr]^{(\pi,\id)} & \\
(M,f^*P) \ar[ur]^{(\sigma,\id)} \ar[rr]_{(\id,\id)} && (M,f^*P) \ar[d]^{(f,F)} \\
(M,\partial M) \ar[u]^{(\id,\sigma|_{\partial M})} \ar[rr]_{(f,g)} && (X,P)
}
$$
Note that differential characters in $\widehat H^{2k}(\id_{f^*P};\Z)$ are uniquely determined by their covariant derivative.
Thus we have 
\begin{align*}
(\pi,\id_{f^*P})^*(f,F)^* \widehat{CCS}_\theta(\lambda,u) 
&= 
\iota_{\id}(F^*CS_\theta(\lambda),0) \\
\intertext{and hence}
(f,g)^*\widehat{CCS}_\theta(\lambda,u) 
&= 
\iota_{i_{\partial M}}(\sigma^*CS_{f^*\theta}(\lambda),0). \\
\intertext{This yields}
\big(\widehat{CCS}_\theta(\lambda,u)\big)((f,g)_*[M,\partial M]) 
&=
\big((f,g)^*\widehat{CCS}_\theta(\lambda,u)\big)[M,\partial M] \\
&=
\exp \Big( 2\pi i \int_M \sigma^*CS_{f^*\theta} \Big).
\end{align*}
Since the left hand side only depends upon $\sigma|_{\partial M}$, so does the right hand side.

\subsection{Dependence upon the connection}\label{subsec:depend_conn_1}
As above let $(P,\theta) \to X$ be a principal $G$-bundle with connection and $(\lambda,u) \in K^{2k}(G;\Z)$.
In this section we discuss the dependence of the Cheeger-Chern-Simons character $\widehat{CCS}_\theta(\lambda,u)$ upon the connection $\theta$.  
We first review the well-known dependencies of the Chern-Weil form $CW_\theta(\lambda)$, the Chern-Simons form $CS_\theta(\lambda)$ and the Cheeger-Simons character $\widehat{CW}_\theta(\lambda,u)$ upon the connection $\theta$.

Let $\theta_0,\theta_1 \in \mathcal A(P)$ be connections on $\pi:P \to X$.
As explained in Section~\ref{sec:Chern-Simons}, the Chern-Weil forms for the two connections differ by the differential of the Chern-Simons form: 
\begin{equation*}
CW_{\theta_1}(\lambda) - CW_{\theta_0}(\lambda) 
\stackrel{\eqref{eq:CS_two_CW}}{=}
d CS(\theta_0,\theta_1;\lambda). %\label{eq:CW_two_conn}
\end{equation*}
Analogously, we find for the Chern-Simons forms of the two connections:
\begin{align}
CS_{\theta_1}(\lambda) - CS_{\theta_0}(\lambda)
&=
d \Big(\underbrace{\fint_{[0,1]} CS_{(1-t)\theta_0 + t \theta_1}(\lambda)}_{:=-\alpha(\theta_0,\theta_1;\lambda)}\Big) - \fint_{[0,1]} dCS_{(1-t)\theta_0 + t \theta_1} \notag \\
&=
-d \alpha(\theta_0,\theta_1;\lambda) - \fint_{[0,1]} \pi^*CW_{(1-t)\theta_0 + t \theta_1} \notag \\
&\stackrel{\eqref{eq:def_CS_two}}{=}
-d \alpha(\theta_0,\theta_1;\lambda) + \pi^*CS(\theta_0,\theta_1;\lambda).
\end{align}
Combining the formulae for the Chern-Weil and Chern-Simons forms, we thus have
\begin{align}
(CW_{\theta_1}(\lambda),CS_{\theta_1}(\lambda)) - &(CW_{\theta_0}(\lambda),CS_{\theta_0}(\lambda)) \notag \\
&=
d_\pi(CS(\theta_0,\theta_1,\lambda),\alpha(\theta_0,\theta_1,\lambda)).
\end{align}
Now consider the Cheeger-Simons characters for two connections $\theta_0,\theta_1 \in \mathcal A(P)$.
Choose smooth classifying maps $f_i:X \to BG$, $i=0,1$, for the bundle with connection $\theta_i$.
Let $\theta(t):= (1-t) \theta_0 + t \theta_1$ be the straight line joining the two connections, and $f_t:X \to BG$ a smooth family of smooth classifying maps for the connections $\theta_t$, $t \in [0,1]$.
Then the map $F:X \times [0,1] \to BG$, $F(t,\cdot) := f_t$, is a smooth homotopy from $f_0$ to $f_1$.
The homotopy formula \eqref{eq:homotopy_abs} yields\footnote{Note that by the orientation conventions, we have $\fint_{[0,1]} \omega = (-1)^{k-1} \int_0^1 \omega_s ds$ for any $k$-form $\omega$.}:
\begin{align*}
\widehat{CW}_{\theta_1}(\lambda,u) - \widehat{CW}_{\theta_0}(\lambda,u)
&=
f_1^*\widehat{CW}_{\Theta}(\lambda,u) - f_0^*\widehat{CW}_{\Theta}(\lambda,u) \\
&\stackrel{\eqref{eq:homotopy_abs}}{=}
\iota \Big( \int_0^1 F^* CW_\Theta(\lambda) \Big) \\
&=
\iota \Big( -\fint_{[0,1]} CW_{(1-t)\theta_0 + t \theta_1}(\lambda) \Big) \\
&=
\iota \big( CS(\theta_0,\theta_1;\lambda)\big).
\end{align*}

We obtain the analogous result for dependence of the Cheeger-Chern-Simons character upon the connection:
\begin{prop}[Dependence upon the connection]
Let $\pi:P \to X$ be a principal $G$-bundle with connections $\theta_0$, $\theta_1 \in \mathcal A(P)$.
Let $(\lambda,u) \in K^{2k}(G;\Z)$.
Then we have:
\begin{equation}
\widehat{CCS}_{\theta_1}(\lambda,u) - \widehat{CCS}_{\theta_0}(\lambda,u)
=
\iota_\pi(CS(\theta_0,\theta_1;\lambda),\alpha(\theta_0,\theta_1;\lambda)). \label{eq:CCS_two_conn}
\end{equation} 
\end{prop}

\begin{proof}
As above choose classifying maps $f_t$ for the connections $\theta_t:=(1-t)\theta_0 + t \theta_1$ for $t \in [0,1]$.
Denote by $F:(X,P)\times [0,1] \to (BG,EG)$ the induced homotopy from $f_0$ to $f_1$. 
Using the homotopy formula \eqref{eq:homotopy_rel} for relative characters, we find:
\begin{align*}
\widehat{CCS}_{\theta_1}(\lambda,u) - \widehat{CCS}_{\theta_0}(\lambda,u)
&\stackrel{\phantom{\eqref{eq:homotopy_rel}}}{=}
f_1^*\widehat{CCS}_{\Theta}(\lambda,u) - f_0^*\widehat{CCS}_{\Theta}(\lambda,u) \\
&\stackrel{\eqref{eq:homotopy_rel}}{=}
\iota_\pi \Big(\! \int_0^1 (F^*CW_\Theta(\lambda),-F^*CS_\Theta(\lambda)) \!\Big) \\
&\stackrel{\phantom{\eqref{eq:homotopy_rel}}}{=}
\iota_\pi \Big(\! -\fint_{[0,1]} (CW_{(1-t)\theta_0 + t\theta_1},CS_{(1-t)\theta_0 + t\theta_1}(\lambda)) \!\Big) \\
&\stackrel{\phantom{\eqref{eq:homotopy_rel}}}{=}
\iota_\pi \Big(\! CS(\theta_0,\theta_1;\lambda),\alpha(\theta_0,\theta_1;\lambda) \!\Big). \qedhere
\end{align*}
\end{proof}

%%%%%%%%%%%%%%%%%%%%%%%%%%%%%%%%%%%%%%%%%%%%%%%%%%%%%%%%%%%%%%%%%%%%%%%%%
\section{Transgression}\label{sec:transgression}
%%%%%%%%%%%%%%%%%%%%%%%%%%%%%%%%%%%%%%%%%%%%%%%%%%%%%%%%%%%%%%%%%%%%%%%%%

In this section we discuss transgression of Cheeger-Simons characters.
On the one hand we have the usual transgression of absolute and relative characters to (free and based) loop spaces.
On the other hand, we derive a generalization of the transgression map in the universal principal $G$-bundle from integral cohomology to Cheeger-Simons characters.
The two transgressions coincide only topologically under the homotopy equivalence between $G$ and $\mathcal L_0(BG)$.
A further notion of transgression may be obtained from work of Murray and Vozzo \cite{MV10} in combination with fiber integration of differential characters.

\subsection{Transgression to loop space}
In \cite[Ch.~9]{BB13} we construct transgression of (absolute) differential characters on $X$ to mapping spaces and in particular to the free loop space $\LL X$:
$$
\tau: \widehat H^*(X;\Z) \to \widehat H^{*}(\LL X;\Z), \quad h \mapsto \widehat\pi_!(\ev^*h) \,.  
$$
Here $\ev:\LL X \times S^1 \to X$, $(\gamma,t) \mapsto \gamma(t)$, denotes the evaluation map and $\widehat\pi_!$ the fiber integration for the trivial bundle $\pi:\LL X \times S^1 \to \LL X$.
In \cite[II, Ch.~5]{BB13} we generalize the concept of fiber integration to relative differential characters.
For a smooth map $\varphi:A \to X$, we thus have the transgression map
$$
\tau: \widehat H^*(\varphi;\Z) \to \widehat H^{*}(\LL(X,A);\Z), \quad h \mapsto \widehat\pi_!(\ev^*h) \,.
$$   
Here $\LL(X,A)$ denotes the Fr\'echet manifold of (pairs of) smooth maps $\gamma: S^1 \to (X,A)$ and $\ev: \LL(X,A) \times S^1 \to (X,A)$, $(\gamma,t) \mapsto \gamma(t)$, denotes the evaluation map.

Note that both transgression maps have natural restrictions to maps with target the based loop spaces.
We will apply these transgression maps in Section~\ref{sec:diff_string} to the Cheeger-Chern-Simons characters.

\subsection{Transgression along smooth maps}
In this section we briefly review the notion of cohomology transgression along a smooth map $f:E \to X$.
Then we specialize to the case of a fiber bundle $\pi:E \to X$.
Note that for principal bundles, two different notions of transgression are present in the literature.
We follow the one used in \cite{CS74,HL74,CS83}, where transgression is a degree $-1$ homomorphism from a subset of the cohomology of the base into a quotient of the cohomology of the fiber.
This transgression is inverse to the one in \cite{B53}.
The former is called suspension by some authors, since it is closely related to the suspension isomorphism.
Moreover, it is related to the transgression to loop space as considered in the previous section.
We review these relations in Appendix~\ref{app:transgression} below.

Let $f:E \to X$ be a smooth map, and let $x \in \im(f)$.
Let $E_x:= f^{-1}(x)$ be the fiber over $x$.
Denote by $i_x: \{x\} \hookrightarrow X$ and $i_{E_x}: E_x \hookrightarrow E$ the natural inclusions.
The relative cohomology $H^*(E,E_x;\Z)$ is the same as the mapping cone cohomology $H^*(i_{E_x};\Z)$ of the inclusion, and similarly for $i_x$.
There are two ways to define the cohomology transgression
$$
T_f: H^*(X;\Z) \supset \dom(T) \to \frac{H^{*}(E_x;\Z)}{i_{E_x}^*H^{*}(E;\Z)},
$$
one using cohomology relative to the fiber $H^*(E,E_x;\Z)$, the other one using the mapping cone cohomology $H^*(f;\Z)$. 
We review both, since they differ in differential cohomology:

A cohomology class $u \in H^*(X;\Z)$, $*\geq 1$, is called \emph{transgressive} if $f^*u=0$. 
The domain of the transgression is the set of trangressive elements in $H^*(X;\Z)$.
By the commutative diagram
\begin{equation}\label{eq:diagram_T_F_rel}
\xymatrix{
&& H^*(X,\{x\};\Z) \ar@<-0.3ex>[d]_{f^*} \ar@<0.3ex>@{.>}[d] \ar[r] & H^*(X;\Z) \ar[d]_{f^*}  \ar@<0.6ex>@{.>}[l] \\
H^{*}(E;\Z) \ar[r] & H^{*}(E_x;\Z) \ar[r] & H^*(E,E_x;\Z) \ar[r] \ar@<0.6ex>@{.>}[l] & H^*(E;\Z),
}
\end{equation}
the set of transgressive classes maps under pull-back by $f$ to the image of $H^{*}(E_x;\Z) \to H^*(E,E_x;\Z)$.
The transgression is the mapping induced by the dashed lines in \eqref{eq:diagram_T_F_rel}.
More explicitly, it is the composition of maps:
\begin{equation}
H^*\!(\!X;\Z\!) \approx H^*\!(\!X,\!\{x\}\!;\Z\!) 
\supset \dom(T)
\xrightarrow{f^*}
H^*\!(\!E,E_x;\Z\!)
\xrightarrow{\id}
\frac{H^{*}\!(\!E_x;\Z\!)}{i_{E_x}^*H^{*}\!(\!E;\Z\!)}
\label{eq:def_T_f_rel}
\end{equation}
Equivalently, we may use the mapping cone cohomology of the map $f$.\footnote{See also \cite[Ch.~A]{BS08} for a similar treatment of transgression.}
Denote by $f_x:E_x \to \{x\}$ the restriction of $f$ to the fiber $E_x$.
By the commutative diagram
\begin{equation}\label{eq:diagram_T_F_cone}
\xymatrix{
H^{*}(E;\Z) \ar[r] \ar[d]_{i_{E_x}^*} & H^*(f;\Z) \ar[r] \ar@<-0.3ex>[d]_{i_{E_x}^*} \ar@<0.3ex>@{.>}[d] & H^*(X;\Z) \ar[d]_{i_x^*} \ar@<0.6ex>@{.>}[l] \ar[r]^{f^*} & H^*(E;\Z) \\
H^{*}(E_x;\Z) \ar[r]^{\approx} & H^*(f_x;\Z) \ar[r] \ar@<0.6ex>@{.>}[l] & H^*(\{x\};\Z), 
}
\end{equation}
the set of transgressive classes is the image of $H^*(f;\Z) \to H^*(X;\Z)$.
The transgression is the mapping induced by the dashed lines in \eqref{eq:diagram_T_F_cone}.
More explicitly, it is the composition of maps:
\begin{equation}
H^*\!(\!X;\Z\!)  
\supset \dom(T)
\xrightarrow{\approx}
\frac{H^*\!(\!f;\Z\!)}{H^{*}(\!E;\Z\!)}
\xrightarrow{i_{E_x}^*}
\frac{H^{*}\!(\!f_x;\Z\!)}{i_{E_x}^*H^{*}\!(\!E;\Z\!)}
\xrightarrow{-\id}
\frac{H^{*}\!(\!E_x;\Z\!)}{i_{E_x}^*H^{*}\!(\!E;\Z\!)}
\label{eq:def_T_f_cone}
\end{equation}
The easiest way to see that the two definitions of $T$ coincide is by describing them on the level of cocycles:
Let $\mu \in Z^k(X;\Z)$ be a cocycle representing the transgressive class $u \in H^*(X;\Z)$.
By assumption, $f^*u=0$, thus there exists a cochain $\nu \in C^{*}(E;\Z)$ such that $f^*\mu=\delta \nu$.
Moreover, since $i_x^*u \in H^*(\{x\};\Z) = \{0\}$, we find a cochain $\alpha \in C^{*}(\{x\};\Z)$ such that $\delta \alpha = i_x^*\mu$.
Thus the pair $(\mu,\alpha) \in Z^{*}(X,\{x\};\Z)$ represents the image of $u$ under the isomorphism $H^*(X;\Z) \to H^*(X,\{x\};\Z)$, whereas the pair $(\mu,\nu) \in Z^*(f;\Z)$ represents a pre-image of $u$ under the map $H^*(f;\Z) \to H^*(X;\Z)$.
Then we have
$$
\delta (f_x^*\alpha - i_{E_x}^*\nu)
= f_x^*\delta \alpha - i_{E_x}^*\delta\nu  
= f_x^*i_x^*\mu - i_{E_x}^*f^*\mu
= 0.
$$
Trangression maps the transgressive class $u = [\mu]$ to the equivalence class of the cohomology class $[f_x^*\alpha - i_{E_x}^*\nu] \in H^{*}(E_x;\Z)$ in the quotient $\frac{H^{*}(E_x;\Z)}{i_{E_x}^*H^{*}(E;\Z)}$.
The description by diagram \eqref{eq:diagram_T_F_rel} realizes this mapping through
\begin{equation*}
\xymatrix@C=4mm{
&& [\mu,\alpha] \ar@{|->}[d] & [\mu] \ar@{|->}[l] \\
[f_x^*\alpha - i_{E_x}^*\nu] & [(0,f_x^*\alpha-i_{E_x}^*\nu) - \delta_{i_{E_x}}(\nu,0)] \ar@{|->}[l] & f^*[\mu,\alpha] \ar@{=}[l] 
}
\end{equation*}
The decription by diagram \eqref{eq:def_T_f_cone} realizes this mapping through
\begin{equation*}
\xymatrix@C=4mm{
&& [\mu,\nu] \ar@{|->}[d] & [\mu] \ar@{|->}[l] \\
[f_x^*\alpha - i_{E_x}^*\nu] & [(0,-f_x^*\alpha+i_{E_x}^*\nu) + \delta_{f_x}(\alpha,0)] \ar@{|->}[l] & [i_x^*\mu,i_{E_x}^*\nu] \ar@{=}[l] 
}
\end{equation*}
Thus both descriptions \eqref{eq:def_T_f_rel}, \eqref{eq:def_T_f_cone} of the transgression yield the same well-defined homomorphism 
$$
T_f: H^*(X;\Z) \supset \dom(T) \to \frac{H^{*}(E_x;\Z)}{i_{E_x}^*H^{*}(E;\Z)}, \qquad [\mu] \mapsto \big[ [f_x^*\alpha - i_{E_x}^*\nu] \big], 
$$
with $\nu \in C^{*}(E;\Z)$ and $\alpha \in C^{*}(\{x\};\Z)$ as above.

\begin{rem}[Naturality]
Transgression along smooth maps is natural with respect to mappings of pairs in the following sense: 
Let $f:E \to X$ and $f':E' \to X'$ be a smooth maps.
Let $\Phi:E \to E'$ and $\varphi:X \to X'$ be smooth maps such that the diagram
\begin{equation*}
\xymatrix{
E' \ar[r]^\Phi \ar_{f'}[d] & E \ar[d]^f \\
X' \ar[r]_\varphi & X
}
\end{equation*}
commutes.
Then the transgressions along $f$ and $f'$ are related through
$$
\Phi^* \circ T_f
=
T_{f'} \circ \varphi^*.
$$
\end{rem}

\begin{rem}\label{rem:T_BG}
Let $G$ be a Lie group and $\pi_{EG}:EG \to BG$ a universal principal $G$-bundle.
Let $x \in BG$ and identify the fiber $EG_x$ with $G$.
Since $EG$ is contractible, all cohomology classes on $BG$ are transgressive. 
Thus transgression along the bundle projection $\pi_{EG}$ is a group homomorphism
$$
T:=T_{\pi_{EG}}: H^*(BG;\Z) \to H^{*}(G;\Z).
$$
\end{rem}
In Section~\ref{sec:univ_transg}, we generalize this transgression homomorphism to a homomorphism of differential cohomology groups.

\subsection{Splitting results}\label{sec:split}
In this section, we show that for a compact Lie group, the even degree differential cohomology of the classifying space $BG$ splits into Cheeger-Simons characters and topologically trivial characters.
The arguments are analogous to those in \cite[Ch.~A.2]{BSS14}, where we derive a splitting of the curvature exact sequence for differential cohomology on a manifold $X$ of finite type. 
Contrary to the case of an arbitrary mani\-fold $X$, the splitting on $BG$ is canonical, once the universal connection is fixed.
For the odd degree differential cohomology of $BG$, we derive (non canonical) splittings into flat characters with torsion class and topologically trivial characters.

\begin{thm}\label{prop:split_even}
Let $G$ be a compact Lie group.
Let $\pi_{EG}:EG \to BG$ be a universal principal $G$-bundle with fixed universal connection $\Theta$.
Let $k \geq 2$.
Then we have canonical splittings
\begin{align}
\widehat H^{2k}(BG;\Z) &= \widehat{CW}_\Theta\big(\,K^{2k}(G;\Z)\,\big) \oplus \iota\big(\,\Omega^{2k-1}(BG)\,\big) \label{eq:split_BG} \\
\widehat H^{2k}(\pi_{EG};\Z) &= \widehat{CCS}_\Theta\big(\,K^{2k}(G;\Z)\,\big) \oplus \iota\big(\,\Omega^{2k-1}(\pi_{EG})\,\big). \label{eq:split_pi_EG}
\end{align}
These splittings are compatible, in the sense that the canonical homomorphism \mbox{$\vds_{\pi_{EG}}:\widehat H^{2k}(\pi_{EG};\Z) \to \widehat H^{2k}(BG;\Z)$} commutes with the projection to the factors.
\end{thm}

\begin{proof}
By work of Cartan \cite{C50} (see also \cite[Ch.~8]{D78}) the map $CW$ induced by the Chern-Weil construction on the universal bundle is an isomorphism:
\begin{equation*}
\xymatrix{
I^k(G) \ar^{CW_\Theta}[rr] \ar^{\approx}[drr]_{CW} && \Omega^{2k}_\mathrm{cl} \ar^{dR}[d] \\
&& H^{2k}(BG;\R).
}
\end{equation*}
The set of all Chern-Weil forms $CW_\Theta(I^k(G)) \subset \Omega^{2k}_\mathrm{cl}(BG)$ is a $b_{2k}(BG)$-dimensional subspace.
Since $b_{2k}(BG) = \dim(I^k(G)) < \infty$, it has topological complements.
Since the embedding $CW_\Theta(I^k(G)) \hookrightarrow \Omega^k_\mathrm{cl}(BG)$ factorizes the isomorphism $CW:I^k(G) \to H^{2k}(BG;\R)$, the space of exact forms $d\Omega^{2k-1}(BG) \subset \Omega^{2k}(BG)$ is a canonical candidate for the complement.

Likewise, the set of all Chern-Weil forms with integral periods $CW_\Theta(I^k_0(G)) \subset \Omega^{2k}_0(BG)$ is a finitely generated free $\Z$-submodule, complemented by $d\Omega^{2k-1}(BG)$. 
Thus we obtain the canonical splittings
\begin{align*}
\Omega^{2k}_\mathrm{cl}(BG) &= CW_\Theta(I^k(G)) \oplus d\Omega^{2k-1}(BG) \\
\Omega^{2k}_0(BG) &= CW_\Theta(I^k_0(G)) \oplus d\Omega^{2k-1}(BG). \\
H^{2k}(BG;\Z) &= CW(I^k_0(G)) \oplus \mathrm{Tor}(H^{2k}(BG;\Z)).
\end{align*}
Here $\mathrm{Tor}(H^{2k}(BG;\Z)) \subset H^{2k}(BG;\Z)$ denotes the torsion subgroup.
The above splittings induce a splitting of the set $R^{2k}(BG;\Z)$ of forms with integral periods and corresponding integer cohomology classes:
\begin{equation}\label{eq:split_R}
R^{2k}(BG;\Z) = \big((c,\curv) \circ \widehat{CW}_\Theta\big)(K^{2k}(G;\Z)) \oplus d\Omega^{2k-1}(BG) \times \{0\}. 
\end{equation}
By the exact sequence \eqref{eq:sequ_R} for absolute differential cohomology, the homomorphism $(\curv,c): \widehat H^{2k}(BG;\Z) \to R^{2k}(BG;\Z)$ is an isomorphism, since $H^{2k-1}(BG;\R)$ vanishes.
From the sequences \eqref{eq:short_ex_sequ}, we obtain the commutative diagram
\begin{equation*}
\xymatrix{
\frac{\Omega^{2k-1}(BG)}{\Omega^{2k-1}_0(BG)} \ar_{\iota}[d] \ar^d[r] & d\Omega^{2k-1}(BG) \ar[d] \\
\widehat H^{2k}(BG;\Z) \ar_\curv[r] & \Omega^{2k}_0(BG) .
}
\end{equation*}
with injective vertical maps and surjective horizontal maps.
The inverse of the homomorphism $(\curv,c)$ maps the first summand in \eqref{eq:split_R} to the space of all Cheeger-Simons characters $\widehat{CW}_\Theta(K^{2k}(G;\Z))$ and the second summand to the space of all topologically trivial characters $\iota(\Omega^{2k-1}(BG))$.
Thus we obtain the splitting \eqref{eq:split_BG}.

For relative differential cohomology, we obtain the splitting \eqref{eq:split_pi_EG} by the same reasoning, using in addition the isomorphisms $H^{2k}(\pi_{EG};\Z) \to H^{2k}(BG;\Z)$, $u \mapsto \tilde u$, and $((\curv,\cov),c):\widehat H^{2k}(\pi_{EG};\Z) \to R^{2k}(\pi_{EG};\Z)$.
By construction, the homomorphism $\vds_{\pi_{EG}}:\widehat H^{2k}(\pi_{EG};\Z) \to \widehat H^{2k}(BG;\Z)$ maps the Cheeger-Chern-Simons character $\widehat{CCS}_\Theta(\lambda,u)$ to the corresponding Cheeger-Simons character $\widehat{CW}_\Theta(\lambda,u)$.
Thus the splittings \eqref{eq:split_BG} and \eqref{eq:split_pi_EG} are compatible.
\end{proof}

Similarly, we can split the odd degree differential cohomology of $BG$.
Here we have no Cheeger-Simons characters.
The splitting neither depends upon the choice of a universal connection, nor is it canonical.
To obtain this splitting, we first split the subspace of flat characters in $\widehat H^n(BG;\Z)$.

\begin{prop}\label{prop:split_odd}
Let $G$ be a compact Lie group and $\pi_{EG}:EG \to BG$ a universal principal $G$-bundle.
Then the subspace of topologically trivial characters $\widehat H^n_\mathrm{triv.}(BG;\Z) = c^{-1}(0) = \iota(\Omega^{n-1}(BG)) \subset \widehat H^n(BG;\Z)$ has the following (non-canonical) splitting:\footnote{By abuse of notation, we denote the inclusion of flat characters $j:H^{n-1}(BG;\Ul) \to \widehat H^n(BG;\Z)$ and its composition with the coefficient homomorphism $H^{n-1}(BG;\R) \to H^{n-1}(BG;\Ul)$ by the same symbol.
}
\begin{align}
\widehat H^n_\mathrm{triv.}(BG;\Z) &\approx j\big(H^{n-1}(BG;\R)\big) \oplus d\Omega^{n-2}(BG). \label{eq:split_triv}%
\intertext{
For odd degree differential cohomology, we obtain the (non-canonical) splitting: 
}
\widehat H^{2k-1}(BG;\Z) &\approx j\big( H^{2k-2}(BG;\Ul) \big) \oplus d\Omega^{2k-2}(BG). \label{eq:split_BG_odd}
\end{align}
\end{prop}

\begin{proof}
The splittings are obtained similarly to those in \cite[Ch.~A.2]{BSS14} for arbitrary manifolds of finite type:
For any degree $n \in \N$, let $F_n \subset \Omega^n(BG)$ be a linear complement to the subspace $\Omega^n_\mathrm{cl}(BG)$ of closed forms.  
The splitting $\Omega^n(BG) = \Omega^n_\mathrm{cl}(BG) \oplus F_n$ induces a splitting of the upper row of the commutative diagram 
\begin{equation*}
\xymatrix{
0 \ar[r] &
\frac{H^{n-1}(BG;\R)_{\phantom{\R}}}{H^{n-1}(BG;\Z)_\R} \ar[r] \ar[d] &
\frac{\Omega^{n-1}_\mathrm{cl}(BG)}{\Omega^{n-1}_0(BG)} \ar^d[r] \ar[d]_\iota & 
d\Omega^{n-1}(BG) \ar[r] \ar[d] & 
0 \\
0 \ar[r] & 
H^{n-1}(BG;\Ul) \ar[r]_j & 
\widehat H^{n-1}_\mathrm{triv.}(BG;\Z) \ar_\curv[r] & 
d\Omega^{n-1}(BG) \ar[r] &
0.
}
\end{equation*}
Using the exact sequences \eqref{eq:short_ex_sequ} and applying the vertical maps, we obtain the splitting $\widehat H^n_\mathrm{triv.}(BG;\Z) = j(H^{n-1}(BG;\R)) \oplus \iota(F^{n-2})$.
This yields \eqref{eq:split_triv}.

Now let $n = 2k-1$.
Since $H^\mathrm{odd}(BG;\R) = \{0\}$, any closed form of odd degree is exact. 
Replacing $\widehat H^{n-1}_\mathrm{triv.}(BG;\Z)$ by $\widehat H^{2k-1}(BG;\Z)$ and $\Omega^{2k-1}_0(BG)$ by $d\Omega^{2k-2}(BG)$ in the above diagram, we still get a diagram with exact rows.
Applying the vertical maps yields $\widehat H^{2k-1}(BG;\Z) = j(H^{2k-2}(BG;\R)) \oplus \iota(F^{2k-2})$ which in turn yields \eqref{eq:split_BG_odd}.
\end{proof}

\begin{rem}
In the proof of Proposition \ref{prop:split_odd} we did not use Cartans theorem.
Thus the proposition also holds for noncompact Lie groups with finitely many components.
\end{rem}

\subsection{Universal transgression in differential cohomology}\label{sec:univ_transg}
Let $G$ be a compact Lie group and $\pi_{EG}:EG  \to BG$ a universal principal $G$-bundle with fixed universal connection $\Theta$.
The goal of this section is to extend the cohomology transgression $T:H^{2k}(BG;\Z) \to H^{2k-1}(G;\Z)$ on a universal principal $G$-bundle of a compact Lie group to a transgression homomorphism
$\widehat T: \widehat H^{2k}(BG;\Z) \to \widehat H^{2k-1}(G;\Z)$ on differential cohomology.
To this end, we first introduce differential cohomology transgression $\widehat T_f$ along any smooth map $f:E \to X$ along the lines of \eqref{eq:diagram_T_F_cone} and \eqref{eq:def_T_f_cone}.
It is then straight forward to show that the differential cohomology transgression $\widehat T_f$ vanishes on products.
Finally we use the splitting \eqref{eq:split_BG} of $\widehat H^\mathrm{even}(BG;\Z)$ to show that in case $f = \pi_{EG}$, the transgression $\widehat T_{\pi_{EG}}$ has a canonical lift $\widehat T:\widehat H^\mathrm{even}(BG;\Z) \to \widehat H^\mathrm{odd}(G;\Z)$.

Before defining the differential cohomology transgression $\widehat T_f$, let us briefly motivate our approach:
As reviewed in Section~\ref{sec:transgression}, there are two equivalent ways to define the ordinary transgression on smooth singular cohomology: \eqref{eq:def_T_f_rel} uses the pull-back along a smooth map $f:E \to X$ of the relative cohomology $H^*(X,\{x\};\Z)$, whereas \eqref{eq:def_T_f_cone} uses the mapping cone sequence of the map $f$.
In case $f= \pi_{EG}$, the contractibility of $EG$ implies that all cohomology classes on $BG$ are transgressive and that transgression on $BG$ is a map to the cohomology of $G \approx EG_x$, see Remark~\ref{rem:T_BG}.
The situation changes when we replace smooth singular cohomology by differential cohomology in these approaches, since the long exact sequence \eqref{eq:long_ex_sequ} for absolute and relative differential cohomology contains only three differential cohomology groups.

Replacing singular cohomology groups by differential cohomology groups in \eqref{eq:diagram_T_F_rel} (with $f=\pi_{EG}$) results in a particularly small set of transgressive characters, namely those characters $h \in \widehat H^*(BG;\Z)$ satisfying $\pi_{EG}^*h=0$.
Since $EG$ is contractible, this condition is equivalent to $h$ being flat.
This is in contrast to singular cohomology, where any cohomology class on $BG$ is transgressive.
On the other hand, this version of differential cohomology transgression takes values in $\frac{\widehat H^*(EG_x;\Z)}{i_{E_x}^*j(H^*(EG;\Ul))} = \widehat H^*(EG_x;\Z)$.
Replacing singular cohomology groups by differential cohomology groups in \eqref{eq:diagram_T_F_cone} (with $f=\pi_{EG}$) yields a transgression map defined on the whole differential cohomology of $BG$ with values in the quotient $\frac{\widehat H^*(EG_x;\Z)}{i_{E_x}^*H^*(EG;\Z)}$.
We will show below that it has a canonical lift to $\widehat H^*(EG_x;\Z)$. 
For this reason, we find it more appropriate to use the mapping cone cohomology to define the differential cohomology transgression.

As with smooth singular cohomology, the differential cohomology transgression may be defined not just for the bundle projection $\pi_{EG}$ but for any smooth map $f:E \to X$:

\begin{defn}\label{def:transgression}
Let $f:E \to X$ be a smooth map and $x \in X$.
A differential character $h \in \widehat H^n(X;\Z)$ is said to be \emph{transgressive along $f$} iff it is topologically trivial along $f$, i.e.~$c(f^*h)=0$.
Denote by $\dom(\widehat T_f) \subset \widehat H^*(X;\Z)$ the set of characters which are transgressive along $f$. 
The \emph{differential cohomology transgression} 
\begin{equation*}
\widehat T_f: H^n(X;\Z) \supset \dom(\widehat T) \to \frac{\widehat H^{n-1}(E_x;\Z)}{i_{E_x}^*\widehat H^{n-1}(E;\Z)}
\end{equation*}
is the group homomorphism defined by the dashed lines in the commutative diagram 
\begin{equation}\label{eq:diagram_T_hat_f}
\xymatrix{
\widehat H^{n-1}(E;\Z) \ar[r] \ar[d]^{i_{E_x}^*} & \widehat H^n(f;\Z) \ar[r] \ar@<0.4ex>[d]^{i_{E_x}^*} \ar@<-0.4ex>@{.>}[d] & \widehat H^n(X;\Z) \ar[d]^{i_x^*} \ar@<-0.6ex>@{.>}[l] \ar[r]^{c \circ f^*} & H^n(E;\Z) \\
\widehat H^{*}(E_x;\Z) \ar[r]_{\ti_{f_x}} & \widehat H^*(f_x;\Z) \ar[r] \ar@<-0.6ex>@{.>}[l] & \widehat H^*(\{x\};\Z). 
}
\end{equation}
More explicitly, for a character $h \in \widehat H^n(X;\Z)$, the transgressed character $\widehat T_f(h)$ is given as follows: choose a relative character $h' \in \widehat H^n(f;\Z)$ such that $\vds_f(h')=h$.
Then put 
\begin{equation}\label{eq:def_T_hat}
\widehat T_f(h)
:= 
[-(\ti_{f_x})^{-1}(i_{E_x}^*h')] \in \frac{\widehat H^{n-1}(E_x;\Z)}{i_{E_x}^*\widehat H^{n-1}(E;\Z)}.          
\end{equation}
\end{defn}

\begin{rem}\label{rem:kernel_T_hat}
It is obvious from the definition that topologically trivial characters are transgressive and the transgression vanishes on those:
Let $h = \iota(\mu)$ for some $\mu \in \Omega^{n-1}(X)$.
Choose an arbitrary form $\nu \in \Omega^{2k-2}(E)$.
Then $\vds_f(\iota(\mu,\nu)) = \iota(\mu)$ and $i_x^*\mu=0$.
Thus $\ti_{\pi_x}(i_{E_x}^*\nu) = (0,i_{E_x}^*\nu) = (i_x,i_{E_x})^*(\mu,\nu)$.
Thus $\widehat T_f$ maps $\iota(\mu)$ to the equivalence class of $i_{E_x}^*\nu$ in the quotient $\frac{\widehat H^{n-1}(E_x;\Z)}{i_{E_x}^*\widehat H^{n-1}(E;\Z)}$, which is $0$. 
\end{rem}

As for smooth singular cohomology (or any generalized cohomology theory), the differential cohomology transgression is natural with respect to maps of pairs:

\begin{prop}[Naturality]
Let $f:E \to X$ and $f':E' \to X'$ be a smooth maps.
Let $\Phi:E \to E'$ and $\varphi:X \to X'$ be smooth maps such that the diagram
\begin{equation*}
\xymatrix{
E' \ar[r]^\Phi \ar_{f'}[d] & E \ar[d]^f \\
X' \ar[r]_\varphi & X
}
\end{equation*}
commutes.
Then the transgressions along $f$ and $f'$ are related through
\begin{equation}\label{eq:T_natural}
\Phi^* \circ \widehat T_f
=
\widehat T_{f'} \circ \varphi^*. 
\end{equation}
\end{prop}

\begin{proof}
Clearly, the maps $\Phi$ and $\varphi$ induce a map of diagrams as in \eqref{eq:diagram_T_hat_f}.
More explicitly, let $h \in \widehat H^n(X;\Z)$, transgressive along $f$.
Let $y \in X'$ and $x:=\varphi(y)$. 
Then $\varphi^*h$ is transgressive along $f'$, since $c({f'}^*\varphi^*h) = c(\Phi^*f^*h) = \Phi^* c(f^*h)=0$.
Let $h' \in \widehat H^n(f;\Z)$ with $\vds_f(h')=h$.
Since the maps $\vds_{f_x}$, $\vds_{f'_y}$ and $\ti_{f_x}$, $\ti_{f'_y}$ in the mapping cone sequence commute with pull-backs along $(\Phi,\varphi)$, we obtain: 
\begin{align*}
\widehat T_{f'}(\varphi^*h)
&=
[-(\ti_{f'_y})^{-1}(i_{E_y}^*(\Phi,\varphi)^*h')] \\
&=
[-(\ti_{f'_y})^{-1}(\Phi^* i_{E_x}^*h')] \\
&=
\Phi^* [-(\ti_{f_x})^{-1}(i_{E_x}^*h')] \\
&=
\Phi^*\widehat T_f(\varphi^*h). \qedhere
\end{align*}
\end{proof}

In the following, we restrict to the case of fiber bundles instead of arbitrary smooth maps. 
Recall that relative differential cohomology is a right module over the absolute differential cohomology ring and the characteristic class is a module homomorphism \cite[II, Ch.~4]{BB13}.
In particular, if $h_1 \in \widehat H^{k_1}(X;\Z)$ is transgressive in the bundle $\pi:E \to X$, then for any character $h_2 \in \widehat H^{k_2}(X;\Z)$, the internal product $h_1*h_2$ is transgressive in $E \to X$, whereas the external product $h_1 \times h_2$ is transgressive in the bundle $\pi \times \id_X: E \times X \to X \times X$.
Moreover, if $h_1' \in \widehat H^{k_1}(\pi;\Z)$ with $\vds_\pi(h_1') = h_1$, then $\vds_\pi(h_1'* h_2) = h_1 * h_2$ and $\vds_{\pi \times \id_X}(h_1'\times h_2) = h_1 \times h_2$.

Once the differential cohomology transgression is known to be natural with respect to bundle maps, it is straight forward to show that it vanishes on products.

\begin{prop}
Let $\pi:E \to X$ be a fiber bundle and $h_1,h_2 \in \widehat H^*(X;\Z)$.
If $h_1$ or $h_2$ is transgressive, then we have $\widehat T_\pi(h_1 * h_2)=0$.
\end{prop}

\begin{proof}
The proof resembles the one for singular cohomology in \cite[Ch.~A]{BS08}.
The argument basically relies on naturality of the external product and of the transgression:

Suppose $h_1 \in \widehat H^{k_1}(X;\Z)$ is transgressive. 
Choose $h_1' \in \widehat H^{k_1}(\pi;\Z)$ such that $\vds_\pi(h_1')=h_1$.
Let $\Delta_X:X \to X \times X$ be the diagonal map.
Then we have the pull-back diagram
\begin{equation*}
\xymatrix@C=70pt{
E \ar[r]^{(\id_E \times \pi) \circ \Delta_E} \ar[d]_\pi & E \times X \ar[d]^{\pi \times \id_X} \\
X \ar[r]_{\Delta_X} & X \times X.
}
\end{equation*}
On the right hand side, the inclusion of the fiber $E_x \times \{x\}$ over $(x,x)$ is given by the map $i_{E_x} \times i_x$.
By naturality of transgression and of the external product, we obtain: 
\begin{align*}
\widehat T_\pi(h_1 * h_2) 
&= 
\widehat T_\pi(\Delta_X^*(h_1 \times h_2)) \\ 
&\stackrel{\eqref{eq:T_natural}}{=}
\big((\id_E \times \pi) \circ \Delta_E\big)^* \widehat T_{\pi \times \id_X} (h_1' \times h_2) \\
\intertext{From \eqref{eq:def_T_hat}, applied to the bundle $E \times X \to X \times X$, we obtain\footnote{To simplifiy notation, we drop the index $i_{E_x} \times i_x$ of the map $\ti$ in the long exact sequence \eqref{eq:long_ex_sequ} of the restricted projection $(\pi \times \id_X)_{(x,x)}:E_x \times \{x\} \to \{(x,x)\}$.}:}
\widehat T_{\pi \times \id_X} (h_1' \times h_2)
&\stackrel{\eqref{eq:def_T_hat}}{=}
\big[\, -\ti\big(\, (i_{E_x} \times i_x)^*(h_1' \times h_2) \,\big) \,\big] \\
&=
\big[\, -\ti\big(\, i_{E_x}^*h_1' \times \underbrace{i_x^*h_2}_{=0} \,\big)\,\big] \\
&=
0.
\end{align*}
Clearly, this yields $\widehat T_\pi(h_1 * h_2)=0$. 

Now assume that $h_2$ is transgressive.
Since the internal product on $\widehat H^*(X;\Z)$ is graded commutative, we have $h_1 * h_2 = (-1)^{k_1 + k_2} h_2 * h_1$.
By the argument above, we obtain $\widehat T_\pi(h_1 * h_2)= (-1)^{k_1 + k_2} \widehat T_\pi(h_2 * h_1) = 0$. 
\end{proof}

We conclude this section by showing that in even degrees the universal differential cohomology transgression $\widehat T_{\pi_{EG}}$ has a canonical lift to a homomorphism with values in the odd degree differential cohomology of $G$.
This is due to the splitting \eqref{eq:split_BG}.
To motivate the construction of this lift, recall that the Chern-Simons construction \cite{CS74} represents the transgression on real cohomology:
for any $\lambda \in I^k_0(G)$, the restriction of the Chern-Simons form $CS_\Theta(\lambda) \in \Omega^{2k-1}(EG)$ to the fiber $EG_x \approx G$ is closed with de Rham cohomology class $[i_{EG_x}^*CS_\Theta(\lambda)]_{dR} = T([CW_\Theta(\lambda)]_{dR}) \in H^{2k-1}_{dR}(G;\Z)$.
In fact, this form neither depends upon the base point $x$ nor upon the choice of universal connection.
It may be expressed explicitly in terms of the Maurer-Cartan form of $G$, as explained in \cite[p.~55]{CS74}.

We thus expect to find a transgression homomorphism which maps a Cheeger-Simons character $\widehat{CW}_\Theta(\lambda,u) \in \widehat H^{2k}(BG;\Z)$ to a differential character on $G$ with curvature $i_{EG_x}^*CS_\Theta(\lambda)$ and characteristic class $T(u)$.

\begin{rem}
If $b_{2k-2}(G) = 0$, the exact sequence \eqref{eq:sequ_R} implies that differential characters in $\widehat H^{2k-1}(G;\Z)$ are uniquely determined by their characteristic class and curvature. 
Thus for any $(\lambda,u) \in K^{2k}(G;\Z)$ there exists a unique differential character $h \in \widehat H^{2k-1}(G;\Z)$ satisfying 
\begin{align*}
c(h) &= T(u) \\ 
\curv (h) &= i_{EG_x}^*CS_\Theta(\lambda).
\end{align*}
This character represents the equivalence class $\widehat T_{\pi_{EG}}(\widehat{CW}_\Theta(\lambda,u)) \in \frac{\widehat H^{2k-1}(G;\Z)}{\iota(\Omega^{2k-2}(G))}$.

The requirements are satisfied e.g.~in degrees $k=2,3$ on a compact, simply connected Lie group $G$:
By the Hopf theorem, $H^*(G;\R)$ is an exterior algebra in odd degree generators.
Now $\pi_1(G)= \{0\}$ implies $b_1(G)=0$ and hence $H^{2k-2}(G;\R) =\{0\}$ for $k = 2,3$. 
\end{rem}

Even if these requirements are not satisfied, we have a canonical construction of characters in $\widehat H^\mathrm{odd}(BG;\Z)$ from Cheeger-Simons characters in $\widehat H^\mathrm{even}(BG;\Z)$:

\begin{thm}[Differential cohomology transgression]
Let $G$ be a compact Lie group.
Let $\pi_{EG}:EG \to BG$ be a universal principal $G$-bundle with fixed universal connection $\Theta$.
Then the transgression $\widehat T_{\pi_{EG}}$ from Definition~\ref{def:transgression} has a canonical lift $\widehat T$ on even degree differential cohomology:
\begin{equation}
\xymatrix{
&& \widehat H^{2k-1}(G;\Z)\ar[d] \\
\widehat H^{2k-1} \ar@{.>}[urr]^{\widehat T} \ar[rr]_{\widehat T_{\pi_{EG}}} && \frac{\widehat H^{2k-1}(G;\Z)}{\iota(\Omega^{2k-2}(G))}}
\end{equation}
With respect to the splitting \eqref{eq:split_BG}, the homomorphism $\widehat T$ is given as
\begin{align}
\widehat T: \widehat{CW}_\Theta(K^{2k}(G;\Z)) \oplus \iota(\Omega^{2k-1}(BG)) &\to \widehat H^{2k-1}(G;\Z) \notag \\
\widehat{CW}_\Theta(\lambda,u) \oplus \iota(\varrho) &\mapsto -(\ti_{\pi_x})^{-1}(i_{EG_x}^* \widehat{CCS}_\Theta(\lambda,u)). \label{def:T_hat}
\end{align}
\end{thm}

\begin{proof}
By Remark~\ref{rem:kernel_T_hat}, the transgression $\widehat T_{\pi_{EG}}$ vanishes on the second factor of $\widehat H^{2k}(BG;\Z)$ in \eqref{eq:split_BG}.
By Definition~\ref{def:transgression}, to evaluate the transgression $\widehat T_{\pi_{EG}}$ on a Cheeger-Simons character $h=\widehat{CW}_\Theta(\lambda,u) \in \widehat H^{2k}(BG;\Z)$, we need to find a relative charactcer $h'$ under the map $\vds_{\pi_{EG}} \in\widehat H^{2k}(\pi_{EG};\Z)$ satisfying $\vds_{\pi_{EG}}(h')=h$. 
By Theorem~\ref{thm:CCS}, the correspponding Cheeger-Chern-Simons character $h'=\widehat{CCS}_\Theta(\lambda,u)$ is the canonical choice of such a relative character.
Inserting $h'=\widehat{CCS}_\Theta(\lambda,u)$ into \eqref{eq:def_T_hat}, we observe that the map $\widehat T$ indeed lifts the transgression $\widehat T_{\pi_{EG}}$.  
\end{proof}

\begin{rem}
Since the Chern-Simons and Cheeger-Chern-Simons constructions are not multiplicative, the map $\widehat T$ does not vanish on products.
For characters $h_1$, $h_2 \in \widehat H^\mathrm{even}(BG;\Z)$, one of which is topologically trivial, we have $\widehat T(h_1*h_2) =0$, since also $h_1 * h_2$ is topologically trivial.
To the contrary, for $(\lambda_i,u_i) \in K^*(G;\Z)$, $i=1,2$, we have 
\begin{equation}\label{eq:T_hat_mult}
\widehat T\big(\widehat{CW}_\Theta(\lambda_1,u_1) * \widehat{CW}_\Theta(\lambda_2,u_2)\big)
=
-\iota(i_{EG_x}^*\varrho). 
\end{equation}
Here $\varrho \in \Omega^{2k-2}(EG)$ is a form which corrects the non-multiplicativity of the Chern-Simons construction:
By \eqref{eq:CCS_mult}, we have 
$$
\widehat{CCS}_\Theta(\lambda_1 \cdot \lambda_2, u_1 \cup u_2) = \widehat{CCS}_\Theta(\lambda_1,u_1) * \widehat{CW}_\Theta(\lambda,u) + \iota_{\pi_{EG}}(0,\varrho).
$$
By pull-back along $(i_x,i_{EG_x}):(\{x\},EG_x) \to (BG,EG)$ and naturality of the product, we obtain 
\begin{align*}
(i_x,i_{EG_x})^*&\widehat{CCS}_\Theta(\lambda_1 \cdot \lambda_2, u_1 \cup u_2) \\
&=
(i_x,i_{EG_x})^*\widehat{CCS}_\Theta(\lambda_1,u_1) * \underbrace{i_x^*\widehat{CW}_\Theta(\lambda,u)}_{=0} + \iota_{\pi_{EG}}(0,{EG_x}^*\varrho) \\
&=
\iota_{\pi_{EG}}(0,i_{EG_x}^*\varrho).
\end{align*}
Inserting into Definition~\ref{def:T_hat}, we obtain \eqref{eq:T_hat_mult}.
\end{rem}

\begin{rem}[Compatibility]\label{rem:T_hat}
The transgression $\widehat T: \widehat H^\mathrm{even}(BG;\Z) \to \widehat H^\mathrm{odd}(G;\Z)$ is compatible with the ususal cohomology transgression $T$ and the characteristic class.
Likewise, it is compatible with the Chern-Simons construction and the curvature.
Thus we have
\begin{align}
c \circ \widehat T  
&= 
T \circ c \label{eq:cT} \\
\curv(\widehat T(\widehat{CW}_\Theta(\lambda,u)))
&=
i_{EG_x}^*CS_\Theta(\lambda). \label{eq:curvT}
\end{align}  
Compatibility with characteristic class follows from the fact that the map $\widehat T$ is constructed through diagram chase arguments analogous to those in \eqref{eq:def_T_f_cone} with singular cohomology replaced by differential cohomology.
Compatibility with curvature follows from the construction of the Cheeger-Chern-Simons characters.
By \eqref{eq:CCS_comm_1}, we have:
\begin{align*}
\curv(\widehat T(\widehat{CW}_\Theta(\lambda,u)))
&=
\curv\big(-(\ti_{\pi_x})^{-1}(i_{EG_x}^* \widehat{CCS}_\Theta(\lambda,u))\big) \\
&=
\cov(i_{EG_x}^* \widehat{CCS}_\Theta(\lambda,u)) \\
&=
i_{EG_x}^*CS_\Theta(\lambda).
\end{align*}
If $b_{2k-2}(G)=0$, then $\widehat T(\widehat{CW}_\Theta(\lambda,u))$ is the unique character in $\widehat H^{2k-1}(G;\Z)$ satisfying \eqref{eq:cT} and \eqref{eq:curvT}.
This follows immediately from the exact sequence \eqref{eq:sequ_R}. 
\end{rem}

\subsection{A differential Hopf theorem}\label{sec:Hopf}
Let $G$ be a compact connected Lie group.
The classical Hopf theorem states that the real cohomology of $G$ is an exterior algebra on odd degree generators $p^1,\ldots,p^N$.
By the Borel transgression theorem, the real cohomology of the classifying space $BG$ is a polynomial algebra on even degree generators $q^1,\ldots,q^N$.
Moreover, the generators of $H^*(BG;\R)$ and $H^*(G;\R)$ are related by transgression, i.e.~$T(q^i) = p^i$ for $i=1,\ldots,N$.
In this section, we use the splitting \eqref{eq:split_BG} to transfer these results to differential cohomology.

Let $k \in \N$.
By the Cartan theorem, the universal Chern-Weil map descends to an isomorphism $CW: I^k_0(G) \to H^{2k}_{dR}(BG;\Z) \cong H^{2k}(BG;\Z)_\R$.
For $q^i$ as above, choose invariant polynomials $\lambda^i \in I^{k_i}_0(G)$ such that $CW(\lambda^i) = q^i$.
Moreover, choose $u^i \in H^{2k_i}(BG;\Z)$ such that $u^i_\R = q^i$ for $i=1,\ldots,N$.
Then $(\lambda^i,u^i) \in K^{2k_i}(G;\Z)$, and we have the corresponding Cheeger-Simons characters $\widehat{CW}_\Theta(\lambda^i,u^i) \in \widehat H^{2k_i}(BG;\Z)$.
Denote the curvatures of the transgressed Cheeger-Simons characters by
$$
\omega^i
:=\curv(\widehat T(\widehat{CW}_\Theta(\lambda^i,u^i))) = i_{EG_x}^*CS_\Theta(\lambda^i) \in \Omega^{2k_i-1}_0(G)
$$
Since $T(q^i) = T(u^i_\R) = T([CW_\Theta(\lambda^i)]_{dR}) = [\omega^i]_{dR} \in H^{2k_i-1}(G;\R)$, the canonical map $\Omega^*_\mathrm{cl}(G) \to H^*(G;\R)$ descends to an isomorphism $\Lambda(\omega^1,\ldots,\omega^N) \xrightarrow{\approx} H^*(G;\R)$.
Denote by $\Lambda_\Z(\omega^1,\ldots,\omega^N) \subset \Omega^*_0(G)$ the integral lattice generated by the forms $\omega^1,\ldots,\omega^N$.
If the cohomology of $G$ has no torsion, then we also obtain an isomorphism $\Lambda_\Z(\omega^1,\ldots,\omega^N) \xrightarrow{\approx} H^*(G;\Z)$.

Now the subring of $\widehat H^*(G;\Z)$ generated by the transgressed Cheeger-Simons characters $\widehat T(\widehat{CW}_\Theta(\lambda^i,u^i))$ is a direct summand of as a graded group.
This may be regarded as a differential cohomology version of the Hopf theo\-rem:
\begin{thm}[Differential Hopf theorem]\label{thm:Hopf}
Let $G$ be a compact, connected Lie group such that $H^*(G;\Z)$ has no torsion.
Choose $(\lambda^i,u^i) \in K^*(G;\Z)$, $i=1,\ldots,N$, such that 
\begin{align*}
H^*(BG;\R)
&=  
\R[{u^1}_\R,\ldots,{u^N}_\R] \\
H^*(G;\R)
&= 
\Lambda(T({u^1}_\R),\ldots,T({u^N}_\R)) = \Lambda(\omega^1,\ldots,\omega^N) \\
H^*(G;\Z)
&=
\Lambda_\Z(\omega^1,\ldots,\omega^N),
\end{align*}
where $\omega^i:=\curv(\widehat T(\widehat{CW}_\Theta(\lambda^i,u^i))) = i_{EG_x}^*CS_\Theta(\lambda_i) \in \Omega^{2k_i-1}_0(G)$ denote the curvature forms of the transgressed Cheeger-Simons characters.
Let $F \subset \Omega^*(G)$ be a topological complement to the subspace $\Omega^*_\mathrm{cl}(G)$ of closed forms. 
Then we have the following topological direct sum decomposition of graded groups:
$$
\widehat H^*(G;\Z)
= 
\big\langle \widehat T (\widehat{CW}_\Theta (\lambda^1,u^1)), \ldots,  \widehat T (\widehat{CW}_\Theta (\lambda^N,u^N)) \big\rangle_\Z \oplus \iota\big( \Lambda(\omega^1,\ldots\omega^N) \big) \oplus \iota(F).
$$
\end{thm}

In other words, up topologically trivial characters the differential cohomology of $G$ is generated by odd degree characters which are transgressions of Cheeger-Simons characters on $BG$.

\begin{proof}
By the assumption that the cohomology of $G$ has no torsion, we have $H^*(G;\Z) = H^*(G;\Z)_\R = H^*_{dR}(G;\Z)$.
Since $G$ is assumed to be compact, we may use the splitting arguments from \cite[Ch.~A]{BSS14}.
Any choice of $n$-forms which represent a basis of $H^n(G;\R)$ provides us with a splitting $\widehat H^n(G;\Z) \approx j(H^{n-1}(G;\Ul)) \oplus \Omega^n_0(G)$.

Since $H^*(G;\Z)$ has no torsion, any flat character is topologically trivial.
In other words, the image of the inclusion $j:H^*(G;\Ul) \to \widehat H^*(G;\Z)$ coincides with the image of the topological trivialization $\iota:\Omega^*(G) \to \widehat H^*(G;\Z)$, restricted to the space of closed forms.
Thus for any $n \in \Z$ we obtain a direct sum decomposition $\widehat H^n(G;\Z) \approx \iota(\Omega^{n-1}_\mathrm{cl}(G)) \oplus \Omega^n_0(G)$.
The inclusion of the first factor corresponds to topological trivialization, the projection to the second factor corresponds to the curvature map.

Using the Hopf theorem, we may choose these splittings more canonically:
Let $(\lambda^i,u^i) \in K^{2k_i}(G;\Z)$ and $\omega^i$ for $i=1,\ldots,N$ as above.
By the isomorphisms $\Lambda_\Z(\omega^1,\ldots,\omega^N) \xrightarrow{\approx} H^*(G;\Z)$ and $\Lambda(\omega^1,\ldots,\omega^N) \xrightarrow{\approx} H^*(G;\R)$ we obtain canonical topological splittings
\begin{align}
\Omega^*_\mathrm{cl}(G) 
&=
\Lambda(\omega^1,\ldots,\omega^N) \oplus d(\Omega^*(G)) \notag \\
\Omega^*_0(G) 
&=
\Lambda_\Z(\omega^1,\ldots,\omega^N) \oplus d(\Omega^*(G)) \label{eq:split_0} \\
\intertext{%
Choose a topological complement $F \subset \Omega^*(G)$ to the subspace $\Omega^*_\mathrm{cl}(G)$ of closed forms, e.g.~by introducing an auxiliary Riemannian metric and using the Hodge decomposition.
Then exterior differential retsricts to a topological isomorphism \mbox{$d:F \to d(\Omega^*(G))$} and we obtain the splittings
}%
\Omega^*(G) 
&=
\Lambda(\omega^1,\ldots,\omega^N) \oplus d(\Omega^*(G)) \oplus F \notag \\
\frac{\Omega^*(G)}{\Omega^*_0(G)}
&=
\frac{\Lambda(\omega^1,\ldots,\omega^N)}{\Lambda_\Z(\omega^1,\ldots,\omega^N)} \oplus F 
=
\frac{H^*(G;\R)}{H^*(G;\Z)} \oplus F. \notag
\end{align}
Now let $h \in \widehat H^*(G;\Z)$ be an arbitrary character.
Then its curvature form $\curv(h) \in \Omega^*_0(G) = \Lambda(\omega^1,\ldots,\omega^N) \oplus d(\Omega^*(G))$ has a unique decomposition 
$$
\curv(h) 
= 
\sum a_{i_1,\ldots,i_k} \cdot \omega^{i_1} \wedge \ldots \wedge \omega^{i_k} + d\mu
$$
with coefficients $a_{i_1,\ldots,i_k} \in \Z$ and $\mu \in F$.
Now put
$$ 
h'
:= 
\sum a_{i_1,\ldots,i_k} \cdot \widehat T(\widehat{CW}_\Theta(\lambda^{i_1},u^{i_1})) * \ldots * \widehat T(\widehat{CW}_\Theta(\lambda^{i_k},u^{i_k})) + \iota(\mu).
$$
Then we have $\curv(h) = \curv(h')$.
Thus there is a uniquely determined differential form $\varrho \in \Lambda(\omega^1,\ldots,\omega^N)$ such that $h = h' + \iota(\varrho)$.
Summarizing, we have obtain a unique decomposition of the character $h$ as 
$$
h 
=
\sum a_{i_1 \ldots i_k} \cdot \widehat T(\widehat{CW}_\Theta(\lambda^{i_1},u^{i_1})) * \ldots * \widehat T(\widehat{CW}_\Theta(\lambda^{i_N},u^{i_N})) + \iota(\varrho) + \iota(\mu) 
$$
This proves the theorem.
\end{proof}

\begin{rem}
If the integral cohomology has $G$ has torsion, then the differential cohomology of $G$ also contains flat characters with torsion characteristic class.
These could also be contained in the subring generated by $\widehat T(\widehat{CW}_\Theta(\lambda_i,u_i))$.
In this case one cannot expect that the differential cohomology has a direct sum decomposition, since the subgroup of torsion class characters generated by Cheeger-Simons characters need not be a direct summand of the subgroup of all torsion class characters.  
\end{rem}

\subsection{Transgression via the caloron correspondence}\label{sec:caloron_transg}
Let $G$ be a compact connected Lie group and $\LL G$ and $\Omega G$ the free and the based loop group. 
Calorons were introduced \cite{HS78, N84} as certain periodic $G$-instantons on a manifold of the form $X \times S^1$.  
Later it was observed \cite{GM88} that calorons are in 1-1 correspondence with $\LL G$-intantons on $X$.    
In mathemcatical terms, the so-called caloron correspondence \cite{MV10, HMV13} is a 1-1 correspondence between principal $\LL G$- or $\Omega G$-bundles over a manifold $X$ and principal $G$-bundles over $X \times S^1$.
Both sides of the correspondence may be equipped with connections: in this case, the caloron correspondence is a 1-1 correspondence between $G$-bundles over $X \times S^1$ with connection and $\LL G$- or $\Omega G$-bundles over $X$ with connection and a so-called \emph{Higgs field}, see \cite{MV10, HMV13} for details.

A particular $\LL G$-bundle with connection arises from the path fibration $\PP G \to G$.
This bundle carries a canonical Higgs field.
The caloron correspondence transfers the path fibration into a canonical principal $G$-bundle with connection $(\tilde P,\tilde \theta) \to G \times S^1$, see \cite{MV10}.
Since the base of this $G$-bundle is a fiber bundle with compact oriented fibers, we may apply fiber integration for differential characters as constructed in \cite{BB13}.
This yields another notion of differential cohomology transgression:

\begin{defn}\label{def:caloron_transg}
Let $G$ be a compact connected Lie group.
Let $\pi_{EG}:EG \to BG$ be a universal principal $G$-bundle with fixed universal connection $\Theta$.
Let $(\tilde P,\tilde \theta) \to G \times S^1$ be the principal $G$-bundle arising from the path fibration via the caloron transform.
Then we define the caloron transgression to be the homomorphism
\begin{align*}
\widehat T_\mathrm{cal}: \widehat H^{2k}(BG;\Z) = \widehat{CW}_\Theta(K^{2k}(G;\Z)) &\to \widehat H^*(G;\Z) \\ 
\widehat{CW}_\Theta(\lambda,u) \oplus \iota(\varrho) &\mapsto \fint_{S^1} \widehat{CW}_{\tilde\theta}(\lambda,u).
\end{align*}
\end{defn}

\begin{rem}\label{rem:T_cal}
Note that we could extend the above homomorphism to a map defined on all of $\widehat H^*(BG;\Z)$ by choosing a classifying map for the bundle with connection $f:G \times S^1 \to BG$ and setting
$$
\widehat H^*(BG;\Z) \to \widehat H^*(G;\Z), \qquad h \mapsto \fint_{S^1} f^*h. 
$$
But this map in general depends upon the choice of classifying map.
In contrast, the Cheeger-Simons characters $\widehat{CW}_{\tilde\theta}(\lambda,u) = f^*\widehat{CW}_\Theta(\lambda,u)$ are natural with respect to connection preserving bundle maps and thus do not depend upon the choice of classifying map.  
\end{rem}

In general, we do not know whether the caloron transgression $\widehat T_\mathrm{cal}$ coincides with the differential cohomology transgression $\widehat T$.
We only know that they have the same curvature and characteristic class, thus they coincide up to topologically trivial flat characters:

\begin{lem}\label{lem:TT}
Let $G$ be a compact connected Lie group.
Let $\widehat T$ be the universal differential cohomology transgression as defined in \ref{def:transgression}.
Let $\widehat T_\mathrm{cal}$ be the caloron transgression as defined in \ref{def:caloron_transg}.
Then we have
\begin{align*}
\curv \circ \widehat T \equiv \curv \circ \widehat T_\mathrm{cal} \\
c \circ \widehat T \equiv c \circ \widehat T_\mathrm{cal}.
\end{align*}
\end{lem}

\begin{proof}
Let $(\lambda,u) \in K^{2k}(G;\Z)$ and $\widehat{CW}_\Theta(\lambda,u) \in \widehat H^{2k}(BG;\Z)$ the corresponding Cheeger-Simons character.
By \eqref{eq:curvT}, we have $\curv(\widehat T(\widehat{CW}_\Theta(\lambda,u)) = \iota_x^*CS_\Theta(\lambda)$.
On the other hand, it is shown in \cite[Prop.~4.11]{MV10} that also $\curv(\widehat T_\mathrm{cal}(\widehat{CW}_\Theta(\lambda,u))) = \iota_x^*CS_\Theta(\lambda)$.
Similarly, by \eqref{eq:cT}, we have $c \circ \widehat T(\widehat{CW}_\Theta(\lambda,u)) = T(u)$.
On the other hand, it is shown in \cite[Prop.~3.4]{CJMSW05} that $c(T'(\widehat{CW}_\Theta(\lambda,u))) = \fint_{S^1} u(P) = T(u)$. 
\end{proof}

\begin{cor}
Let $G$ be a compact connected Lie group with $b_{2k-2}(G)=0$.
Then the transgressions
$$
\widehat T, \widehat T_\mathrm{cal}: \widehat H^{2k}(BG;\Z) \to \widehat H^{2k-1}(G;\Z)
$$
coincide.
\end{cor}

\begin{proof}
The claim immediately follows from Lemma~\ref{lem:TT} and the exact sequence \eqref{eq:sequ_R}.
\end{proof}

Taking the map $\widehat H^*(BG;\Z) \to \widehat H^*(G;\Z)$ as in Remark~\ref{rem:T_cal}, we obtain another variant of a Hopf theorem on differential cohomology:
As in Section~\ref{sec:Hopf}, choose $(\lambda^i,u^i) \in K^{2k_i}(G;\Z)$ such that $H^*(G;\R) = \Lambda(u^1_\R,\ldots,u^N_\R;\R)$.
Putting 
$$
\omega^i
= \curv\big(\widetilde T_\mathrm{cal}(\widehat{CW}_\Theta(\lambda^i,u^i))\big)
\stackrel{\eqref{eq:curvT}}{=} \curv\big(\widehat T(\widehat{CW}_\Theta(\lambda^i,u^i))\big),
$$
we have $H^*(G;\R) = \Lambda(\omega^1,\ldots,\omega^N)$.

\begin{prop}
Let $G$ be a compact, connected Lie group such that $H^*(G;\Z)$ has no torsion.
Let $\pi_{EG}:EG \to BG$ be a universal principal $G$-bundle with fixed universal connection $\Theta$.
Let $(\tilde P,\tilde\theta) \to G \times S^1$ be the canonical $G$-bundle arising from the path fibration via the caloron transform.
Choose $(\lambda^i,u^i) \in K^{2k_i}(G;\Z)$, $i=1,\ldots,N$, as above such that  
\begin{align*}
H^*(BG;\R)
&=  
\R[u^1_\R,\ldots,u^N_\R] \\
H^*(G;\R)
&= 
\Lambda(T(u^1_\R),\ldots,T(u^N_\R)) 
= \Lambda(\omega^1,\ldots,\omega^N).
\end{align*}
Let $f:G \times S^1 \to BG$ be a classifying map for the bundle with connection. 
Assume that $f(G \times S^1) \subset BG$ is an embedded submanifold.
Define
$$
\widetilde T_\mathrm{cal}: \widehat H^*(BG;\Z) \to \widehat H^*(G;\Z), \qquad h \mapsto \fint_{S^1}f^*h.
$$
Then we have the topological direct sum decomposition
$$
\widehat H^*(G;\Z)
= 
\big\langle \widetilde T_\mathrm{cal} (\widehat{CW}_\Theta (\lambda^1,u^1)), \ldots , \widetilde T_\mathrm{cal} (\widehat{CW}_\Theta (\lambda^N,u^N)) \big\rangle_\Z \oplus \widetilde T_\mathrm{cal}\big(\iota(\Omega^*(BG))\big).
$$
\end{prop}

\begin{proof}
Since $G$ is compact, we may choose a topological complement $F \subset \Omega^*(G)$ for the space $\Omega^*_\mathrm{cl}(G)$ of closed forms.
Since the cohomology of $G$ has no torsion, any flat character is topologically trivial. 
Thus $j(H^*(G;\Ul)) = \iota(\Omega^*_\mathrm{cl}(G))$.
By arguments similar to those in \cite[Ch.~A]{BSS14} and in the proof of Theorem~\ref{thm:Hopf}, we obtain the splitting of the curvature sequence:
\begin{align*}
\widehat H^*(G;\Z)
&\stackrel{\phantom{\eqref{eq:split_0}}}{\approx}
\iota(\Omega^*_\mathrm{cl}(G)) \oplus \Omega^*_0(G) \\
&\stackrel{\eqref{eq:split_0}}{=}
\iota(\Omega^*_\mathrm{cl}(G)) \oplus d(\Omega^*(G)) \oplus \Lambda_\Z(\omega^1,\ldots,\omega^N) \\
&\stackrel{\phantom{\eqref{eq:split_0}}}{=}
\iota(\Omega^*_\mathrm{cl}(G)) \oplus \iota(F) \oplus \Lambda_\Z(\omega^1,\ldots,\omega^N) \\
&\stackrel{\phantom{\eqref{eq:split_0}}}{=}
\iota(\Omega^*(G)) \oplus \Lambda_\Z(\omega^1,\ldots,\omega^N).
\end{align*}
By construction, the subspace $\Lambda_\Z(\omega^1,\ldots,\omega^N) \subset \Omega^*_0(G)$ is generated by the curvatures of the characters $\widetilde T_\mathrm{cal}(\widehat{CW}_\Theta(\lambda^i,u^i))$.
In other words, the curvature provides an isomorphism
$$
\curv:\big\langle \widetilde T_\mathrm{cal} (\widehat{CW}_\Theta (\lambda^1,u^1)), \ldots , \widetilde T_\mathrm{cal} (\widehat{CW}_\Theta (\lambda^N,u^N)) \big\rangle_\Z
\xrightarrow{\approx}
\Lambda_\Z(\omega^1,\ldots,\omega^N).
$$
The assumption that $f(G \times S^1) \subset BG$ is an embedded submanifold implies that the pull-back of differential forms $f: \Omega^*(BG) \to \Omega^*(G \times S^1)$ is surjective.
Note that also the fiber integration $\fint_{S^1}: \Omega^*(G \times S^1) \to \Omega^*(G)$ is surjective:
Let $\vartheta \in \Omega^1(S^1)$ be any $1$-form with integral $1$. 
Given a form $\mu \in \Omega^*(G)$, the up-down formula for the bundle $\mathrm{pr_1}:G \times S^1 \to G$ yields $\fint_{S^1}\mathrm{pr}_1^*\mu \wedge \mathrm{pr}_2^*\vartheta = \mu \wedge \fint_{S^1} \mathrm{pr}_2^*\vartheta = \mu$.
Thus the map $\fint_{S^1} \circ f^*: \Omega^*(BG) \to \Omega^*(G)$ is surjective.
Hence $\widetilde T_\mathrm{cal}\big(\iota(\Omega^*(BG))\big) = \iota(\Omega^*(G))$.
\end{proof}

%%%%%%%%%%%%%%%%%%%%%%%%%%%%%%%%%%%%%%%%%%%%%%%%%%%%%%%%%%%%%%%%%%%%%%%%%
\section{Differential trivializations of universal characteristic classes}\label{sec:diff_triv}
%%%%%%%%%%%%%%%%%%%%%%%%%%%%%%%%%%%%%%%%%%%%%%%%%%%%%%%%%%%%%%%%%%%%%%%%%
In this section we use Cheeger-Chern-Simons characters to establish a notion of differential refinements of trivializations of universal characteristic classes for principal $G$-bundles. 
Specializing to the class $\frac{1}{2}p_1 \in H^4(B\Spin_n;\Z)$ this yields our notion of differential String classes.

\subsection{Trivializations of universal characteristic classes}\label{subsec:triv}
Throughout this section let $G$ be a Lie group with finitely many components and $\pi:P \to X$ a principal $G$-bundle.
As above let $\pi_{EG}:EG \to BG$ be a universal principal $G$-bundle over the classifying space of $G$, i.e.~a principal $G$-bundle with contractible total space.
Let $u \in H^n(BG;\Z)$ be a universal characteristic class for principal $G$-bundles.
Equivalently, we may consider $u$ as a homotopy class of maps $BG \xrightarrow{u} K(\Z,n)$.
In the following we briefly review the notion and basic properties of trivializations of universal characteristic classes from \cite{R11}.

Let $\widetilde{BG_u}$ be the homotopy fiber of the map $BG \xrightarrow{u} K(\Z,n)$.
Let $f:X \to BG$ be a classifying map for the bundle $\pi:P \to X$, i.e.~$f^*EG \cong P$ as principal $G$-bundles over $X$.  
A trivialization of the class $u(P):=f^*u \in H^n(X;\Z)$ is by definition a homotopy class of lifts
\begin{equation*}
\xymatrix{
&& \widetilde{BG_u} \ar[d] \\
X \ar@{.>}[urr]^{\widetilde f} \ar_f[rr] && BG. 
}
\end{equation*}
The class $u(P)$ admits trivializations if and only if it is trivial, i.e.~$u(P)=0$.
By \cite[Prop.~2.3]{R11}, a trivialization of $u(P)$ gives rise to a cohomology class $q \in H^{n-1}(P;\Z)$ such that for any $x \in X$ we have:
\begin{equation}\label{eq:def_triv}
H^{n-1}(P_x;\Z) \ni i_{P_x}^*q = T(u) \in H^{n-1}(G;\Z). 
\end{equation}
Here $i_{P_x}:P_x \to P$ denotes the inclusion of the fiber $P_x:= \pi^{-1}(x) \subset P$ over $x \in X$.
A cohomology class $q \in H^{n-1}(P;\Z)$ satisfying \eqref{eq:def_triv} is called a \emph{$u$-trivialization class}.

The cohomology of the base acts on $u$-trivialization classes by $q \mapsto q + \pi^*w$, where $w \in H^{n-1}(X;\Z)$.
If $\widetilde H^j(G;\Z)= \{0\}$ for $j < n-1$, then trivializations of $u(P)$ are classified up to homotopy by $u$-trivialization classes $q \in H^{n-1}(P;\Z)$.
In this case, the transgression $T:H^n(BG;\Z) \to H^{n-1}(G;\Z)$ is an isomorphism and we have the Serre exact sequence
\begin{equation*}
\xymatrix@C=27pt{
\{0\} \ar[r] & H^{n-1}(X;\Z) \ar^{\pi^*}[r] & H^{n-1}(P;\Z) \ar^{i_{P_x}^*}[r] & \ar[r] H^{n-1}(G;\Z) \ar^{f^* \circ T^{-1}}[r] & H^n(X;\Z). 
}
\end{equation*}
In particular, the set of $u$-trivialization classes is a torsor for $H^{n-1}(X;\Z)$. 

\subsection{Differential trivializations}
We are looking for an appropriate notion of differential refinements of $u$-trivia\-lization classes.
Naively, one could define a differential $u$-trivialization to be any differential character $\widehat q \in \widehat H^{n-1}(P;\Z)$ whose characteristic class $c(\widehat q)$ is a $u$-trivialization class.
However, by the exact sequences \eqref{eq:short_ex_sequ} this would determine those differential characters only up to an infinite dimensional space of differential forms on $P$.
Instead, we expect that for an appropriate notion of differential $u$-trivializations, the space of all those is a torsor for the differential cohomology $\widehat H^{n-1}(X;\Z)$ (respectively $\pi^*\widehat H^{n-1}(X;\Z)$, in case pull-back by $\pi$ is not injective).  

Let $u \in H^n(BG;\Z)$ be a universal characteristic class in the image of the Chern-Weil map, i.e.~$n = 2k$ and there exists an invariant polynomial $\lambda \in I^k(\mathfrak g)$ such that $u_\R = [CW_\Theta(\lambda)]_{dR} \in H^n(BG;\R)$.  
By \cite{CS74}, the Chern-Simons construction is a version of the transgression homomorphism on the level of differential forms in the sense that restiction of the Chern-Simons form to any fiber represents the transgression of the associated Chern-Weil class:  
$$
H^{2k-1}_{dR}(P_x) \ni[i_{P_x}^* CS_\theta(\lambda)]_{dR} = T(u)_\R \in H^{2k-1}(G;\R),
$$
where $x \in X$ is an arbitrary point.
Thus we expect the curvature of a differential $u$-trivialization $\widehat q$ to be related to the Chern-Simons form $CS_\theta(\lambda)$.
However, this form is not closed, since $dCS_\theta(\lambda) = \pi^*CW_\theta(\lambda)$.

By assumption, we have $u(P) = f^*u=0$, and hence $[CW_\theta(\lambda)]_{dR} = u_\R =0$.
Thus there exist differential forms $\varrho \in \Omega^{2k-1}(X)$ such that $d\varrho = CW_\theta(\lambda)$.
Then the form $CS_\theta(\lambda) - \pi^*\varrho$ is closed.
Moreover, we may choose $\varrho$ such that $CS_\theta(\lambda) - \pi^*\varrho$ has integral periods.
This follows from the long exact sequence for mapping cone de Rham cohomology with integral periods:
$$
\xymatrix{
\ldots \ar[r] & H^{2k-1}_{dR}(P;\Z) \ar[r] & H^{2k}_{dR}(\pi;\Z) \ar[r] & H^{2k}_{dR}(X;\Z) \ar[r] & \ldots 
}
$$
Since $[CW_\theta(\lambda)]_{dR}=0$, the mapping cone class $[CW_\theta(\lambda),CS_\theta(\lambda)]_{dR}$ lies in the image of the homomorphism $H^{2k-1}_{dR}(P;\Z) \to H^{2k}_{dR}(\pi;\Z)$.
We thus find a pair of forms $(\varrho,\eta) \in \Omega^{2k-1}(\pi)$ such that $(CW_\theta(\lambda),CS_\theta(\lambda)) - d_\pi(\varrho,\eta)$ lies in the image of $\Omega^{2k-1}_0(P) \to \Omega^{2k}_0(\pi)$.
Thus $CW_\theta(\lambda) = d\varrho$ and $CS_\theta(\lambda) - \pi^*\varrho + d\eta$ has integral periods. 
But then also $CS_\theta(\lambda) - \pi^*\varrho$ has integral periods.

The space of forms $\varrho \in \Omega^{2k-1}(X)$ with $d\varrho = CW_\theta(\lambda)$ and $CS_\theta(\lambda) - \pi^*\varrho \in \Omega^{2k-1}_0(P)$ is a torsor for the infinite dimensional group $\Omega^{2k-1}_0(X)$.
For a fixed such form and a fixed $u$-trivialization class $q$, the set of differential characters $\widehat q \in \widehat H^{2k-1}(P;\Z)$ with curvature $\curv(\widehat q) = CS_\theta(\lambda) - \pi^*\varrho$ and characteristic class $c(\widehat q)=q$ is a torsor for the torus $\frac{H^{2k-1}(P;\R)_{\phantom{\R}}}{H^{2k-1}(P;\Z)_\R}$.

The condition that $CS_\theta(\lambda) - \pi^*\varrho$ has integral periods has a nice interpretation in terms of global sections of the Cheeger-Simons character $\widehat{CW}_\theta(\lambda,u) \in \widehat H^{2k}(X;\Z)$:
As above assume that $u(P)=0$.
Thus $\widehat{CW}_\theta(\lambda,u)$ is topologically trivial. 
By the long exact sequence \eqref{eq:long_ex_sequ} for the map $\id_X$ it has a global section.
By \eqref{eq:glob_sect} global sections are uniquely determined by their covariant derivative.
Thus for any form $\varrho \in \Omega^{2k-1}(X)$ with $d\varrho = CW_\theta(\lambda)$, we have $\vds_{\id}(\iota_{\id}(\varrho,0)) = \widehat{CW}_\theta(\lambda,u)$.

Now consider the long exact sequence \eqref{eq:long_ex_sequ} twice, once for the identity $\id_X$, once for the bundle projection $\pi:P \to X$.
Pull-back along the map $(\id_X,\pi):(X,P) \to (X,X)$ yields the commutative diagram:
\begin{equation}\label{eq:diag_diff_triv}
\xymatrix{
H^{2k-2}(X;\Ul) \ar^{\pi^*\circ j}[r] & \widehat H^{2k-1}(P;\Z) \ar^{\ti_\pi}[r] &  \widehat H^{2k}(\pi;\Z) \ar^{\vds_\pi}[r] & \widehat H^{2k}(X;\Z) \\
H^{2k-2}(X;\Ul) \ar_j[r] \ar_{\id_X^*}[u] & \widehat H^{2k-1}(X;\Z) \ar_{\ti_{\id}}[r] \ar_{\pi^*}[u]&  \widehat H^{2k}(\id_X;\Z) \ar_{\vds_{\id}}[r] \ar[u]_{(\id_X,\pi)^*} & \widehat H^{2k}(X;\Z) \ar[u]_{\id_X^*} \\  
}
\end{equation}
Pull-back along $(\id_X,\pi)$ maps the global section $\iota_{\id}(\varrho,0) \in \widehat H^{2k}(\id_X;\Z)$ of the Cheeger-Simons character $\widehat{CW}_\theta(\lambda,u)$ to the relative character $\iota_\pi(\varrho,0) \in \widehat H^{2k}(\pi;\Z)$.
Thus commutativity of the right square yields 
$$
\vds_\pi\big(\iota_\pi(\varrho,0)\big) 
= 
\vds_{\pi}\big((\id_X,\pi))^*\iota_{\id}(\varrho,0)\big) 
=
\vds_{\id}\big(\iota_{\id}(\varrho,0)\big) 
= 
\widehat{CW}_\theta(\lambda,u).
$$
On the other hand we have $\vds(\widehat{CCS}_\theta(\lambda,u)) = \widehat{CW}_\theta(\lambda,u)$.
Thus the relative characters $\vds_\pi(\widehat{CCS}_\theta(\lambda,u))$ and $\iota_\pi(\varrho,0)$ differ by a character in the image of the homomorphism $\ti:\widehat H^{2k-1}(P;\Z) \to \widehat H^{2k}(\pi;\Z)$.
This observation yields our notion of differential $u$-trivializations:

\begin{defn}\label{def:q_hat}
Let $G$ be a Lie group with finitely many components.
Let $\pi:(P,\theta) \to X$ a principal $G$-bundle with connection. 
Let $u \in H^{2k}(BG;\Z)$ be a universal characteristic class for principal $G$-bundles and $(\lambda,u) \in K^{2k}(G;\Z)$. 
A \emph{differential $u$-trivialization} is a differential character $\widehat q \in \widehat H^{2k-1}(P;\Z)$ such that
\begin{equation}\label{eq:def_q_hat}
-\ti_\pi(\widehat q)
=
\widehat{CCS}_\theta(\lambda,u) - \iota_\pi(\varrho,0) 
\end{equation}
for some $\varrho \in \Omega^{2k-1}(X)$.
\end{defn}

To establish our notion of differential $u$-trivializations, we started from a global section $\varrho \in \Omega^{2k-1}(X)$ of the Cheeger-Simons character $\widehat{CW}_\theta(\lambda,u)$.
We show that any differential $u$-trivialization uniquely determines a global section:

\begin{lem}\label{lem:q_hat_varrho}
Let $G$ be a Lie group with finitely many components and $(\lambda,u) \in K^{2k}(G;\Z)$.
Let $\pi:(P,\theta) \to X$ be a principal $G$-bundle with connection.
Let $\widehat q \in \widehat H^{2k-1}(P;\Z)$ be a differential $u$-trivialization.
Then the differential form $\varrho \in \Omega^{2k-1}(X)$ in \eqref{eq:def_q_hat} is uniquely determined by the character $\widehat q$.
Moreover, it satisfies $\vds_{\id} (\iota_{\id}(\varrho,0))=\widehat{CW}_\theta(\lambda,u)$.
In other words,  $\iota_{\id}(\varrho,0) \in \widehat H^{2k}(\id_X;\Z)$ is the unique global section of the Cheeger-Simons character $\widehat{CW}_\theta(\lambda,u)$ with covariant derivative $\varrho$.
In particular, we have $d\varrho = CW_\theta(\lambda)$ and thus $u(P)=0$.

Conversely, any such global section determines differential $u$-trivializations, uniquely up to characters of the form $j(\pi^*w) \in \widehat H^{2k-1}(P;\Z)$ for some $w \in H^{2k-2}(X;\Ul)$.
\end{lem}

\begin{proof}
Assume that \eqref{eq:def_q_hat} holds for two differential forms $\varrho,\varrho' \in \Omega^{2k-1}(X)$.
Then we have $\iota_\pi(\varrho-\varrho',0)=0$ and hence $\cov(\iota_\pi(\varrho-\varrho',0)) = \pi^*(\varrho-\varrho')=0$.
Since pull-back of differential forms along the bundle projection $\pi:P \to X$ is injective, we conclude $\varrho = \varrho'$.

Next we show that any differential form $\varrho$ satisfying \eqref{eq:def_q_hat} determines global sections:
From the commutative diagram \eqref{eq:diag_diff_triv} and condition \eqref{eq:def_q_hat} we conclude
$$
\vds_{\id}(\iota_{\id}(\varrho,0)) 
= 
\vds_\pi(\iota_\pi(\varrho,0)) 
=
\vds_\pi(\widehat{CCS}_\theta(\lambda,u)+\ti(\widehat q)) 
= 
\widehat{CW}_\theta(\lambda,u).
$$
Thus the form $\varrho$ canonically determines a global section of $\widehat{CW}_\theta(\lambda,u)$ with covariant derivative $\varrho$.
Hence $\curv(\widehat{CW}_\theta(\lambda,u)) = CW_\theta(\lambda) = d\varrho$ and $c(\widehat{CW}_\theta(\lambda,u)) = u(P)=0$.

Conversely, let $u(P)=0$ and $\varrho \in \Omega^{2k-1}(X)$ such that $d\varrho = CW_\theta(\lambda,u)$.
Then $\iota_{\id}(\varrho,0) \in \widehat H^{2k}(\id_X;\Z)$ is the unique global section of $\widehat{CW}_\theta(\lambda,u)$ with covariant derivative $\varrho$.
Hence $\iota_\pi(\varrho,0) \in \widehat H^{2k}(\pi;\Z)$ is a section along $\pi$.
Thus we have $\vds(\widehat{CCS}_\theta(\lambda,u) - \iota_\pi(\varrho,0)) =0$.
By the long exact sequence \eqref{eq:long_ex_sequ} we find a differential character $\widehat q \in \widehat H^{2k-1}(P;\Z)$ such that $-\ti(\widehat q) = \widehat{CCS}_\theta(\lambda,u) - \iota_\pi(\varrho,0)$.
By \eqref{eq:long_ex_sequ} again it is uniquely determined up to a character of the form $j(\pi^*w)$ for some $w \in H^{2k-2}(X;\Ul)$. 
\end{proof}

We show that differential $u$-trivializations are differential refinements of $u$-trivialization classes.
Conversely, any $u$-trivialization class is the characteristic class of a differential $u$-trivialization.  
In particular, the property for a principal $G$-bundle with connection $\pi:(P,\theta) \to X$ to admit differential $u$-trivializations is a purely topological condition, namely vanishing of the characteristic class $u(P) \in H^*(X;\Z)$.

\begin{prop}\label{prop:diff_triv}
Let $G$ be a Lie group with finitely many components and $(\lambda,u) \in K^{2k}(G;\Z)$.
Let $\pi:(P,\theta) \to X$ be a principal $G$-bundle with connection.
Then the following holds: 
\begin{enumerate}
\item[(i)] The bundle $\pi:(P,\theta) \to X$ admits differential $u$-trivializations iff $u(P)=0$.
\item[(ii)] If $\widehat q \in \widehat H^{2k-1}(P;\Z)$ is a differential $u$-trivialization, then $c(\widehat q) = H^{2k-1}(P;\Z)$ is a $u$-trivialization class.
\item[(iii)] For any $u$-trivialization class $q \in H^{2k-1}(P;\Z)$, there exist differential $u$-trivializations $\widehat q \in \widehat H^{2k-1}(P;\Z)$ with $c(\widehat q) = q$.
\item[(iv)] If $\widehat q \in \widehat H^{2k-1}(P;\Z)$ is a differential $u$-trivialization, then we have:
\begin{equation}\label{eq:curv_diff_triv}
\curv(\widehat q) 
=
CS_\theta(\lambda) - \pi^*\varrho.
\end{equation}
\item[(v)] For any differential $u$-trivialization $\widehat q \in \widehat H^{2k-1}(P;\Z)$ and any $x \in X$, we have:
\begin{equation}\label{eq:iota_t_hat}
\widehat H^{2k-1}(P_x;\Z) \ni i_{P_x}^*\widehat q = \widehat T(\widehat{CW}_\theta(\lambda,u)) \in \widehat H^{2k-1}(G;\Z).  
\end{equation}
\end{enumerate}
\end{prop}

\begin{proof}
We first prove (i):
Assume $u(P)=0$.
Then the Cheeger-Simons character $\widehat{CW}_\theta(\lambda,u)$ is topologically trivial.
Choose $\varrho \in \Omega^{2k-1}(X)$ such that $\iota_{\id}(\varrho,0) = \widehat{CW}_\theta(\lambda,u)$.
By Lemma~\ref{lem:q_hat_varrho} there exist differential $u$-trivializations with differential form $\varrho$.
The converse implication follows from (ii).

Next we prove (v):
Let $\widehat q \in \widehat H^{2k-1}(P;\Z)$ be a differential character satisfying \eqref{eq:def_triv}.
We compute the pull-back to the fiber $P_x$ over any point $x \in X$.
From the commutative diagram
\begin{equation*}
\xymatrix@C=50pt{
(\{x\},P_x) \ar^{(i_x,i_{P_x})}[r] \ar_{(\id_x,\pi_x)}[d] & (X,P) \ar^{(\id_X,\pi)}[d] \\
(\{x\},\{x\}) \ar_{(i_x,i_x)}[r] & (X,X)
}
\end{equation*}
and naturality of the long exact sequence \eqref{eq:long_ex_sequ} we obtain
\begin{align*}
\ti_{\pi_x}(i_{P_x}^*\widehat q)
&\stackrel{\phantom{\eqref{eq:def_T_hat}}}{=}
(i_x,i_{P_x})^* \ti_\pi(\widehat q) \\
&\stackrel{\eqref{eq:def_q_hat}}{=}
(i_x,i_{P_x})^*\big(-\widehat{CCS}_\theta(\lambda,u) + \iota_\pi(\varrho,0)\big) \\
&\stackrel{\phantom{\eqref{eq:def_T_hat}}}{=}
-(i_x,i_{P_x})^*\widehat{CCS}_\theta(\lambda,u) + (i_x,i_{P_x})^*(\id_X,\pi)^*\iota_{\id}(\varrho,0) \\
&\stackrel{\eqref{eq:def_T_hat}}{=}
\ti_{\pi_x}(\widehat T(\widehat{CW}_\theta(\lambda,u))) + (\id_x,\pi_x)^*\underbrace{\iota_{\id} (i_x^*\varrho,0)}_{=0}.
\end{align*}
By the long exact sequence \eqref{eq:long_ex_sequ} for the bundle projection $\pi_x$ over a point $x \in X$ the map $\ti_{\pi_x}:\widehat H^{2k-1}(P_x;\Z) \to \widehat H^{2k}(\pi_x;\Z)$ is an isomorphism.
We thus conclude $i_{P_x}^*\widehat q = \widehat T(\widehat{CW}_\theta(\lambda,u))$.

Next we prove (ii):
From \eqref{eq:iota_t_hat} we conclude 
$$
i_{P_x}^*c(\widehat q) 
= 
c(i_{P_x}^*\widehat q) 
\stackrel{\eqref{eq:iota_t_hat}}{=}
c(\widehat T(\widehat{CW}_\theta(\lambda,u)))
\stackrel{\eqref{eq:cT}}{=}
T(c(\widehat{CW}_\theta(\lambda,u)))
=
T(u).
$$
Thus the characteristic class $c(\widehat q)$ of any differential $u$-trivialization $\widehat q$ is a $u$-trivialization class.
In particular, $u(P)=0$.

Now we prove (iii):
Let $q \in H^{2k-1}(P;\Z)$ be a $u$-trivialization class, in particular, $u(P)=0$.
By (i) e know that differential $u$-trivializations exist. 
We aim at constructing a differential $u$-trivialization $\widehat q$ with characteristic class $c(\widehat q)=q$.
Let $\widehat q' \in \widehat H^{2k-1}(P;\Z)$ be any differential $u$-trivialization with differential form $\varrho'$.
Put $q':= c(\widehat q')$.
Since $q$ and $q'$ are both $u$-trivialization classes, we have $q - q' = \pi^*w$ for some $w \in H^{2k-1}(X;\Z)$.
Choose a differential character $\widehat w \in \widehat H^{2k-1}(X;\Z)$ with characteristic class $c(\widehat w) = w$.
Now put $\widehat q:= \widehat q' + \pi^*\widehat w$.
Then we have $c(\widehat q)=q$.
Put $\varrho:= \varrho'-\curv(\widehat w)$.
Then we have $\ti_{\id}(\widehat w) = \iota_{\id}(-\curv(w),0)$ and hence
\begin{align*}
-\ti_\pi(\widehat q)
&=
-\ti_\pi(\widehat q') - \ti_\pi(\pi^*\widehat w) \\
&=
\widehat{CCS}_\theta(\lambda,u) - \iota_\pi(\varrho',0) - (\id_X,\pi)^*\ti_{\id}(\widehat w) \\
&=
\widehat{CCS}_\theta(\lambda,u) - \iota_\pi(\varrho' - \curv(\widehat w),0) \\
&=
\widehat{CCS}_\theta(\lambda,u) -\iota_\pi(\varrho,0). 
\end{align*}
Thus $\widehat q$ is a differential $u$-trivialization.

Finally, (iv) follows immediately from \eqref{eq:def_q_hat} and $\cov(\ti(\widehat q)) = -\curv(\widehat q)$.
\end{proof}

It is well-known that the set of all $u$-trivialization classes $u \in H^{2k-1}(P;\Z)$ is a torsor for the action of the additive group $H^{2k-1}(X;\Z)$.
We show that the analoguous statement holds for the set of differential $u$-trivializations.

\begin{prop}\label{prop:torsor}
Let $G$ be a Lie group with finitely many components and $(\lambda,u) \in K^{2k}(G;Z)$.
Let $\pi:(P,\theta) \to X$ be a principal $G$-bundle with connection.
The differential cohomology group $ \widehat H^{2k-1}(X;\Z)$ acts on the set of all differential $u$-trivializations by $(\widehat q,h) \mapsto \widehat q + \pi^*h$.

Moreover, the set of differential $u$-trivializations is a torsor for the additive group $\pi^*\widehat H^{2k-1}(X;\Z)$.
\end{prop}

\begin{proof}
Let $\widehat q$ be a differential $u$-trivialization with differential form $\varrho$ and $h \in \widehat H^{2k-1}(X;\Z)$.
As above we have $\ti_{\id}(h) = \iota_{\id}(-\curv(h),0)$ and hence 
\begin{align*}
-\ti_\pi(\widehat q + \pi^*h)
&=
\widehat{CCS}_\theta(\lambda,u) - \iota_\pi(\varrho,0) - \iota_\pi(-\curv(h),0) \\
&=
\widehat{CCS}_\theta(\lambda,u) - \iota_\pi(\varrho-\curv(h)). 
\end{align*}
Thus $\widehat q + \pi^*h$ is a differential $u$-trivialization with differential form $\varrho - \curv(h)$.
Hence the differential cohomology group $\widehat H^{2k-1}(X;\Z)$ acts on the set of differential $u$-trivializations.

The action of the group $\widehat H^{2k-1}(X;\Z)$ on the set of differential $u$-trivializations is in general not free.
The kernel of the map $\pi^*:\widehat H^{2k-1}(X;\Z) \to \widehat H^{2k-1}(P;\Z)$ is contained in the image of the map $j:H^{2k-2}(X;\Ul) \to \widehat H^{2k-1}(X;\Z)$.
By injectivity of the latter, the kernel consists of characters of the form $j(v)$, where $v$ is in the kernel of $\pi^*:H^{2k-2}(X;\Ul) \to H^{2k-2}(P;\Ul)$.  

We need to show that the action is transitive. 
Let $\widehat q$ and $\widehat q'$ be differential $u$-trivializations with differential forms $\varrho$ and $\varrho'$, respectively.
By Lemma~\ref{lem:q_hat_varrho} we have $\vds_{\id}(\iota_{\id}(\varrho,0)) = \widehat{CW}_\theta(\lambda,u) = \vds_{\id}(\iota_{\id}(\varrho',0))$.
From the long exact sequence \eqref{eq:long_ex_sequ} we thus obtain a differential character $h' \in \widehat H^{2k-1}(X;\Z)$ such that $\ti_{\id}(h') = \iota_{\id}(\varrho-\varrho',0)$.
This yields 
$$
\ti_\pi(\widehat q-\widehat q') = \iota_\pi(\varrho-\varrho',0) = \ti_\pi(\pi^*h).
$$
From the upper row in \eqref{eq:diag_diff_triv}, we conclude $\widehat q - \widehat q' = \pi^*(h' + j(v))$ for some $v \in H^{2k-2}(X;\Ul)$.
Put $h:= h' + j(u) \in \widehat H^{2k-1}(X;\Z)$.
Then we have $\widehat q - \widehat q' = \pi^*h$.
\end{proof}

\begin{rem}\label{rem:weak}
In general, the condition~\eqref{eq:iota_t_hat} is weaker than \eqref{eq:def_q_hat}:
Let $\mu \in \Omega^{2k-2}(X)$ and $f \in C^\infty(P)$ not constant along the fibers.
Put $\eta:= f \cdot \pi^*\mu$.
Then $\eta$ vanishes upon pull-back to any fiber $P_x$.
Moreover, $d\eta = df \wedge \pi^*\mu + f \cdot \pi^*d\mu$ is not the pull-back of a form on the base $X$.
Now let $\widehat q$ be any differential $u$-trivialization.
Then $\widehat q + \iota(\eta) \in \widehat H^{2k-1}(P;\Z)$ still satisfies \eqref{eq:iota_t_hat}, since 
$$
i_{P_x}^*(\widehat q + \iota(\eta)) 
= 
i_{P_x}^*\widehat q + \iota(i_{P_x}^*\eta) 
= 
\widehat T(\widehat{CW}_\theta(\lambda,u)).
$$
But $\widehat q + \iota(\eta)$ does not satisfy \eqref{eq:curv_diff_triv} since $\curv(\widehat q + \iota(\eta)) - CS_\theta(\lambda) = \pi^*\varrho + d\eta$ is not the pull-back of a form on $X$.
This implies that $\widehat q + \iota(\varrho)$ does not satisfy \eqref{eq:def_q_hat}.
\end{rem}

\begin{rem}
In general, even the two conditions \eqref{eq:curv_diff_triv} and \eqref{eq:iota_t_hat} together do not imply \eqref{eq:def_q_hat}: 
Suppose there exists a closed form $\nu \in \Omega^{2k-1}(X)$ such that $\pi^*\nu = d\eta$ for some $\eta \in \Omega^{2k-2}(P)$ and $i_{P_x}^*\eta$ is exact for any $x \in X$.
Without loss of generality we assume that $\nu$ does not have integral periods.
Let $\widehat q$ be a differential $u$-trivialization with differential form $\varrho$.
Put $h:= \widehat q + \iota(\eta)$.
Then we have 
$$
i^*_{P_x}h 
= 
i_{P_x}^*\widehat q + \iota(i_{P_x}^*\eta) 
= 
\widehat T(\widehat{CW}_\theta(\lambda,u)),
$$ 
since $i_{P_x}^*\eta$ is exact and thus $\iota(i_{P_x}^*\eta)=0$.
Thus $h$ satisfies \eqref{eq:iota_t_hat}.
Moreover, $h$ satisfies \eqref{eq:curv_diff_triv}, since
$$
\curv(h) 
= 
\curv(\widehat q) + d\eta 
= 
CS_\theta(\lambda) - \pi^*(\varrho-\nu).
$$
But we have $-\ti(h) = -\ti(\widehat q) - i_\pi(0,\eta) = \widehat{CCS}_\theta(\lambda,u) - i_\pi(\varrho,\eta)$.
Thus $h$ satisfies \eqref{eq:def_q_hat} if and only if $\iota_\pi(\varrho,\eta) = \iota_\pi(\varrho',0)$ for some $\varrho' \in \Omega^{2k-1}(X)$.
The latter condition is equivalent to $(\varrho-\varrho',\eta)$ being closed with integral periods.
By assumption, $d\eta = \pi^*\nu$ and $d\nu=0$.
Thus we necessarily have $\varrho-\varrho' = \nu$.
But by assumption $\nu$ does not have integral periods, and so neither does $(\nu,\eta) = (\varrho-\varrho',\eta)$.
\end{rem}

\subsection{Dependence upon the connection}\label{subsec:depend_conn_2}
Since the Cheeger-Chern-Simons character $\widehat{CCS}_\theta(\lambda,u)$ depends upon the connection $\theta$, so do differential $u$-trivializations:

\begin{prop}\label{prop:depend_conn_diff_triv}
Let $G$ be a Lie group with finitely many components and $(\lambda,u) \in K^{2k}(G;Z)$.
Let $\pi:P \to X$ be a principal $G$-bundle with connections $\theta_0$, $\theta_1 \in \mathcal A(P)$.
Define $\alpha(\theta_0,\theta_1;\lambda) \in \Omega^{2k-2}(P)$ as in Section~\ref{subsec:depend_conn_1}.

Then the following holds: if $\widehat q \in \widehat H^{2k-1}(P;\Z)$ is a differential $u$-trivialization on $(P,\theta_0)$ with differential form $\varrho$, 
then $\widehat q - \iota(\alpha(\theta_0,\theta_1;\lambda))$ is a differential $u$-trivialization on $(P,\theta_1)$ with differential form $\varrho + CS(\theta_0,\theta_1;\lambda)$. 
\end{prop}

\begin{proof}
By definition, we have
\begin{align*}
-\ti(\widehat q) 
&\stackrel{\eqref{eq:def_q_hat}}{=}
\widehat{CCS}_{\theta_0}(\lambda,u) - \iota_\pi(\varrho,0)  \\
&\stackrel{\eqref{eq:CCS_two_conn}}{=}
\widehat{CCS}_{\theta_1}(\lambda,u) - \iota_\pi(CS(\theta_0,\theta_1;\lambda),\alpha(\theta_0,\theta_1;\lambda)) - \iota_\pi(\varrho,0) \\
&=
\widehat{CCS}_{\theta_1}(\lambda,u) - \iota_\pi(\varrho+CS(\theta_0 ,\theta_1;\lambda),0) - \iota_\pi(0,\alpha(\theta_0,\theta_1;\lambda)) \\
&=
\widehat{CCS}_{\theta_1}(\lambda,u) - \iota_\pi(\varrho+CS(\theta_0 ,\theta_1;\lambda),0) - \ti(\iota(\alpha(\theta_0,\theta_1;\lambda))).
\end{align*}
Thus the differential character $\widehat q - \iota(\alpha(\theta_0,\theta_1;\lambda)) \in \widehat H^{2k-1}(P;\Z)$ and the differential form $\varrho + CS(\theta_0,\theta_1;\lambda) \in \Omega^{2k-1}(X)$ together satisfy condition~\eqref{eq:def_q_hat} on $(P,\theta_1)$. 
\end{proof}

%%%%%%%%%%%%%%%%%%%%%%%%%%%%%%%%%%%%%%%%%%%%%%%%%%%%%%%%%%%%%%%%%%%%%%%%%
\section{Differential String classes and (higher) Chern-Simons theories}\label{sec:diff_string}
%%%%%%%%%%%%%%%%%%%%%%%%%%%%%%%%%%%%%%%%%%%%%%%%%%%%%%%%%%%%%%%%%%%%%%%%%
In this section, we establish our notion of differential String classes on a principal $\Spin_n$-bundle with connection $\pi:(P,\theta) \to X$, where $n \geq 3$.
We obtain this notion by specializing the notion of differential $u$-trivialization to the case $u = \frac{1}{2} p_1 \in H^4(B\Spin_n;\Z)$.
This notion of corresponds to \emph{geometric String structures} from \cite{W13}, i.e.~trivializations of the Chern-Simons bundle $2$-gerbe, together with a compatible connection.

There are four essentially equivalent ways to define differential String classes:
the first due to Waldorf views differential String classes as stable isomorphism classes of geometric String structures \cite{W13,W14}.
The second is implicit in the work of Redden \cite{R11}: on a compact Riemannian manifold $(X,g)$ there is a canonical String class $q \in \widehat H^3(P;\Z)$, defined via Hodge theory and adiabiatic limits.
It has a canonical refinement to a differential String class $\widehat q \in \widehat H^3(P;\Z)$. 
Thirdly, one may define differential String classes analogously to String classes as characters on $P$ which restrict on any fiber to the so-called \emph{basic} differential character in $\widehat H^3(\Spin_n;\Z)$, the stable isomorphism class of the so-called \emph{basic gerbe} on $\Spin_n$.
Finally, our own notion regards differential String classes as differential $\frac{1}{2}p_1$-trivializations in the sense of Section~\ref{sec:diff_triv}.
None of the first three ways to regard differential String classes directly generalizes to higher order structures.
In contrast, our notion does: differential $u$-trivializations can be defined for any universal characteristic class $u \in \widehat H^*(BG;\Z)$ on any Lie group $G$ with finitely many components.

In this Section, we first recall the notion of String structures and String classes.
Next we introduce our notion of differential String classes by specializing differential $u$-trivializations to the case $u = \frac{1}{2} p_1 \in H^4(B\Spin_n;\Z)$.
Then we prove the equivalence of the four above mentioned notions of differential String classes.
Finally, we discuss (higher) Chern-Simons theories arising from differential String classes (and more generally, differential $u$-trivializations) and transgression to loop space.

\subsection{String structures and String classes}
The group $\String_n$ is by definition a $3$-connected cover of $\Spin_n$.
It is defined only up to homotopy.
As is well known, the homotopy type $\String_n$ cannot be represented a finite dimensional Lie group since any such group has non-vanishing $\pi_3$.
There exist several models of $\String_n$, either as a topological group \cite{S96,ST04}, as a Lie $2$-group1 \cite{BCSS07,H08,S11} or as an infinite dimensional Fr\'echet Lie group \cite{NSW13}.
In the latter case a String structure (in the Lie theoretic sense) is defined as a lift of the structure group of $\pi:P \to X$ from $\Spin_n$ to $\String_n$.

$\String_n$ is defined as the homotopy fiber of a classifying map $\lambda:B\Spin_n \to K(\Z;4)$ for the generator $\frac{1}{2} p_1 \in H^4(B\Spin_n;\Z) \cong H^3(\Spin_n;\Z) \cong \pi_3(\Spin_n) \cong \Z$. 
A String structure (in the homotopy theoretic sense) on a principal $\Spin_n$-bundle $\pi:P\to X$ is a homotopy class of lifts $\widetilde f$ of classifying maps $f$ for the bundle $\pi:P \to X$:
\begin{equation*}
\xymatrix{
&& B\String_n \ar[d] && \\
X \ar@{.>}[urr]^{\widetilde f} \ar_f[rr] && B\Spin_n \ar[rr]^{\lambda} && K(\Z;4). 
}
\end{equation*}
Isomorphism classes of String structures in the Lie theoretic sense correspond to String structures in the homotopy theoretic sense.   
Clearly, a principal $\Spin_n$-bundle $\pi:P \to X$ admits a String structure (in either sense) if and only if $\frac{1}{2} p_1(P)=0$.
Since $H^j(\Spin(n);\Z)=0$ for $j= 1,2$, by \cite{R11} there is a 1-1 correspondence between isomorphism classes of String structures on $P$ and $\frac{1}{2} p_1$-trivialization classes $q \in H^3(P;\Z)$.
These classes are called \emph{String classes}.

\subsection{Differential String classes}
We derive our notion of differential String classes by specializing the concept of differential trivializations of universal characteristic classes of principal $G$-bundles from Section~\ref{sec:diff_triv}.
For $G = \Spin_n$, $n \geq 3$, the Chern-Weil construction yields an isomorphism\footnote{Here $I^k_0(G)$ denotes the space of invariant polynomials of degree $k$, the Chern-Weil forms of which have integral periods.} $I^2_0(\Spin_n) \to H^4_{dR}(B\Spin_n;\Z) \cong H^4(B\Spin_n;\Z) \cong \Z$.
We thus write the elements of $K^4(\Spin_n;\Z)$ simply as $\lambda$ or $u$ instead of pairs $(\lambda,u)$. 

Let $(P,\theta) \to X$ be a principal $\Spin_n$-bundle with connection.
The invariant polynomial $\lambda \in I^2_0(\Spin_n) \cong H^4(B\Spin_n;\Z)$ yields the Chern-Weil form $CW_\theta(\lambda)\in \Omega^4_0(X)$ and the Cheeger-Simons character $\widehat{CW}_\theta(\lambda) \in \widehat H^4(X;Z)$ with curvature $\curv(\widehat{CW}_\theta(\lambda)) = CW_\theta(\lambda)$ and characteristic class $c(\widehat{CW}_\theta(\lambda)) = u$.
Moreover, we have the Chern-Simons form\footnote{For $\lambda = \frac{1}{2}p_1$ the Chern-Simons form $CS_\theta(\frac{1}{2}p_1)$ is the usual Chern-Simons $3$-form for $\Spin_n$.} $CS_\theta(\lambda) \in \Omega^3(P)$ and the Cheeger-Chern-Simons character $\widehat{CCS}_\theta(\lambda) \in \widehat H^4(\pi;\Z)$ with covariant derivative $\cov(\widehat{CCS}_\theta(\lambda)) = CS_\theta(\lambda)$.
Since $H^4(B\Spin_n;\Z) \cong I^2_0(\Spin_n) \cong K^4(\Spin_n;\Z) \cong \Z$ with generator $\frac{1}{2}p_1$, we may write $\lambda = \ell \cdot \frac{1}{2} p_1$ for some $\ell \in \Z$.
As customary in the physics literature, we call $\ell$ the \emph{level} of $\lambda$.

\begin{defn}\label{def:diff_string}
Let $\pi:(P,\theta) \to X$ be a principal $\Spin_n$-bundle with connection, where $n \geq 3$.
A \emph{differential String class} on $(P,\theta)$ is a differential $\frac{1}{2} p_1$-trivialization, i.e.~a differential character $\widehat q \in \widehat H^3(P;\Z)$ such that
\begin{equation}\label{eq:def_q_hat_Spin}
-\ti_\pi(\widehat q)
=
\widehat{CCS}_\theta(\frac{1}{2} p_1) - \iota_\pi(\varrho,0) 
\end{equation} 
for some $\varrho \in \Omega^3(X)$.
A \emph{differential String class at level $\ell$} is a differential $\ell \cdot \frac{1}{2} p_1$-trivialization, i.e.~a differential character $\widehat q \in \widehat H^3(P;\Z)$ such that
\begin{equation}\label{eq:def_q_hat_Spin_level}
-\ti_\pi(\widehat q)
=
\widehat{CCS}_\theta(\ell \cdot \frac{1}{2} p_1) - \iota_\pi(\varrho,0) 
\end{equation} 
for some $\varrho \in \Omega^3(X)$.
\end{defn}

\begin{prop}
Let $\widehat q \in \widehat H^3(P;\Z)$ be a differential String class on $\pi:(P,\theta) \to X$ with differential form $\varrho \in \Omega^3(X)$.
Let $x \in X$ be an arbitrary point.
Then we have
\begin{align}
\curv(\widehat q) 
&= 
CS_\theta(\frac{1}{2}p_1) - \pi^*\varrho \label{eq:curv_String} \\
-\ti_\pi(c(\widehat q)) 
&= 
\frac{1}{2} \widetilde{p_1}(P) \label{eq:c_String} \\
i_{P_x}^*\widehat q 
&=
\widehat T(\widehat{CW}_\Theta(\frac{1}{2} p_1)). \label{eq:transg_String}
\end{align}
Here $\frac{1}{2}\widetilde p_1 \in H^4(\pi_{E\Spin_n};\Z)$ denotes the mapping cone class corresponding to the universal cohomology class $\frac{1}{2} p_1 \in H^4(B\Spin_n;\Z)$ under the isomorphism $H^*(\pi_{E\Spin_n};\Z) \to H^*(B\Spin_n;\Z)$. 
\end{prop}

\begin{proof}
By \eqref{eq:def_q_hat_Spin}, we have $\curv(\widehat q) = -\cov(\ti_\pi(\widehat q)) = CS_\theta(\frac{1}{2}p_1) - \pi^*\varrho$.
Similarly, we have $-\ti_\pi(\widehat q) = c(\widehat{CCS}_\theta(\frac{1}{2}p_1)) = \frac{1}{2} \widetilde{p_1}(P)$.
Finally, by restiction to the fiber over $x$, we obtain $-\ti_{\pi_x}(i_{P_x}^*\widehat q) = (i_x,i_{P_x})^*\widehat{CCS}_\theta(\frac{1}{2} p_1)$.
By Definition~\ref{def:T_hat}, this is equivalent to $i_{P_x}^*\widehat q = \widehat T(\widehat{CW}_\Theta(\frac{1}{2}p_1))$.
\end{proof}

From Proposition \ref{prop:diff_triv}, we conclude:

\begin{cor}\label{cor:String_prop}
Let $n \geq 3$.
A principal $\Spin_n$-bundle $\pi:(P,\theta) \to X$ with connection admits differential String classes if and only if it is String, i.e.~$\frac{1}{2}p_1(P)=0$.
The characteristic class $c(\widehat q) \in H^3(P;\Z)$ of a differential String class $\widehat q \in \widehat H^3(P;\Z)$ is a String class.
The differential form $\varrho$ in the notion of differential String classes is uniquely determined by the character $\widehat q$.
It satisfies $d\varrho = CW_\theta(\frac{1}{2}p_1)$.

Conversely, if $P$ is String, then for any String class $q \in H^3(P;\Z)$ and any differential form $\varrho \in \Omega^3(X)$ with $d\varrho = CW_\theta(\frac{1}{2}p_1)$, there exist differential String classes $\widehat q \in \widehat H^3(P;\Z)$ with differential form $\varrho$ and $c(\widehat q) = q$.
\end{cor}

The differential cohomology group $\widehat H^3(X;\Z)$ acts on differential String classes by $(\widehat q,h) \mapsto \widehat q + \pi^*h$.
In the general case of differential $u$-trivializations, this action of the differential cohomology of the base is not free. 
But in the case of differential String classes it is, and we have:

\begin{cor}\label{cor:String_torsor}
Let $\pi:(P,\theta) \to X$ be a principal $\Spin_n$-bundle with connection and $n \geq 3$.
Then the set of differential String classes $\widehat q \in \widehat H^3(P;\Z)$ is a torsor for the additive group $\widehat H^3(X;\Z)$. 
\end{cor}

\begin{proof}
Since $\widetilde H^i(\Spin_n;\Z) = \{0\}$ for $i\leq 2$, the Leray-Serre sequence yields the following exact sequence \cite[Prop.~2.5]{R11}:
\begin{equation}\label{eq:Serre_sequ}
\xymatrix{
0 \ar[r] & H^3(X;\Z) \ar^{\pi^*}[r] & H^3(P;\Z) \ar[r]^(0.45){i_x^*} & H^3(\Spin_n;\Z) \ar^(0.55){T^{-1}}[r] & H^4(X;\Z). 
}
\end{equation}
From the long exact sequence \eqref{eq:short_ex_sequ} we conclude that the pull-back $\pi^*:\widehat H^3(X;\Z) \to \widehat H^3(P;\Z)$ is injective.
Thus the action of the additive group $\widehat H^3(X;\Z)$ on differential String classes is free.
Hence by Proposition~\ref{prop:diff_triv}, the set of differential String classes on $(P,\theta)$ is a torsor for the differential cohomology group $\widehat H^3(X;\Z)$.
\end{proof}

As discussed in Section~\ref{subsec:depend_conn_2}, differential $u$-trivializations depend upon the choice of connection.
Obviously, we find the same dependence of differential String classes and their differential forms on the connection:

\begin{cor}
Let $\theta_0,\theta_1 \in \mathcal A(P)$ be connections on the principal $\Spin_n$-bundle $\pi:P \to X$.
Define $CS(\theta_0,\theta_1;\frac{1}{2}p_1) \in \Omega^3(X)$ and $\alpha(\theta_0,\theta_1;\frac{1}{2}p_1) \in \Omega^2(P)$ be as in Section~\ref{subsec:depend_conn_1}.

Then the following holds: if $\widehat q $ is a differential String class on $(P,\theta_0)$ with differential form $\varrho$, then $\widehat q - \iota(\alpha(\theta_0,\theta_1;\frac{1}{2}p_1))$ is a differential String class on $(P,\theta_1)$ with differential form $\varrho + CS(\theta_0,\theta_1;\frac{1}{2}p_1)$. 
\end{cor}

The transgression map $T:H^4(B\Spin_n;\Z) \to H^3(\Spin_n;\Z)$ is an isomorphism.
The restriction of the Chern-Simons form $CS_\theta(\frac{1}{2}p_1)$ to any fiber $P_x$ represents the class $T(\frac{1}{2}p_1) \in H^3(\Spin_n;\Z)$ in de Rham cohomology.
The pull-back of the Chern-Simons form to any fiber $i_{P_x}^*CS_\theta(\frac{1}{2}p_1) \in \Omega^3(\Spin_n)$ can be expressed purely in terms of the Maurer-Cartan form of $\Spin_n$, see \cite{CS74}.
In particular, it does not dependend upon the connection $\theta$ on $P$.

Since $b_2(\Spin_n)=0$, by Remark~\ref{rem:T_hat}, $\widehat T(\widehat{CW}_\theta(\frac{1}{2}p_1))$ is the unique differential character in $\widehat H^3(\Spin_n;\Z)$ with curvature $i_{P_x}^*CS_\theta(\frac{1}{2}p_1)$ and characteristic class $T(\frac{1}{2}p_1)$.
We call $T(\widehat{CW}_\theta(\frac{1}{2}p_1))$ the \emph{basic} $3$-character on $\Spin_n$, since it coincides with the stable isomorphism class of the so-called \emph{basic} gerbe on $\Spin_n$, see \cite{M03}.

Next we obtain equivalent characterizations for a character $\widehat q \in \widehat H^3(P;\Z)$ to be a differential String class:
\begin{prop}\label{prop:diff_string}
Let $\pi:(P,\theta) \to X$ be a principal $\Spin_n$-bundle with connection.
Let $\widehat q \in \widehat H^3(P;\Z)$ and $\varrho \in \Omega^3(X)$.
Then the following conditions are equivalent:
\begin{itemize}
\item[(i)] The characteristic class $c(\widehat q) \in H^3(P;\Z)$ is a String class and the curvature satisfies $\curv(\widehat q) = CS_\theta(\frac{1}{2} p_1) - \pi^*\varrho$.
\item[(ii)] The curvature satisfies $\curv(\widehat q) = CS(\frac{1}{2} p_1) - \pi^*\varrho$ and for any $x \in X$, we have $i_{P_x}^*\widehat q = \widehat T(\widehat{CW}_\Theta(\frac{1}{2} p_1))$.
\item[(iii)] The character $\widehat q$ is a differential String class with differential form $\varrho$.
\end{itemize}
\end{prop}

\begin{proof}
Before we prove the claim let us recall that $b_1(\Spin_n) = b_2(\Spin_n) =0$ implies that degree $3$ characters on $\Spin_n$ are uniquely determined by their characteristic class and curvature.
Moreover, the bundle projection provides an isomorphism $\pi^*:H^2(X;\R) \to H^2(P;\R)$.
Thus differential characters in $\widehat H^3(P;\Z)$ are uniquely determined by their curvature and characteristic class up to characters of the form $\pi^*\iota(\nu)$ for some closed form $\nu \in \Omega^2(X)$.

We first prove the implication (i) $\implies$ (ii):
By the curvature condition, we have $\curv(i_{P_x}^*\widehat q) = i_{P_x}^*CS_\theta(\frac{1}{2}p_1) = \curv(\widehat T(\widehat{CW}_\Theta(\frac{1}{2}p_1)))$.
By assumption, $c(\widehat q)$ is a String class, thus $c(i_{P_x}^*(\widehat q)) = T(\frac{1}{2}p_1) = \curv(\widehat T(\widehat{CW}_\Theta(\frac{1}{2}p_1)))$.
Hence $i_{P_x}^*\widehat q = \widehat T(\widehat{CW}_\Theta(\frac{1}{2}p_1))$.

Now we prove the implication (ii) $\implies$ (iii):
Since $i_{P_x}^*\widehat q = \widehat T(\widehat{CW}_\theta(\frac{1}{2}p_1))$, the characteristic class $c(\widehat q)$ is a String class.
By Lemma~\ref{lem:q_hat_varrho} there exists a differential String class $\widehat q' \in \widehat H^3(P;\Z)$ with characteristic class $c(\widehat q)$.
Let $\varrho' \in \Omega^3(X)$ be the differential form of the differential String class $\widehat q'$.
By \eqref{eq:curv_String} for $\widehat q'$ and the curvature condition for $\widehat q$, we have $\curv(\widehat q) - \curv(\widehat q') = \pi^*(\varrho-\varrho')$.
Since the pull-back $\pi^*:\Omega^*(X) \to \Omega^*(P)$ is injective, the form $\varrho-\varrho'$ is closed.
On the other hand, $\curv(\widehat q) - \curv(\widehat q')$ is exact, since both forms represent the cohomology class $c(\widehat q)_\R$. 
By the exact sequence \eqref{eq:Serre_sequ} the pull-back $\pi^*:H^3(X;\Z) \to H^3(P;\Z)$ is injective.
Thus we find a form $\eta \in \Omega^2(X)$ such that $\varrho-\varrho' = d\eta$.

Put $\widehat q'':= \widehat q' + \iota(\pi^*\eta)$.
Then we have $\curv(\widehat q'') = \curv(\widehat q') + d\pi^*\eta = \curv(\widehat q)$ and $c(\widehat q'')=c(\widehat q')=c(\widehat q)$.
Thus there exists a closed form $\nu \in \Omega^2(X)$ such that $\widehat q=\widehat q'' + \iota(\pi^*\nu) = \widehat q' + \iota(\pi^*(\eta+\nu))$.
Since $\widehat q'$ is a differential String class, we have:
\allowdisplaybreaks{
\begin{align*}
-\ti_\pi(\widehat q) 
&= 
-\ti_\pi(\widehat q' + \iota(\pi^*(\eta+\nu))) \\
&= 
\widehat{CCS}_\theta(\frac{1}{2}p_1) - \iota_\pi(\varrho',0) + \iota_\pi(0,\pi^*(\eta + \nu)) \\
&=
\widehat{CCS}_\theta(\frac{1}{2}p_1) - \iota_\pi(\varrho'-d\eta,0) + \iota_\pi(d_\pi(\eta+\nu,0)) \\
&=
\widehat{CCS}_\theta(\frac{1}{2}p_1) - \iota_\pi(\varrho,0) . 
\end{align*}
}
Thus $\widehat q$ is a differential String class with differential form $\varrho$.

Finally, the implication (iii) $\implies$ (i) follows from \eqref{eq:curv_String} and Corollary~\ref{cor:String_prop}.
\end{proof}

\subsection{Related concepts}
In this section we relate differential String classes as in Definition~\ref{def:diff_string} to canonical trivializations of the class $\frac{1}{2}p_1 \in H^4(X;\Z)$ \cite{R11} and to geometric String structures \cite{W13}.
The former are certain $3$-forms on a compact Riemannian manifold $(X,g)$, while the latter are certain bundle gerbes with connection on $P$.
We show that the stable isomorphism class of a geometric String structure is a differential String class and vice versa, any differential String class is the stable isomorphism class of a geometric String structure.

Another approach to differential (twisted) String structures has appeared in \cite{SSS08}, where the authors describe those structures in terms of $\infty$-connections on $G$-principal $\infty$-bundles.
Their approach also applies to higher lifts in the Whitehead tower of $B\OO_n$, which account for higher geometric structures, described in homotopy theoretic terms.
The authors also discuss how these higher geometric structures are related to anomaly cancellation in String theory.
As the lift in the Whitehead tower from $B\Spin_n$ to $B\String_n$ accounts for anomaly cancellation of supersymmetric Strings on the target space $X$, the next lift is related to anomaly cancellation for fivebranes on $X$. 
That's why the geometric structure corresponding to such a lift is called a \emph{fivebrane structure} in \cite{SSS08}.

\subsubsection{Canonical differential refinements}
Let $(X,g)$ be a compact Riemannian manifold and $q \in H^3(P;\Z)$ a fixed String class.
Let $\varrho_0$ be the unique co-exact $3$-form satisfying $d\varrho_0 = CW_\theta(\frac{1}{2}p_1)$.
Denote by $\mathcal H^3_g(X)$ the space of harmonic $3$-forms on $X$ with respect to the metric $g$.
By \cite[Thm.~3.7]{R11}, there is a unique harmonic form with integral periods $\varrho \in \mathcal H^3_g(X;\Z)$ such that 
\begin{equation}\label{eq:CS_q}
H^3_{dR}(P;\Z) \ni \big[CS_\theta(\frac{1}{2}p_1) - \pi^*(\varrho_0+\varrho)\big]_{dR} = q_\R \in H^3(P;\Z)_\R. 
\end{equation}
Thus the form $CS_\theta(\frac{1}{2}p_1) - \pi^*(\varrho_0+\varrho)$ is a canonical representative of the String class $q$ in de Rham cohomology.
In particular, it has integral periods.
Taking differential characters on $P$ with characteristic class a String class and curvature the associated Redden form $CS_\theta(\frac{1}{2}p_1) - \pi^*(\varrho_0+\varrho)$ yields another equivalent characterization of differential String classes:

\begin{prop}
Let $(X,g)$ be a compact Riemannian manifold and $\pi:(P,\theta) \to X$ a principal $\Spin_n$-bundle with connection, for $n \geq 3$.
Let $q \in H^3(P;\Z)$ be a fixed String class.
Let $\varrho_0$ be the unique co-exact $3$-form satisfying $d\varrho_0 = CW_\theta(\frac{1}{2}p_1)$.
Let $\varrho \in \mathcal H^3_g(X;\Z)$ such that $CS_\theta(\frac{1}{2}p_1) - \pi^*(\varrho_0+\varrho)$ represents $q_\R$.
Then for a character $\widehat q \in \widehat H^3(P;\Z)$, the following are equivalent:
\begin{enumerate}
\item[(i)] The character $\widehat q$ has characteristic class $c(\widehat q)=q$ and curvature $\curv(\widehat q) = CS_\theta(\lambda) - \pi^*(\varrho_0+\varrho)$.
\item[(ii)] The character $\widehat q$ is a differential String class with differential form $\varrho_0+\varrho$ and $c(\widehat q)=q$. 
\end{enumerate}
\end{prop}

\begin{proof}
The implication (i) $\implies$ (ii) follows from Proposition~\ref{prop:diff_string}.
The implication (iI) $\implies$ (i) follows from ~\eqref{eq:curv_String}.
\end{proof}

Since the $3$-form $\varrho_0$ is a canonical trivialization of $\frac{1}{2}p_1(P) \in H^4(X;\Z)$, one might consider any differential character $\widehat q \in \widehat H^3(P;\Z)$ with characteristic class a String class $q$ and curvature the associated Redden form $CS_\theta(\frac{1}{2}p_1) - \pi^*(\varrho_0+\varrho)$ as a canonical differential refinement of the String class $q$.
In this sense our notion of differential String classes recovers the notion of canonical differential refinements of String classes that is implicit in \cite{R11}. 

\subsubsection{Geometric String structures}\label{subsubsec:geom_string}
A geometric String structure in the sense of \cite{W13} is a certain bundle gerbe with connection on $P$.
More precisely, it is a trivialization of the so-called Chern-Simons bundle $2$-gerbe $\mathcal{C\!S}$ with compatible connection.
In order to relate differential String classes in the sense of Definition~\ref{def:diff_string}, we first relate Cheeger-Chern-Simons characters to the Chern-Simons bundle $2$-gerbe:

\begin{prop}
Let $\pi:(P;\theta) \to X$ be a principal $\Spin_n$-bundle with connection, and $n \geq 3$.
Let $\mathcal{C\!S}$ be the bundle $2$-gerbe on $X$, defined over the submersion $\pi:P \to X$, as in \cite{CJMSW05,W13}.
Let $h_{\mathcal{C\!S}} \in \widehat H^3(\pi;\Z)$ be the relative character associated with $\mathcal{C\!S}$, as in Appendix~\ref{app:bundle_gerbes}.
Then we have $h_\mathcal{C\!S} = \widehat{CCS}_\theta(\frac{1}{2}p_1)$. 
\end{prop}

\begin{proof}
Both the Chern-Simons bundle $2$-gerbe $\mathcal{C\!S}$ and the Cheeger-Simons and Cheeger-Chern-Simons characters $\widehat{CW}_\theta(\frac{1}{2}p_1)$, $\widehat{CCS}_\theta(\frac{1}{2}p_1)$ are natural with respect to connection preserving bundle maps.
Thus it suffices to compare them on a universal principal $\Spin_n$-bundle $\pi_{E\Spin_n}:E\Spin_n \to B\Spin_n$ with fixed universal connection $\Theta$.
The curvature of the character $h_{\mathcal{C\!S}}$ equals the $4$-form curvature $H \in \Omega^4(X)$ of the bundle $2$-gerbe $\mathcal{C\!S}$, which in turn equals the Chern-Weil-form $CW_\Theta(\frac{1}{2}p_1) = \curv(\widehat{CCS}_\Theta(\frac{1}{2}p_1))$.
Similarly, the covariant derivative of the character $h_{\mathcal{C\!S}}$ equals the curving $B \in \Omega^3(P)$ of the bundle $2$-gerbe $\mathcal{C\!S}$, which equals the Chern-Simons-form $CS_\Theta(\frac{1}{2}p_1) = \cov(\widehat{CCS}_\Theta(\frac{1}{2}p_1))$.
The characteristic class of $h_{\mathcal{C\!S}}$ maps under the isomorphism $H^4(\pi_{E\Spin_n};\Z) \to H^4(B\Spin_n;\Z)$ to the characteristic class of $\mathcal{C\!S}$, which equals $\frac{1}{2}p_1 \in H^4(B\Spin_n;\Z)$. 
Thus the relative characters $h_\mathcal{C\!S} , \widehat{CCS}_\Theta(\frac{1}{2}p_1) \in \widehat H^4(\pi_{E\Spin_n};\Z)$ have the same curvature, covariant derivative and characteristic class.
Since $H^3(\pi_{E\Spin_n};\R) \cong H^3(B\Spin_n;\R) = \{0\}$, we conclude $h_\mathcal{C\!S} = \widehat{CCS}_\theta(\frac{1}{2}p_1)$.
\end{proof}

A geometric String structure on $(P,\theta)$ is by definition a trivialization of the Cheeger-Simons bundle $2$-gerbe $\mathcal{C\!S}$ together with a compatible connection.
In particular, it is a bundle gerbe with connection $\mathcal S$ on $P$.
Its Dixmier-Douady class is a String class.
Moreover, there is a uniquely determined $3$-form $\varrho \in \Omega^3(X)$ such that $\curv(\mathcal G) = CS_\theta(\frac{1}{2} p_1) - \pi^*\varrho$.
Thus we obtain a 1-1 correspondence between differential String classes and isomorphism classes of geometric String structures.
In particular, the notion of differential String classes from Definition~\ref{def:diff_string} is compatible with the one in \cite{W14}.

\begin{prop}
Let $\pi:(P;\theta) \to X$ be a principal $\Spin_n$-bundle with connection, and $n \geq 3$.
The the following holds:
\begin{enumerate}
\item[(i)] If  $\mathcal S$ is a geometric String structure with differential form $\varrho$, then its isomorphism class $[\mathcal S] \in \widehat H^3(P;\Z)$ is a differential String class. 
Their differential forms coincide.
\item[(ii)] For any differential String class $\widehat q \in \widehat H^3(P;\Z)$ with differential form $\varrho$, there exists a geometric String structure $\mathcal S$ such that $[\mathcal S] = \widehat q$.  
\end{enumerate}
In other words, the set of isomorphism classes of geometric String structures on $(P,\theta)$ coincides with the set of differential String classes on $(P,\theta)$.
\end{prop}

\begin{proof}
The isomorphism class of a geometric String structure on $(P,\theta)$ is a differential cohomology class $[\mathcal S] \in \widehat H^3(P;\Z)$ with characteristic class a String class and curvature form $\curv([\mathcal S]) = CS_\theta(\frac{1}{2}p_1) - \pi^*\varrho$.
Thus by Proposition~\ref{prop:diff_string}, $[\mathcal S]$ is a differential String class with differential form $\varrho$.

By \cite[Cor.~2.11]{W13}, the set of isomorphism classes of geometric String structures on $(P,\theta)$ is a torsor for the differential cohomology group $\widehat H^3(X;\Z)$.
By Corollary~\ref{cor:String_torsor}, the same holds for the space of geometric String structures.
Since these torsors have a nonempty intersection, they coincide.
\end{proof}

\subsection{(Higher) Chern-Simons theories}
In this section, we briefly explain, how differential String classes at level $\ell$ lead to trivializations of a version of extended Chern-Simons theory.\footnote{Note that we do not claim to have a rigorous notion of extended field theory. In the literature, fully extended field theories are expected to be functors from geometrically defined higher categories to algebraically defined higher categories, like topological field theories in the sense of Atiyah are functors from the cobordism category to the category of vector spaces. We only introduce a notion of how to evaluate differential characters on lower dimensional closed manifolds. In a fully extended field theory one would also expect to have evluations on manifolds with cornes, which e.g.~appear as boundaries of higher degree.}
We describe this Chern-Simons theory by its evaluations on closed oriented manifolds of dimension $0,1,2,3$.
These evaluations are constructed by transgression of the Cheeger-Simons character $\widehat{CW}_\theta(\lambda)$.
The trivializations are given by transgression of the differential String class and its differential form.\footnote{There are several ways to construct transgression of differential cohomology to loop spaces and higher mapping spaces. For a nice geometric method, see \cite{BB13}.}

As above, let $\lambda = \ell \cdot \frac{1}{2}p_1 \in H^4(B\Spin_n;\Z) \cong K^4(\Spin_n;\Z) \cong \Z$.
The Cheeger-Simons character $\widehat{CW}_\theta(\lambda) \in \widehat H^4(X;\Z)$  defines a version of Chern-Simons theory at level $\ell$ for the group $\Spin_n$ with target space $X$, extended down to points:
For a point $*$, the mapping space $\mathcal C^\infty(*,X)$ is canonically identified with $X$ itself.
The evaluation of $\widehat{CW}_\theta(\lambda)$ on a point $*$ is the character $\widehat{CW}_\theta(\lambda)$ itself.
Its classical action is given by the $\Ul$-valued holonomy around closed oriented $3$-manifolds mapped to $X$. 
The evaluation of $\widehat{CW}_\theta(\lambda)$ on the closed oriented $1$-manifold $S^1$ is the transgression $\tau(\widehat{CW}_\theta(\lambda)) \in \widehat H^3(C^\infty(S,X);\Z)$ to the free loop space $\LL X$.
Its classical action is the corresponding surface holonomy on the free loop space $\LL X$. 
Similarly, evaluation of $\widehat{CW}_\theta(\lambda)$ on a closed oriented surface $\Sigma^2$ is the transgression $\tau_{\Sigma}(\widehat{CW}_\theta(\lambda)) \in \widehat H^2(C^\infty(\Sigma^2,X);\Z)$ to the space of smooth maps $f:\Sigma^2 \to X$.
It yields an isomorphism class of hermitean line bundles with connection over the mapping space $C^\infty(\Sigma^2,X)$, and the classical action is the holonomy of this bundle around oriented loops in $C^\infty(\Sigma^2,X)$.
The evaluation of $\widehat{CW}_\theta(\lambda)$ on a closed oriented $3$-manifold $M^3$ is the transgression $\tau_{M}(\widehat{CW}_\theta(\lambda)) \in \widehat H^1(C^\infty(M^3,X);\Z) \cong C^\infty(C^\infty(M^3,X),\Ul)$ to the space of smooth maps $f: M^3\to X$.
Its classical action on a point $f \in C^\infty(M^3,X)$ coincides with the evaluation of the three manifold holonomy $\mathrm{hol}_{M^3}(\widehat{CW}_\theta(\lambda)):C^\infty(M^3,X) \to \Ul$. 

As we observed in Section~\ref{sec:Chern-Simons}, replacing the Cheeger-Simons character $\widehat{CW}_\theta(\lambda)$ by the associated Cheeger-Chern-Simons character $\widehat{CCS}_\theta(\lambda)$ allows us to extend the Chern-Simons action to oriented manifolds with boundary, together with sections of the bundle $\pi:P \to X$ along the boundary.
Similarly, we obtain the extended Chern-Simons theory by transgression to loop spaces and higher mapping spaces.\footnote{For transgression of relative differential cohomology, see \cite{BB13}.}

By \cite{ST04}, we expect a notion of geometric String structure on $(P,\theta)$ to provide a trivialization of the corresponding extended Chern-Simons theory.
In the case of the Chern-Simons theory defined by the characters $\widehat{CW}_\theta(\lambda)$, $\widehat{CCS}_\theta(\lambda)$ and its transgressions, a trivialization is given by a differential String class and its differential form, together with their transgressions:
Any differential String class $\widehat q$ on $(P,\theta)$ with differential form $\varrho \in \Omega^3(X)$ provides a real lift of the $3$-manifold holonomy of $\widehat{CW}_\theta(\frac{1}{2}p_1))$.
By Lemma~\ref{lem:q_hat_varrho}, we have:
$$
\mathrm{hol}_M(\widehat{CW}_\theta(\frac{1}{2}p_1)))(f)
:=
\big(f^*\mathrm{hol}_M(\widehat{CW}_\theta(\frac{1}{2}p_1)))\big)[M]
=
\exp \Big( 2\pi i \int_M f^*\varrho \Big).
$$
Since transgression for absolute and relative differential cohomology commutes with the exact sequence \eqref{eq:long_ex_sequ}, transgression of the pair $(\widehat q,\varrho)$ yields differential trivializations of the transgressions of $\widehat{CW}_\theta(\frac{1}{2}p_1)$.
Thus the pair $(\widehat q,\varrho)$ consisting of a differential String class and its associated differential form provides us with a trivialization of the Chern-Simons theory of $\widehat{CW}_\theta(\frac{1}{2}p_1)$, extended down to points. 
Thus for an oriented closed surface $\Sigma^2$ the transgressed differential String class $\tau_\Sigma(\widehat q) \in \widehat H^2(C^\infty(\Sigma,X);\Z)$ satisfies:
$$
-\ti_\pi(\tau_\Sigma(\widehat q))
=
\tau_{\Sigma}(\widehat{CCS}_\theta(\frac{1}{2}p_1)) - \iota_{\pi}(\tau_\Sigma(\varrho),0).
$$
In particular, $\tau_\Sigma(\varrho) \in \Omega^1(C^\infty(\Sigma,X))$ defines an isomorphism class of sections of the line bundle over $C^\infty(\Sigma,X)$ associated with the transgressed Cheeger-Simons character $\tau_\Sigma(\widehat{CW}_\theta(\frac{1}{2}p_1))$.
Likewise, transgression of $\widehat q$ along $S^1$ satisfies
$$
-\ti_\pi(\tau_S(\widehat q))
=
\tau_S(\widehat{CCS}_\theta(\frac{1}{2}p_1)) - \iota_{\pi}(\tau_S(\varrho),0).
$$
Again $\tau_S(\varrho) \in \Omega^2(\mathcal C^\infty(S,X))$ yields a global section of the transgressed character  $\tau_S(\widehat{CW}_\theta(\frac{1}{2}p_1))$. 

Summarizing, we obtain the following table of objects in our version of extended Chern-Simons theory, the components of the classical action, and its trivializations:

\begin{center}
\begin{tabular}{|c|c|c|c|c|}
\hline
mfd. & dim. & Chern-Simons obj. & action &  trivialization \\ \hline
* & 0 & $\widehat{CCS}_\theta(\lambda)$  & $3$-mfd. hol.~on $X$ & $(\widehat q,\varrho)$ \\ \hline
$S^1$ & 1 & $\tau(\widehat{CCS}_\theta(\lambda))$ & surface hol.~on $\LL X$ & $\tau(\widehat q,\varrho)$ \\ \hline
$\Sigma^2$ & 2 & $\tau_\Sigma(\widehat{CCS}_\theta(\lambda))$ & hol.~on $C^\infty(\Sigma,X)$ & $\tau_\Sigma(\widehat q,\varrho)$ \\ \hline
$M^3$ & 3 & $\tau_M(\widehat{CCS}_\theta(\lambda))$ & Chern-Simons action & $\tau_M(\widehat q,\varrho)$ \\ \hline
\end{tabular}
\end{center}

\begin{rem}[Higher order Chern-Simons theories]
In the same way as for differential String classes, we can define trivilizations of higher degree Chern-Simons theories defined by higher degree Cheeger-Chern-Simons characters:
Let $G$ be a Lie group with finitely many components and $\pi:(P,\theta) \to X$ a principal $G$-bundle with connection.
Let $(\lambda,u) \in K^{2k}(G;\Z)$ and $\widehat{CW}_\theta(\lambda,u)$, $\widehat{CCS}_\theta(\lambda,u)$ the associated Cheeger-Simons and Cheeger-Chern-Simons characters.
As explained in Section~\ref{sec:action}, the Cheeger-Simons character defines a classical actions on smooth maps $f:M^{2k-1} \to X$ from a closed oriented $(2k-1)$-manifold $M^{2k-1}$.
Likewise, the Cheeger-Chern-Simons character $\widehat{CCS}_\theta(\lambda,u)$ defines a classical actions on pairs $(f,\sigma)$, consisting of a smooth map $f:M \to X$ from an oriented $(2k-1)$-manifold $M$ with boundary and a section $\sigma: \partial M \to P|_{\partial M}$, see Section~\ref{subsec:CS_action_bound}.
Transgression to loop space and higher mapping spaces yields a higher degree Chern-Simons theory, extended down to points, in the sense explained above.
A differential $u$-trivialization $\widehat q$ with differential form yields a trivialization of the $(2k-1)$-manifold holonomy of $\widehat{CW}_\theta(\lambda,u)$.
Similarly, transgression of the pair $(\widehat q,\varrho)$ to loop space and higher mapping spaces yields a trivialization of the extended higher order Chern-Simons theory.
\end{rem}

\begin{rem}[Dijkgraaf-Witten correspondence]
For a compact Lie group $G$, it is well-known that topological charges of $3$-dimensional Chern-Simons theories correspond to cohomology classes in $H^4(BG;\Z)$.
Likewise, topological charges of the Wess-Zumino model correspond to classes in $H^3(G;\Z)$.
By work of Dijkgraaf and Witten \cite{DW90}, the correspondence between $3$-dimensional Chern Simons theories and Wess-Zumino models is realized by the cohomology transgression $T:H^4(BG;\Z) \to H^3(G;\Z)$, applied to the topological charges.
The differential cohomology transgression $\widehat T: \widehat H^*(BG;\Z) \to H^*(G;\Z)$ from Definition~\ref{def:T_hat} generalizes the Dijkgraaf-Witten correspondence to fully extended Chern-Simons theories of arbitrary degree. 
\end{rem}

\subsection{String geometry}
Let $\pi:P \to X$ be a principal $\Spin_n$-bundle.
Applying the loop space functor yields a principal $\LL \Spin_n$-bundle $\pi:\LL P \to \LL X$.
In analogy to the construction of the spinor bundle on $X$ one aims at a construction of vector bundles associated with the loop group bundle $\LL P \to \LL X$ by nice representations of the loop group $\LL \Spin_n$.
However, the positive energy representations of $\LL \Spin_n$ are all projective, \cite{PS86}.
Therefor, one needs to lift the structure group of the loop bundle $\pi:\LL P \to \LL X$ from $\LL \Spin_n$ to its universal central extension
\begin{equation}\label{eq:centr_ext}
1 \to \Ul \to \widehat{\LL \Spin_n} \to \LL \Spin_n \to 1.
\end{equation}
The obstruction to such lifts is precisely the transgression to loop space of the class $\frac{1}{2}p_1(P) \in H^4(X;\Z)$, \cite{K87}.
Such lifts are sometimes called \emph{Spin structures} on the loop bundle $\pi:\LL P \to \LL X$.
A manifold $X$ is called \emph{String} if it admits a Spin structure $\pi:P \to X$ and $\frac{1}{2}p_1(P)=0$.
More generally, one may call a principal $\Spin_n$-bundle \emph{String} if $\frac{1}{2}p_1(P)=0$.

If $X$ is a String manifold, then it is possible to lift the structure group of the loop bundle $\pi:\LL P \to \LL X$ from $\LL \Spin_n$ to its universal central extension.
In a next step one may want to construct associated vector bundles and interesting operators on sections of those.
However, there remain serious analytical difficulties when dealing with differential operators on the infinite dimensional loop space $\LL X$. 
A famous conjecture of Witten says that the $S^1$-equivariant index of a hypothetical Dirac operator on $\LL X$ should be given by the so-called Witten genus \cite{W88}.
So far, construction of Dirac operators on loop space is far beyond reach, let alone analytical features like the Fredholm property, which are required to have a well-defined index. 
A related conjecture due to H\"ohn and Stolz \cite{S96} (which can be formulated without using those hypothetical Dirac operators) expects the Witten genus on a String manifold to be an obstruction against positive Ricci curvature.

Instead of struggling with the analysis on the free loop space $\LL X$, one may also study the obstruction $\frac{1}{2}p_1(P)$ and its trivializations on the manifold $X$ itself.
This is the program of \emph{String geometry}:
The universal characteristic class $\frac{1}{2}p_1$ generates $H^4(B\Spin_n;\Z) \cong H^3(\Spin_n;\Z) \cong \pi_3(\Spin_n) \cong \Z$.
Therefor, $\frac{1}{2}p_1(P) \in H^4(X;\Z)$ is the obstruction to lift the structure group of a principal $\Spin_n$-bundle $\pi:P \to X$ to its $3$-connected cover $\String_n \to \Spin_n$.

As explained in Section~\ref{subsec:triv}, homotopy classes of lifts of classifying maps 
\begin{equation*}
\xymatrix{
&& \widetilde{BG}_u \ar[d] && \\
X \ar@{.>}[urr]^{\widetilde f} \ar_f[rr] && BG \ar[rr]^{u} && K(\Z;n).  
}
\end{equation*}
give rise to so-called $u$-trivialization classes $q \in H^{n-1}(P;\Z)$.
In the special case of $G = \Spin_n$, $n \geq 3$, and $u = \frac{1}{2}p_1 \in H^4(B\Spin_n;\Z)$, such lifts are called \emph{String structures} (in the homotopy theoretic sense).
By \cite{R11}, there is a 1-1 correspondence between String structures in the homotopy theoretic sense and String classes $q \in H^3(P;\Z)$.
Similarly, one may consider differential String classes on a principal $\Spin_n$-bundle with connection $(P,\theta)$ as isomorphism classes of String structures with additional geometric structure.
As explained in Section~\ref{subsubsec:geom_string}, String classes on $(P,\theta)$ are precisely the stable isomorphism classes of so-called \emph{geometric String structures}.

In the String geometry program one may regard (geometric) String structures on $(P,\theta)\to X$ as replacements of (geometric) Spin structures on the loop bundle $\pi:\LL P \to \LL X$. 
More explicitly, transgression to loop space can be applied not only to the obstruction class $\frac{1}{2}p_1$ but also to its trivializations:
(geometric) String structures on $\pi:P \to X$ are transgressed to (geometric) Spin structures on the loop bundle $\pi:\LL P \to \LL X$.
For more details, see \cite{W12,W14}.

Here we notice that transgression of relative and absolute differential characters fits into that picture:

\begin{prop}
Let $\pi:(P,\theta) \to X$ be a principal $\Spin_n$-bundle and $\widehat q$ a differential String class with differential form $\varrho$.
Then the transgressed character $\tau(\widehat q) \in \widehat H^2(\LL P;\Z)$ represents a line bundle with connection on $\LL P$, which over any fiber of the loop bundle $\pi:\LL P \to \LL X$ represents the universal central extension $\widehat{\LL \Spin_n} \to \LL \Spin_n$.
In other words, transgression to loop space maps differential String classes on $(P,\theta)$ to differential Spin classes on the loop bundle $\LL P \to \LL X$.
\end{prop}

\begin{proof}
As above we notice that the transgressed differential String class satisfies 
$$
-\ti_\pi(\tau) 
= 
\tau(\widehat{CCS}_\theta(\frac{1}{2}p_1)) - \iota_{\pi}(\tau(\varrho),0).
$$
In particular, the transgressed differential form $\tau(\varrho)$ yields a global section of the transgressed character $\tau(\widehat{CW}_\theta(\frac{1}{2}p_1))$.
Moreover, by naturality of the transgression, the transgressed character $\tau(\widehat q) \in \widehat H^2(\LL P;\Z)$ satisfies $i_{\LL P_\gamma}^*(\tau) = \tau(i_{P_x}^*\widehat q) \in \widehat H^2(\LL \Spin_n;\Z)$.
This follows from the commutative diagram of evaluation maps
\begin{equation*}
\xymatrix{
\LL \Spin_n \times S^1 \ar[r] \ar[d]^{\widehat\pi_!} \ar`u`[rr]^{\ev}[rr] & \LL P \times S^1 \ar^(0.65){\ev}[r] \ar[d]^{\widehat\pi_!} & P \\  
\LL \Spin_n \ar[r] & \LL P &   
}
\end{equation*}
and the identification of the fiber $\LL P_\gamma = \pi^{-1}(\gamma)$ over $\gamma \in \LL X$ with the loop group $\LL \Spin_n$.

Note that transgression to loop space induces an isomorphism on cohomology $\tau:H^3(\Spin_n;\Z) \to H^2(\LL \Spin_n;\Z)$.
By definition, the class $c(\widehat q)$ represents the generator of $H^3(\Spin_n;\Z)$ over any fiber $P_x \approx G$.
Thus $c(\tau(\widehat q))$ represents the generator of $H^2(\LL \Spin_n;\Z)$ over any fiber of the loop bundle.
Hence $\tau(\widehat q)$ represents fiberwise the universal central extension.

In this sense, String classes transgress to Spin classes on the loop bundle, and differential String classes transgress to differential refinements of those. 
\end{proof}

\appendix
\section{Differential characters}\label{app:characters}
In this section we briefly recall the notion of (absolute and relative) differential characters as introduced in \cite{CS83} and \cite{BT06}.
We recall some facts on the relation between absolute and relative characters from \cite{BB13}. 

Let $\varphi:A\to X$ be a smooth map.
Let $Z_*(\varphi;\Z)$ be the group of smooth singular cycles of the mapping cone complex.
Denote the differential of the mapping cone complex by $\partial_\varphi(v,w):=(\partial v + \varphi_*w,-\partial w)$.
Similarly, denote by $\Omega^*(\varphi)$ the mapping cone de Rham complex with differential $d_\varphi(\omega,\vartheta) = (d\omega,\varphi^*\omega - d\vartheta)$.

Let $k \geq 2$.
The group of degree-$k$ \emph{relative} (or mapping cone) \emph{differential charac\-ters} $\widehat H^k(\varphi;\Z)$ is defined as:
\begin{equation*}
\widehat H^k(\varphi;\Z)
:=
\big\{ h \in \Hom(Z_{k-1}(\varphi;\Z),\Ul) \;\big|\; h \circ \partial_\varphi \in \Omega^k(\varphi)\big\}.
\end{equation*}
The notation $h \circ \partial_\varphi \in \Omega^k(\varphi)$ means that there exists a pair of differential forms $(\omega,\vartheta) \in \Omega^k(\varphi)$ such that for any chain $(v,w) \in C_k(\varphi;\Z)$ we have 
\begin{align*}
h(\partial_\varphi(v,w)) 
&= \exp \Big( 2\pi i \int_{(v,w)} (\omega,\vartheta)\Big) \\
&= \exp \Big( 2\pi i \int_v \omega + \int_w \vartheta \Big).
\end{align*}

It turns out that the pair of forms $(\omega,\vartheta) \in \Omega^k(\varphi):=\Omega^k(X) \times \Omega^{k-1}(A)$ is uniquely determined by the character $h$.
Moreover, it is $d_\varphi$-closed and has integral periods.
We call $\omega =: \curv(h)$ the \emph{curvature} of the character $h$ and $\vartheta =: \cov(h)$ its \emph{covariant derivative}.
We also have a homomorphism $c:\widehat H^k(\varphi;\Z) \to H^k(\varphi;\Z)$, called \emph{characteristic class}.

The group $\widehat H^k(X;\Z)$ of absolute differential characters on $X$ is obtained as above by replacing the mapping cone complexes by the smooth singular and the de Rham complex of $X$.
A character $h \in \widehat H^k(X;\Z)$ then has a characteristic class in $c(h) \in \widehat H^k(X;\Z)$ in integral cohomology and a curvature $\curv(h) \in \Omega^k_0(X)$ in the space of closed $k$-forms with integral periods.

In \cite[Ch.~8]{BB13} we establish a 1-1 correspondence between $\widehat H^2(\varphi;\Z)$ and the group of isomorphism classes of hermitean line bundles with connection and section along $\varphi$.
By a  section we mean a nowhere vanishing section or a section of the associated $\Ul$-bundle.
The curvature of a character corresponds to the (normalized) curvature form of the line bundle.
Its covariant derivative corresponds to the covariant derivative of the section.
Hence the name.
The characteristic class corresponds to the first Chern class of the line bundle. 
Similarly, $\widehat H^2(X;\Z)$ corresponds to the group of isomorphism classes of hermitean line bundles with connection.

The group $\widehat H^k(\varphi;\Z)$ fits into the following commutative diagram of short exact sequences: 

\begin{equation}
\xymatrix@C=30pt{
& 0 \ar[d] & 0 \ar[d] & 0 \ar[d] & \\
0 \ar[r] & \frac{H^{k-1}(\varphi;\R)_{\phantom{\R}}}{H^{k-1}(\varphi;\Z)_\R} \ar[d] \ar[r] 
  & \frac{\Omega^{k-1}(\varphi)}{\Omega^{k-1}_0(\varphi)} \ar[d]^{\iota} \ar[r]^d & d\Omega^{k-1}(\varphi) \ar[d] \ar[r] & 0 \\
0 \ar[r] & H^{k-1}(\varphi;\Ul) \ar_p[d] \ar[r]^(0.55)j & \widehat H^k(\varphi;\Z) \ar[d]^c \ar[r]^{(\curv,\cov)} 
  & \Omega^k_0(\varphi) \ar[d] \ar[r] & 0 \\
0 \ar[r] & \Tor(H^k(\varphi;\Z)) \ar[d] \ar[r]_(0.55){j_{\Tor}} & H^k(\varphi;\Z) \ar[d] \ar[r] 
  & H^k(\varphi;\Z)_{\R} \ar[d] \ar[r] & 0 \\
& 0 & 0 & 0 &
}
\label{eq:short_ex_sequ}
\end{equation}
Here $\Omega^k_0(\varphi)$ denotes the space of pairs of forms which are $d_\varphi$-closed and have integral periods.
The homomorphism 
$$
\iota:\Omega^{k-1}(\varphi) \to \widehat H^k(\varphi;\Z),
\quad 
(\mu,\nu) \mapsto \Big(\,(s,t) \mapsto \exp\big(2\pi i \int_{(s,t)}(\mu,\nu)\,\big)\,\Big),
$$ 
is called \emph{topological trivialization}.
The Kronecker product defines the homomorphism 
$$
j:H^{k-1}(\varphi;\Z) \to \widehat H^k(\varphi;\Z), \quad w \mapsto \big(\,(s,t) \mapsto \langle w,[s,t] \rangle \,\big), 
$$ 
which we call \emph{inclusion of flat characters}.
Finally, $j_{\Tor}$ denotes the inclusion of the torsion subgroup of $H^k(\varphi;\Z)$.
The left column is built from the coefficient sequence $\Z \hookrightarrow \R \twoheadrightarrow \Ul$ and the canonical identification $\Tor(H^k(\varphi;\Z)) \cong \Ext(H_{k-1}(\varphi;\Z)\Z)$.
It is easy to see that the rows of the above diagram split algebraically.   
In \cite[Ch.~A]{BSS14} we show that they also have (non-canonical) topological splittings, if the real cohomology of $X$ is finite dimensional.

The mapping cone de Rham cohomology class of the curvature and covariant derivative coincides with the image of the characteristic class in real cohomology.
Thus the two middle sequences of \eqref{eq:short_ex_sequ} may be joined to the following exact sequence:
\begin{equation}\label{eq:sequ_R}
\xymatrix{
0 \ar[r] & \frac{H^{k-1}(\varphi;\R)_{\phantom{\R}}}{H^{k-1}(\varphi;\Z)_\R} \ar[r] & \widehat H^k(\varphi;\Z) \ar^{(\curv,\cov,c)}[rr] && R^k(\varphi;\Z) \ar[r] & 0.
}
\end{equation}
Here 
$$
R^k(\varphi;\Z):=\big\{\;((\omega,\vartheta),\tilde u) \in \Omega^k_0(\varphi) \times H^k(\varphi;\Z)\,\big|\, [(\omega,\varphi)]_{dR} = \tilde u_\R \in H^k(\varphi;\R) \;\big\}
$$ 
denotes the set of pairs of differential forms and integral cohomology classes that match in real cohomology.
We also obtain the corresponding short exact sequences for the group of absolute characters.
For details, see \cite{BT06} and \cite[Ch.~8]{BB13}.

Moreover, we have natural homomorphisms between absolute and relative characters.
These fit into the following long exact sequence \cite[Ch.~8]{BB13}: 
\begin{equation}\label{eq:long_ex_sequ}
\xymatrix@C=20pt{
H^{k-2}(X;\Ul) \ar^{\varphi^* \circ j}[r] &
\widehat H^{k-1}(A;\Z) \ar^{\ti_\varphi}[r] & \widehat H^k(\varphi;\Z) \ar^{\vds_{\varphi}}[r] & \widehat H^k(X;\Z) \ar^{\varphi^* \circ c}[r] & H^k(A;\Z) 
} 
\end{equation}
The sequence proceeds by the long exact sequences for smooth singular mapping cone cohomology with $\Ul$-coefficients on the left and with integer coefficients on the right.
In degree $2$, the homomorphism $\vds_\varphi$ corresponds to the forgetful map that ignores sections.

A relative character $h' \in \widehat H^k(\varphi;\Z)$ is called a \emph{section along $\varphi$} of the absolute character $h = \vds_\varphi(h')$.
The sequence \eqref{eq:long_ex_sequ} in particular tells us that a character $h \in \widehat H^k(X;\Z)$ admits sections along $\varphi$ if and only if it is topologically trivial along $\varphi$, i.e.~$\varphi^*c(h)=0$.
Moreover, it is shown in \cite[Ch.~8]{BB13} that global sections are uniquely determined by their covariant derivative, i.e.~we have an isomorphism
\begin{equation}\label{eq:glob_sect}
\cov: \widehat H^k(\id_X;\Z) \xrightarrow{\cong} \Omega^{2k-1}(X).  
\end{equation}
The inverse is given by $\cov^{-1}(\varrho) = \iota_{\id}(\varrho,0)$.

Differential cohomology is not homotopy invariant.
Instead, there is the following homotopy formula \cite[Ex.~56]{BB13}:
for a homotopy $f:[0,1] \times X \to Y$ between smooth maps $f_0,f_1:X \to Y$ and a differential character $h \in \widehat H^k(Y;\Z)$, we have:
\begin{equation}\label{eq:homotopy_abs}
f_1^*h - f_0^*h 
=
\iota \Big( \int_0^1 f_s^*\curv(h) ds \,\Big) \,. 
\end{equation}
Likewise, for a homotopy $(f,g):[0,1] \times (X,A) \to (Y,B)$ between smooth maps $(f_0,g_0),(f_1,g_1): (X,A) \to (Y,B)$, and a relative differential character $h \in \widehat H^k(\psi;\Z)$, where $\psi:B \to Y$, we have \cite[II, Cor.~40]{BB13}:
\begin{equation}\label{eq:homotopy_rel}
(f_1,g_1)^*h - (f_0,g_0)^*h 
=
\iota_\varphi \Big( \int_0^1 f_s^*\curv(h) ds,-\int_0^1 g_s^*\cov(h) ds \Big) \,.
\end{equation}

The graded abelian group $\widehat H^*(X;\Z)$ of absolute differential characters carries a ring structure compatible with the exact sequences in \eqref{eq:short_ex_sequ} and with the wedge product of differential forms and the cup product on singular cohomology.
We derive a nice characterization of the ring structure in \cite[II, Ch.~6]{BB13}. 
Moreover, the graded abelian group $\widehat H^*(\varphi;\Z)$ of relative carries the structure of a right module over the ring $\widehat H^*(X;\Z)$, compatible with the module structures on mapping cone differential forms and mapping cone cohomology, see \cite[II, Ch.~4]{BB13}.

In \cite[I, Ch.~7]{BB13} and \cite[II, Ch.~5]{BB13} we construct fiber integration and transgression maps for absolute and relative differential characters.
Thus on any fiber bundle $\pi:E \to X$ with compact oriented fibers $F$ we obtain fiber integration homomorphisms
\begin{align*}
\widehat\pi_!: \widehat H^k(E;\Z) &\to \widehat H^{k-\dim F}(X;\Z) \\ 
\widehat\pi_!: \widehat H^k(\Phi;\Z) &\to \widehat H^{k-\dim F}(\varphi;\Z).
\end{align*}
Here $\varphi:A \to X$ denotes a smooth map and $\Phi:\varphi^*E \to E$ the induced bundle map.
The fiber integration maps commute with the usual fiber integrations on differential forms and cohomology and with the homomorphisms in the short exact sequences \eqref{eq:short_ex_sequ} and the long exact sequence \eqref{eq:long_ex_sequ}.

Transgression to the free loop space $\LL X:= \{\mbox{$\gamma: S^1 \to X$ smooth} \}$ is defined via pull-back by the evaluation map $\ev:\LL X \times S^1 \to X$, $(\gamma,t) \mapsto \gamma(t)$, and fiber integration in the trivial bundle:
$$
\tau: \widehat H^k(X;\Z) \to \widehat H^{k-1}(\LL X;\Z), \quad h \mapsto \widehat\pi_!(\ev^*h).
$$ 
Likewise, transgression to the free loop space is defined for relative or mapping cone characters:
For a smooth map $\varphi:A \to X$ let $\overline{\varphi}: \LL A \to \LL X$, $\gamma \mapsto \varphi \circ \gamma$, be the induced map of loop spaces.
Then we have the transgression $\tau: \widehat H^k(\varphi;\Z) \to \widehat H^{k-1}(\overline{\varphi};\Z)$, $h \mapsto \widehat\pi_!(\ev^*h)$.
Here $\LL(X,A)$ denotes the set of pairs of smooth maps $(f,g):S^1 \to (X,A)$ such that $\varphi \circ g = f$ and $\ev: \LL(X,A) \times S^1 \to (X,A)$ and $\LL(X,A)$, $((f,g),t) \mapsto (f(t),g(t))$ denotes the evaluation map.
Moreover, the transgression maps for absolute and relative characters commute with the maps in the exact sequence \eqref{eq:long_ex_sequ}.
For details, see \cite[II, Ch.~5]{BB13}.

\section{Bundle $2$-gerbes}\label{app:bundle_gerbes}
In \cite[II, Sect.~3.2.2]{BB13} we show that a bundle gerbe with connection $\mathcal G$, represented by a submersion $\pi:Y \to X$, defines a relative differential character $h_{\mathcal G} \in \widehat H^3(\pi;\Z)$.
Moreover, we have $(\curv,\cov)(h_{\mathcal G}) = (H,B)(\mathcal G)$, i.e.~the curvature and covariant derivative of the character $h_{\mathcal G}$ coincide with the curvature and curving of the bundle gerbe $\mathcal G$.
The image of the relative character under the map $\vds:\widehat H^3(\pi;\Z) \to \widehat H^3(X;\Z)$ coincides with the stable isomorphism class of the bundle gerbe.

In this section, we describe the analogous statement for bundle $2$-gerbes with connection.\footnote{For more details on bundle $2$-gerbes with connection and their trivializations see \cite{S04,W13}.}
As a particular instance of this fact, we conclude that for any principal $\Spin_n$-bundle with connection $\pi:(P,\theta) \to X$ the Cheeger-Simons bundle $2$-gerbe with respect to the connection $\theta$ represents the Cheeger-Chern-Simons character $\widehat{CCS}_\theta(\frac{1}{2} p_1) \in \widehat H^4(\pi;\Z)$.
This in turn implies that differential $\String$ structures in the sense of \cite{W13} represent differential $\String$ classes in our sense.

Recall that a bundle $2$-gerbe with connection $\mathcal G$, represented by a submersion $\pi:Y \to X$, consists of a bundle gerbe with connection $\mathcal P \to Y^{[2]}$ and a $3$-form $B \in \Omega^3(Y)$, subject to several compatibility conditions for tensor products of pull-backs to the various higher fiber products.
The $3$-form $B$ is called the \emph{curving} of the connection on $\mathcal G$. 
Moreover, the  connection of the bundle $2$-gerbe has a \emph{curvature} $4$-form $H \in \Omega_0^4(X)$.
The curvature and curving are related by $\pi^*H = dB$.
The \emph{characteristic class} of a bundle $2$-gerbe is a cohomology class $CC(\mathcal G) \in H^4(X;\Z)$.

A \emph{trivialization with connection} of a bundle $2$-gerbe with connection is a bundle gerbe with connection $\mathcal S$ over $Y$, subject to several compatibility conditions for tensor products of pull-backs to the various higher fiber products.
In particular, any trivialization with connection comes together with a uniquely determined $3$-form $\varrho \in \Omega^3(X)$ such that $\pi^*\varrho = \curv(\mathcal S) + B$ and $d\varrho = H$. 
A bundle $2$-gerbe admits trivializations if and only if its characteristic class vanishes, and any trivialization admits compatible connections.

Now let $\mathcal G$ be a bundle $2$-gerbe with connection, represented by a submersion $\pi:Y \to X$.
We define the differential character $h_{\mathcal G} \in \widehat H^4(\pi;\Z)$ as follows:\footnote{We use the notations from \cite{BB13}.}
For a cycle $(s,t) \in Z_3(\pi;\Z)$ choose a geometric relative cycle $(\zeta,\tau) \in \ZZ_3(\pi)$ that represents the homology class of $(s,t)$.
Let $(\zeta,\tau)$ be represented by a smooth map $(S,T) \xrightarrow{(f,g)} (X,Y)$.
Choose a chain $(a,b) \in C_4(\pi;\Z)$ such that $[(s,t)-\partial_\pi(a,b)]_{\partial_\pi S_4} = [\zeta,\tau]_{\partial_\pi S_4}$.
Since $\dim(S)=3$, the pull-back bundle $2$-gerbe $f^*\mathcal G$ is trivial.
Choose a trivialization $\mathcal S$ with compatible connection and $3$-form $\omega \in \Omega^3(S)$.
The map $g:T \to Y$ factors through the induced map $f^*Y \xrightarrow{F}Y$ and thus induces a map $g:T \to f^*Y$. 
Since $\dim(T)=2$, the pull-back bundle gerbe $g^*\mathcal S$ is trivial.
Choose a trivialization with compatible connection and $2$-form $\vartheta \in \Omega^2(T)$. 

Now put:
\begin{align}
h_{\mathcal G}(s,t)
:=
\exp \Big( 2\pi i \Big( \int_{(S,T)} (\varrho,\vartheta) + \int_{(a,b)} (H,B) \Big) \Big).
\end{align}
In the same way as for bundle gerbes with connection \cite[II, Ch.~3]{BB13}, one can show that $h_{\mathcal G}$ is indeed a differential character in $\widehat H^4(\pi;\Z)$ with 
$$
(\curv,\cov)(h_{\mathcal G}) 
=
(H,B).
$$

\section{Transgression}\label{app:transgression}
In this section, we discuss the relation of transgression to loop space with transgression in universal bundles.
We show that the two transgressions agree in singular cohomology.
This is of course well-known.
For convenience of the reader, we give the argument here. 
Most of the material can be found e.g.~in \cite[Ch.~10.3]{DK01}. 

Let $X$ be a smooth manifold and $x_0 \in X$ an arbitrary base point. 
Transgression to the based loop space $\Omega X:= \{\gamma \in \LL X \,|\, \gamma(1) = x_0 \}$ is defined as for the free loop space by pull-back via the evaluation map $\ev:\Omega X \times S^1 \to X$, $(\gamma,t) \mapsto \gamma(t)$ and fiber integration in the trivial bundle:
$$
\tau^0: H^*(X;\Z) \to H^{*}(\Omega X;\Z), \quad u \mapsto \pi_!(\ev^*u).
$$
On the other hand, for any fibration $P \xrightarrow{\pi} X$ with fiber $F$ and contractible total space, we have the transgression defined by
\begin{equation}\label{eq:def_transgr}
\xymatrix@C=20pt{
T_P:H^*(X;\Z) \ar[r]^{\cong} & H^*(X,x_0;\Z) \ar[r]^{\pi^*} & H^*(P,P_{_0};\Z) \ar[r]^{\delta^{-1}} & H^{*}(P_{x_0};\Z). 
}
\end{equation}
The connecting homomorphism $\delta: H^{*}(P_{x_0};\Z) \to H^*(P,P_{x_0};\Z)$ is an isomorphism since $P$ is assumed to be contractible. 

For a topological space $Y$ let $CY:= Y \times [0,1] / Y \times \{0\}$ denote the cone over $Y$ and $SY:= CY/Y \times \{1\}$ the suspension.
Let $p:CY \to SY$ denote the projection. 
The suspension isomorphism $S^*:H^*(SY;\Z) \to H^{*}(Y;\Z)$ is the concatenation of isomorphisms
\begin{equation}\label{eq:susp_iso}
\xymatrix{
H^*(SY;\Z) 
\ar[r]^{p^*} &
H^*(CY,Y;\Z)
\ar[r]^{\delta^{-1}} &
H^{*}(Y;\Z).
}
\end{equation}
The connecting homomorphism $\delta: H^{*}(Y;\Z) \to H^*(CY,Y;\Z)$ is an isomorphism since the cone $CY$ is contractible.

Now let $Y=\Omega X$.
The suspension is a quotient of the trivial fiber bundle with fiber $S^1$.
In other words, $S\Omega X = \Omega X \times S^1/ \Omega X \times \{1\}$.
Let $\pr:\Omega X \times S^1 \to S \Omega X$ denote the projection.
Define a map $\ell:\Omega X \to X$ by $[s,\gamma] \mapsto \gamma(s)$.
Then we have the commutative diagram:
\begin{equation*}
\xymatrix{
\Omega X \times S^1 \ar^\ev[rr] \ar_\pr[d] && X \\
S\Omega X. \ar[urr]_\ell && 
} 
\end{equation*}
Fiber integration $\pi_!$ in smooth singular cohomology is defined by the Leray-Serre spectral sequence, see \cite[p.~482f.]{BH53}.
It is realized by pre-composition of cocycles with the transfer map on smooth singular chains, see \cite[Ch.~4]{BB13}.
This identification yields the commutative diagram:
\begin{equation*}
\xymatrix{
H^*(S\Omega X;\Z) \ar^{S^*}[rr] \ar[d]_{\pr^*} && H^{*}(\Omega X;\Z) \\
H^*(\Omega X \times S^1;\Z). \ar_{\pi_!}[urr]
} 
\end{equation*}
Thus we obtain 
\begin{equation}\label{eq:tauSell}
\tau^0
=
\pi_! \circ \ev^* 
=
\pi_! \circ \pr^* \circ \ell^*
=
S^* \circ \ell^*: H^*(X;\Z) \to H^{*}(\Omega X;\Z). 
\end{equation}
Now let $\PP_0 X := \{\mbox{$\gamma:[0,1] \to X$ smooth, $\gamma(0)=x_0$ } \}$ be the based path space of $X$.
Let $e: \PP_0 X \to X$, $\gamma \mapsto \gamma(1)$, be the path fibration with fiber $\Omega X$.
Since $\PP_0 X$ is contractible we obtain from \eqref{eq:def_transgr} the transgression $T_{\PP_0}:H^*(X;\Z) \to H^{*}(\Omega X;\Z)$ for the path fibration.

We have a map of pairs $g:(C\Omega X,\Omega X) \to (\PP_0 X,\Omega X)$ induced by the map $g:C\Omega X \to \PP_0 X$, $(t,\gamma) \mapsto (s \mapsto \gamma(st))$.
Since $C\Omega X$ and $\PP_0 X$ are contractible total spaces of fibrations with the same fiber, the induced map on cohomology is an isomorphism due to the five lemma.
This yields the commutative diagram of isomorphisms:
\begin{equation}\label{eq:pelleg}
\xymatrix{
H^*(\PP_0 X,\Omega X;\Z) \ar^{g^*}[rr] &&  H^*(C\Omega X,\Omega X;\Z) \\
& H^{*}(\Omega X;\Z) \ar[ur]_{\delta} \ar[ul]^{\delta} & 
}
\end{equation}
Moreover, we have the commutative diagram
\begin{equation}\label{eq:egpell}
\xymatrix{
C\Omega X \ar^{e \circ g}[rr] \ar[dr]_p && X \\
& S\Omega X \ar[ur]_\ell &
}
\end{equation}
This yields:
$$
\tau^0
\stackrel{\eqref{eq:tauSell}}{=}
S^* \circ \ell^*
\stackrel{\eqref{eq:susp_iso}}{=}
\delta^{-1} \circ p^* \circ \ell^*
\stackrel{\eqref{eq:pelleg}}{=}
\delta^{-1} \circ (e \circ g)^*
\stackrel{\eqref{eq:egpell}}{=}
\delta^{-1} \circ e^*
=
T_{\PP_0}.   
$$
As above, let $\pi_{EG}:EG \to BG$ be a universal principal $G$-bundle, i.e.~a principal $G$-bundle with contractible total space $EG$.
A contraction of $EG$ to $y_0 \in EG$ associates to any point $y \in EG$ a path $\gamma_y:[0,1] \to EG$ with $\gamma_y(0) = y_0$ and $\gamma_y(1) = y$.
Thus $\gamma_y \in \PP_0(EG)$.

Let $x_0:= \pi_{EG}(y_0)$.
Then we obtain a map $H:EG \to \PP_0 BG$, $y \mapsto \pi_{EG} \circ \gamma_y$.
Since $\ev(\pi_{EG} \circ \gamma_y) = \pi_{EG}(\gamma_y(1)) = \pi_{EG}(y)$, the map preserves the fibers.
It is in fact a homotopy equivalence of fibrations from the universal principal $G$-bundle $\pi_{EG}:EG \to BG$ to the path fibration $e:\PP_0 BG \to BG$ over the classifying space $BG$.
In particular, it yields a homotopy equivalence of the fibers and thus isomorphisms $H^*:H^*(\LL_0(BG);\Z) \to H^*(G;\Z)$.
Moreover, as a homotopy equivalence of fibrations with contractible total spaces the map $H$ identifies the transgressions 
$H^* \circ T_{\PP_0} = T$.

Summarizing, we obtain the following identifications of transgressions:
\begin{align*}
\tau^0 
&= S^* \circ \ell^* = T_{\PP_0}: H^*(X;\Z) \to H^{*}(\Omega X;\Z) \\
H^* \circ T_{\PP_0} 
&= T: H^*(BG;\Z) \to H^{*}(G;\Z).  
\end{align*}

%%%%%%%%%%%%%%%%%%%%%%%%%%%%%%%%%%%%%%%%%%%%%%%%%%%%%%%%%%%%%%%%%%%%%%%%%

\end{document}